\documentclass[a4paper, reqno]{amsart}
\usepackage[margin=1.5in]{geometry}
\usepackage[utf8]{inputenc}
\usepackage[english]{babel}

\usepackage{csquotes}
\usepackage{graphicx}
\usepackage[utf8]{inputenc}
\usepackage[T1]{fontenc}
\usepackage[dvipsnames]{xcolor}
\usepackage{subfiles}

\usepackage[mathscr]{eucal}
\usepackage[shortlabels]{enumitem}
\usepackage[super]{nth}
\usepackage{leftindex}
\usepackage[normalem]{ulem} 

\usepackage[backend=biber,
    style=ext-alphabetic,
    sorting=nyt,
    giveninits=true,
    isbn=false,
    maxalphanames=4,
    date=year,
    maxbibnames=99,
    backref=false,
    articlein=false]{biblatex}
\DeclareRobustCommand{\SkipTocEntry}[5]{}

\usepackage{amsmath}
\usepackage{amssymb}

\usepackage{stmaryrd}
\usepackage{amsthm}
\usepackage{graphicx}

\usepackage{dsfont} 
\usepackage{mathtools}
\usepackage{xparse}
\usepackage{tikz-cd}
\tikzset{
    symbol/.style={%
        draw=none,
        every to/.append style={%
            edge node={node [sloped, allow upside down, auto=false]{$#1$}}}
    }
}
\usepackage{quiver}
\usepackage{spectralsequences}

\usepackage{aliascnt}
\let\oldtheorem\newtheorem
\RenewDocumentCommand{\newtheorem}{s m o m O{}}{%
\IfBooleanTF{#1}%
{\oldtheorem*{#2}{#4}}%
{\IfNoValueTF{#3}{\oldtheorem{#2}{#4}[#5]}%
{\newaliascnt{#2}{#3}%
\oldtheorem{#2}[#2]{#4}%
\aliascntresetthe{#2}}}}

\usepackage[pdfborder={0 0 0}]{hyperref}
\hypersetup{
   colorlinks=true,
   allcolors=.,
   bookmarksdepth = 3
}
\usepackage[capitalize, nameinlink]{cleveref}
\crefformat{equation}{#2(#1)#3}
\Crefformat{equation}{#2(#1)#3}
\crefname{section}{Section}{Sections}
\crefname{subsection}{Section}{Sections}
\crefname{subsubsection}{Section}{Sections}
\crefname{subappendix}{Section}{Sections}
\crefname{subsubappendix}{Section}{Sections}
\crefname{construction}{Construction}{Constructions}
\crefname{observation}{Observation}{Observations}
\crefname{introcorollary}{Corollary}{Corollaries}

\usepackage{microtype}

\theoremstyle{plain}
\newtheorem{theorem}{Theorem}[section]
\newtheorem{introtheorem}{Theorem}
\newtheorem{introcorollary}[introtheorem]{Corollary}
\newtheorem{lemma}[theorem]{Lemma}
\newtheorem{proposition}[theorem]{Proposition}
\newtheorem{corollary}[theorem]{Corollary}

\newtheorem*{theorem*}{Theorem}
\newtheorem*{claim*}{Claim}

\Crefname{introtheorem}{Theorem}{Theorems}

\theoremstyle{definition}
\newtheorem{definition}[theorem]{Definition}
\newtheorem{remark}[theorem]{Remark}
\newtheorem{example}[theorem]{Example}

\newtheorem{construction}[theorem]{Construction}

\newtheorem{warning}[theorem]{Warning}
\newtheorem{observation}[theorem]{Observation}

\theoremstyle{remark}

\makeatletter
\usepackage{contour}

\contourlength{0.6pt} 

\makeatletter
\DeclareRobustCommand{\myuline}[2][0pt]{%
  \ifmmode
    \uline{\hphantom{#2}\kern-#1}%
    \kern#1%
    \mathllap{\mathpalette\my@cont@{#2}}%
  \else
    \uline{\phantom{#2}\kern-#1}%
    \kern#1%
    \llap{\contour{white}{#2}}%
  \fi
}
\newcommand{\my@cont@}[2]{\contour{white}{\mbox{$\m@th#1#2$}}}
\makeatother

\definecolor{cbblue}{RGB}{0,114,178}
\definecolor{cborange}{RGB}{230,159,0}

\DeclareMathOperator{\im}{Im}

\DeclareMathOperator{\End}{End}
\DeclareMathOperator{\coEnd}{coEnd}

\newcommand{\id}{\mathrm{id}}
\newcommand{\bartimes}{\mathbin{\bar\times}}
\DeclareMathOperator*{\colim}{colim}

\DeclareMathOperator{\ev}{ev}

\DeclareMathOperator{\free}{free}
\DeclareMathOperator{\fgt}{fgt}
\DeclareMathOperator{\forget}{forget}

\DeclareMathOperator{\triv}{triv}

\DeclareMathOperator{\LMod}{LMod}
\DeclareMathOperator{\RMod}{RMod}
\DeclareMathOperator{\BMod}{BiMod}

\DeclareMathOperator{\CMon}{CMon}

\DeclareMathOperator{\Alg}{Alg}
\DeclareMathOperator{\coAlg}{coAlg}
\DeclareMathOperator{\CAlg}{CAlg}
\DeclareMathOperator{\coCAlg}{coCAlg}
\DeclareMathOperator{\Map}{Map}
\DeclareMathOperator{\ulmap}{\myuline{\mathrm{Map}}}
\DeclareMathOperator{\ulnat}{\myuline{\mathrm{Nat}}}
\DeclareMathOperator{\Sym}{Sym}

\DeclareMathOperator{\cross}{cr}

\DeclareMathOperator{\Bahr}{Bar}
\DeclareMathOperator{\Cobar}{Cobar}

\DeclareMathOperator{\const}{const}

\DeclareMathOperator{\Sh}{Sh}
\DeclareMathOperator{\Sp}{Sp}
\DeclareMathOperator{\red}{red}

\newcommand{\hide}[1]{}

\newlist{numberenum}{enumerate}{1}
\setlist[numberenum]{\upshape(\arabic*)}

\newcommand{\Fun}{\mathrm{Fun}}
\newcommand{\LaxFun}{\mathrm{LaxFun}}
\newcommand{\OplaxFun}{\mathrm{OplaxFun}}

\newcommand{\FunL}{\mathrm{Fun}^{\mathrm{L}}}

\newcommand{\Funtwo}{\mathrm{Fun}_2}

\newcommand{\FunLM}{\Fun_{\lten}}
\newcommand{\FunRM}{\Fun_{\rten}}
\newcommand{\FunBM}{\Fun_{\bten}}
\newcommand{\iFunLM}{\underline\Fun_{\lten}}
\newcommand{\iFunRM}{\underline\Fun_{\rten}}
\newcommand{\FunLL}[1]{\FunLM}
\newcommand{\FunRR}[1]{\FunRM}
\newcommand{\FunBB}[2]{\FunBM}
\newcommand{\iFunL}[1]{\iFunLM}
\newcommand{\iFunR}[1]{\iFunRM}

\newcommand{\Fin}{\mathrm{Fin}}

\newcommand{\cA}{\sA}
\newcommand{\cB}{\sB}
\newcommand{\cC}{\sC}
\newcommand{\cD}{\sD}
\newcommand{\cE}{\sE}

\newcommand{\cM}{\sM}

\newcommand{\cO}{\sO}
\newcommand{\cP}{\sP}

\newcommand{\cV}{\sV}

\newcommand{\sA}{\mathscr{A}}
\newcommand{\sB}{\mathscr{B}}
\newcommand{\sC}{\mathscr{C}}
\newcommand{\sD}{\mathscr{D}}
\newcommand{\sE}{\mathscr{E}}
\newcommand{\sF}{\mathscr{F}}

\newcommand{\sM}{\mathscr{M}}

\newcommand{\sO}{\mathscr{O}}
\newcommand{\sP}{\mathscr{P}}
\newcommand{\sQ}{\mathscr{Q}}

\newcommand{\sT}{\mathscr{T}}

\newcommand{\sV}{\mathscr{V}}
\newcommand{\sW}{\mathscr{W}}
\newcommand{\sX}{\mathscr{X}}
\newcommand{\sY}{\mathscr{Y}}
\newcommand{\sZ}{\mathscr{Z}}

\newcommand{\LSp}{{L\mathrm{Sp}}}
\newcommand{\St}{\mathrm{St}}

\newcommand{\diff}{\mathscr{D}\mathrm{iff}}
\newcommand{\diffsp}{\mathscr{D}\mathrm{iff}_{\mathrm{St}}}
\newcommand{\Diff}{\diff}

\newcommand{\fdiffsp}{\mathscr{D}\mathrm{iff}^{\mathrm{FSym}}_{\mathrm{St}}}

\newcommand{\presl}{\mathscr{P}\mathrm{r}^{\mathrm{L}}}
\newcommand{\preslst}{\mathscr{P}\mathrm{r}^{\mathrm{L}}_{\mathrm{St}}}

\newcommand{\pressym}{\mathscr{P}\mathrm{r}^{\mathrm{Sym}}}

\newcommand{\pressymst}{\mathscr{P}\mathrm{r}^{\mathrm{Sym}}_{\mathrm{St}}}

\newcommand{\pressymstc}{\mathscr{P}\mathrm{r}^{\mathrm{Sym}}_{\mathrm{St}, \geq 1}}

\newcommand{\Pfin}{\mathscr{P}^{\mathrm{fin}}}
\newcommand{\Spc}{\mathscr{S}}

\newcommand{\Cat}{\mathscr{C}\mathrm{at}}

\newcommand{\PreCattwo}{\mathscr{P}\mathrm{re}\mathscr{C}\mathrm{at}_2}

\newcommand{\MMor}{\mathscr{M}\mathrm{or}}

\newcommand{\Lie}{\mathbb{L}}

\newcommand{\Env}{\mathrm{Env}}

\newcommand{\op}{\mathrm{op}}

\newcommand{\coop}{\mathrm{coop}}

\newcommand{\SymFun}{\mathrm{SymFun}}
\newcommand{\SymFunlin}{\mathrm{SymFun}^{\mathrm{lin}}}

\newcommand{\SSeq}{\mathrm{SSeq}}

\DeclareMathOperator{\Ind}{Ind}
\DeclareMathOperator{\Coind}{Coind}

\newcommand{\un}{\smallint}

\newcommand{\Tw}{\mathrm{Tw}}

\DeclareMathOperator{\Ar}{Ar}

\newcommand{\Exc}{\mathrm{Exc}}

\newcommand{\mlin}{P_{\vec{1}}}

\newcommand{\maAlg}{\mathcal{G}_{\mathrm{Alg}}}

\newcommand{\reAlg}{\Lambda_{\mathrm{Alg}}}

\newcommand{\Sph}{\mathbb{S}}
\newcommand{\unit}{\mathbf{1}}

\newcommand{\bbR}{\mathbb{R}}

\makeatletter
\newcommand{\oset}[3][0ex]{%
	\mathrel{\mathop{#3}\limits^{
			\vbox to#1{\kern-2\ex@
				\hbox{$\scriptstyle#2$}\vss}}}}
\makeatother
\newcommand{\eqarrow}{\oset[-.16ex]{\sim}{\longrightarrow}}

\tikzset{
	rot90/.style={anchor=south, rotate=90, inner sep=.5mm}
} 

\NewDocumentCommand\derprojlim{e{_}}{\mathchoice
{\varprojlim  \IfValueT{#1}{_{\mathclap{#1}}}{}^{\!1\!}\mathop{}}
{\varprojlim^1  \IfValueT{#1}{_{#1}}}
{\varprojlim^1  \IfValueT{#1}{_{#1}}}
{\varprojlim^1  \IfValueT{#1}{_{#1}}}}

\newcommand{\Op}{\mathrm{Op}}
\newcommand{\Mon}{\mathrm{Mon}}

\newcommand{\lax}{\mathrm{lax}}
\newcommand{\oplax}{\mathrm{oplax}}

\newcommand{\lten}{\mathrm{LM}}
\newcommand{\rten}{\mathrm{RM}}
\newcommand{\bten}{\mathrm{BM}}

\makeatletter
 \def\subsection{\@startsection{subsection}{1}%
 \z@{.7\linespacing\@plus\linespacing}{.5\linespacing}%
 {\normalfont\bfseries\centering}}
\makeatother

\renewcommand{\subset}{\subseteq}

\renewcommand{\epsilon}{\varepsilon}

\begin{document}

\title{The product rule in Goodwillie calculus}

\author{Max Blans}
\address{University of Oxford}
\email{max.blans@maths.ox.ac.uk}

\author{Thomas Blom}
\address{Max Planck Institute for Mathematics, Bonn}
\email{blom@mpim-bonn.mpg.de}

\begin{abstract}
In this paper, we prove a product rule for Goodwillie derivatives: given a differentiable $\infty$-category $\sC$ whose stabilization is equivalent to the $\infty$-category $\Sp$ of spectra, we show that the derivatives functor 
\[
\partial_* \colon \Fun^\omega(\sC, \Sp) \to \RMod_{\partial_*{\id_\sC}}(\SSeq(\Sp))
\]
is strong symmetric monoidal, where the source is equipped with the pointwise tensor product and the target with Day convolution.
Since the Koszul dual of $\partial_*{\id_\sC}$ can be recovered as a coendomorphism operad from Day convolution, this product rule is useful for calculating the operad $\partial_*{\id_\sC}$ in examples.

We derive the product rule as a consequence of the more general statement that taking derivatives preserves cartesian products on the $(\infty, 2)$-categorical level. In fact, the main theme of this paper is that the extraction of Goodwillie derivatives preserves a lot of structure when regarded as a functor of $(\infty, 2)$-categories: apart from products, it also preserves cotensors and certain pullbacks. 

We illustrate how our results can be used to calculate Goodwillie derivatives by determining the operad structure on the derivatives of the identity functor in pointed spaces, algebras over an operad and sheaves on a site.

\end{abstract}

\maketitle

\setcounter{tocdepth}{1} 
\tableofcontents

\section{Introduction}

In \cite{blansblom2025chainrulegoodwilliecalculus} we showed that for a large class of categories\footnote{We write category for $(\infty,1)$-category.} the Goodwillie derivatives of the identity functor carry the structure of an operad.
This generalizes the well-known result of Ching that the derivatives of the identity functor in pointed spaces are equivalent to the spectral Lie operad.
It is the purpose of the present paper to introduce techniques for calculating these operad structures.

Let us start by describing the setting to which our results apply.
Suppose that $\sC$ is a compactly generated, or more generally a differentiable category such that its stabilization is equivalent to the category $\Sp$ of spectra.
Then the derivatives $\partial_*{\id_\sC}$ carry the structure of an operad in spectra and for every 
filtered colimit preserving functor $F \colon \sC \to \Sp$ the derivatives $\partial_*F$ form a right $\partial_*{\id_\sC}$-module in the category $\SSeq(\Sp)$ of symmetric sequences in $\Sp$.
We will write $\Fun^\omega(\sC, \Sp)$ for the category of filtered colimit preserving functors from $\sC$ to $\Sp$.
The category $\SSeq(\Sp)$ has a symmetric monoidal structure given by Day convolution, which lifts to the category $\RMod_{\partial_*{\id_\sC}}(\SSeq(\Sp))$.
Our first main result is the following product rule for Goodwillie derivatives.

\begin{introtheorem}[Product rule] \label{introthm: product-rule}
    The Goodwillie derivatives functor
    \[
    \partial_* \colon \Fun^{\omega}(\sC, \Sp) \to \RMod_{\partial_*{\id_\sC}}(\SSeq(\Sp))
    \]
    admits a strong symmetric monoidal structure, where the source is equipped with the pointwise tensor product and the target is equipped with Day convolution.
\end{introtheorem}

\begin{remark}
    In the statement of this theorem and throughout the rest of this introduction, one can replace $\Sp$ by any of its stable presentable localizations.
\end{remark}

\begin{remark}
    Hahn and Yuan proved a product rule on the level of Goodwillie towers in \cite{HahnYuan2019}, which yields a symmetric monoidal functor
    \[
    \partial_* \colon \Fun^\omega(\sC, \Sp) \to \SSeq(\Sp)
    \]
    by passing to the associated graded.
    \cref{introthm: product-rule} is a strengthening of their result, providing a lift of this symmetric monoidal functor to the category $\RMod_{\partial_*{\id_\sC}}(\SSeq(\Sp))$.
    The existence of such a lift was conjectured by Malin in \cite[Conjecture 3.2.1]{malinUnstable1semiadditivityClassifying2025}.
\end{remark}

To understand the usefulness of the product rule for determining the operad structure on $\partial_*{\id_\sC}$, recall that if $\sO$ is a (reduced) operad in spectra, we can construct another operad $K\sO$ called the \emph{Koszul dual} of $\sO$.
In case all the terms $\sO(n)$ of the operad are dualizable spectra, we have that $KK\sO \simeq \sO$.
The starting point of our approach is that the operad $K\sO$ can be recovered from the symmetric monoidal category $\RMod_{\sO}(\SSeq(\Sp))$, as we will now explain.

Given any stable presentably symmetric monoidal category $\cV$ and an object $x \in \cV$, we can form the coendomorphism operad $\coEnd(x)$.
This is an operad in spectra whose $n$th term is given by 
\[
\coEnd(x)_n \simeq \ulmap(x, x^{\otimes n}),
\]
where $\ulmap$ denotes the mapping spectrum in $\cV$.
We will sometimes write $\coEnd_\cV(x)$ to emphasize that $x$ is an object of $\cV$.
If we let $\cV = \RMod_{\sO}(\SSeq(\Sp))$ with the Day convolution product and assume that $\sO$ is levelwise dualizable, then it turns out that there is an equivalence of operads
\[
\coEnd(\unit) \simeq K\sO,
\]
where $\unit$ denotes the trivial right $\sO$-module on the sphere spectrum considered as a symmetric sequence concentrated in arity $1$.
Since $\partial_*{\Sigma^\infty_\sC} \simeq \unit$, the strong monoidal structure on $\partial_*$ coming from the product rule induces a map of operads
\[
\coEnd_{\Fun(\sC, \Sp)}(\Sigma^\infty_\sC) \to \coEnd_{\RMod_{\partial_*{\id_\sC}}}(\unit),
\]
which turns out to be an equivalence. 
We therefore obtain the following result.

\begin{introcorollary} \label{introcor: product-rule-koszul-duality}
Let $\sC$ be a differentiable category with $\Sp(\sC) \simeq \Sp$, such that $\partial_n{\id_\sC}$ is dualizable for all $n \geq 1$.
Then there is an equivalence of operads
\[
\coEnd_{\Fun(\sC, \Sp)}(\Sigma^\infty_\sC) \simeq K\partial_*{\id_\sC},
\]
where the left-hand side denotes the coendomorphism operad of the stabilization functor $\Sigma^\infty_\sC$ in the symmetric monoidal category
$\Fun^{\omega}(\sC, \Sp)$ equipped with the pointwise tensor product.
\end{introcorollary}

As the following example illustrates, the coendomorphism operad of $\Sigma^\infty_\sC$ can sometimes be easily identified, so that we can calculate $\partial_*{\id_\sC}$.

\begin{example}
    Suppose that $\sC$ is the category of pointed spaces $\Spc_*$.
    Evaluation in $S^0$ induces a map of operads
    \[
    \coEnd_{\Fun(\Spc_*, \Sp)}(\Sigma^\infty) \to \coEnd_{\Sp}(\Sph),
    \]
    which is an isomorphism in positive degrees by the Yoneda lemma.
    Since $\coEnd_{\Sp}(\Sph)$ is a model for the unital commutative operad in spectra, it follows that $\coEnd_{\Fun(\Spc_*, \Sp)}(\Sigma^\infty)$ is equivalent to the \emph{nonunital} commutative operad $\mathbf{Com}$.
    We conclude that $\partial_*{\id_{\Spc_*}}$ is given by the Koszul dual of $\mathbf{Com}$, which is the definition of the spectral Lie operad. In this way, we recover Ching's result in our setting.
    See \cref{ssec:pointed-spaces} for full details.
\end{example}

The product rule can also be used to determine the right $\partial_*{\id_\sC}$-module structure on the derivatives of a functor.
Given a reduced finitary\footnote{A functor is reduced if it preserves the zero object; it is finitary if it preserves filtered colimits.} functor $F \colon \sC \to \Sp$, we can form a symmetric sequence 
\[
\ulnat(F, (\Sigma^\infty_\sC)^{\otimes \bullet}) \coloneqq \{ \ulnat(F, (\Sigma^\infty_{\sC})^{\otimes n}) \}_{n \geq 1}
\]
which is a right module over $\coEnd(\Sigma^\infty_\sC)$.
The product rule then implies:

\begin{introcorollary} \label{introcor: right-module-der}
    Suppose that both $\partial_n{\id_\sC}$ and $\partial_nF$ are dualizable for all $n \geq 1$.
    Under the equivalence $\coEnd(\Sigma^\infty_\sC) \simeq K\partial_*{\id_\sC}$, the right module $\ulnat(F, (\Sigma^\infty_\sC)^{\otimes \bullet})$ is equivalent to the Koszul dual of the right $\partial_*{\id_\sC}$-module $\partial_*F$.
\end{introcorollary}

We will derive the product rule as a consequence of the much more general statement that the extraction of Goodwillie derivatives preserves cartesian products as a functor of $2$-categories\footnote{We write $2$-category for $(\infty, 2)$-category.}.
To explain this, we will need to recall some facts about the chain rule.

Suppose that $\sC$ and $\sD$ are differentiable categories such that we have fixed equivalences
\[
\Sp(\sC) \simeq \Sp \simeq \Sp(\sD).
\]
Given a reduced finitary functor $F \colon \sC \to \sD$, we showed in \cite{blansblom2025chainrulegoodwilliecalculus} that $\partial_*F$ carries a $(\partial_*{\id_\sD}, \partial_*{\id_\sC})$-bimodule structure.
We also constructed a Morita $2$-category $\MMor_{\Op}$ whose objects are operads in spectra, such that for any pair of operads $\sO$ and $\sP$ there is an equivalence
\[
\MMor_{\Op}(\sO, \sP) \simeq \BMod_{(\sP, \sO)}(\SSeq(\Sp)),
\]
where the left-hand side denotes the mapping category in $\MMor_{\Op}$.
Composition in $\MMor_{\Op}$ is given by the two-sided bar construction of bimodules, also known as the relative composition product.
Write $\diff_{\Sp}$ for the $2$-category of differentiable categories with stabilization equivalent to $\Sp$ and reduced finitary functors between them. 
The main theorem proved in \cite{blansblom2025chainrulegoodwilliecalculus} is that there exists a strong functor of $2$-categories
\[
    \partial_* \colon \diff_{\Sp} \to \MMor_{\Op}
\]
that sends $\sC$ to $\partial_*{\id_\sC}$ and a functor $F \colon \sC \to \sD$ to the bimodule $\partial_*F$.
Since this is a strong functor of $2$-categories, it follows that if $\sC \xrightarrow{G} \sD \xrightarrow{F} \sE$ is a pair of composable morphisms in $\diff_{\Sp}$, then
\[
\partial_*{FG} \simeq \partial_*{F} \circ_{\partial_*{\id_\sD}} \partial_*{G},
\]
which is precisely the chain rule in Goodwillie calculus.

In fact, the restriction to categories whose stabilization is equivalent to $\Sp$ is not necessary and the functor above can be extended to the entire $2$-category $\diff$ of differentiable categories and reduced finitary functors.
In this generality, the notion of symmetric sequence needs to be generalized so that we also have to replace $\MMor_{\Op}$ by a larger Morita category $\MMor$.
See \cref{ssec: prelim-chain-rule} for a full explanation.

The main thrust of this paper is that the $2$-categorical formulation of the chain rule is not just a convenient way of phrasing the result, but also reveals a lot of structure that is preserved by the Goodwillie derivatives functor, providing new means of calculation.
The following theorem is an example of this perspective.

\begin{introtheorem} \label{introthm: products-cotensors-pullbacks}
    The Goodwillie derivatives functor
    \[
    \partial_* \colon \diff \to \MMor
    \]
    preserves finite cartesian products.
    It also preserves cotensors and certain pullbacks.
\end{introtheorem}

One of the most useful consequences of this result is that it allows one to differentiate symmetric monoidal categories and functors.
Suppose that $\sC$ is a nonunital symmetric monoidal differentiable category such that the tensor product functor is reduced and finitary.
Then $\sC$ defines a nonunital commutative monoid in $\diff$.
Since $\partial_*$ preserves products, it also preserves commutative monoids, so that we obtain a nonunital commutative monoid structure on $\partial_*{\id_\sC}$ in $\MMor$.
Applying any corepresentable functor $\MMor(\sO, -)$, we obtain a nonunital symmetric monoidal structure on $\BMod_{(\partial_*{\id_\sC}, \sO)}(\SSeq(\Sp))$.
It is this idea that leads to a proof of the product rule.

The fact that the derivatives functor preserves cotensors entails a computation of operad structure on the derivatives of the identity in functor categories.

\begin{remark}
From the $2$-categorical derivatives functor one can easily construct a functor
\[
\maAlg \colon \diff_{\Sp} \to \diff_{\Sp}
\]
of $2$-categories that sends $\sC$ to $\Alg_{\partial_*{\id_\sC}}(\Sp)$ and a functor $F \colon \sC \to \sD$ to
\[
\partial_*{F} \circ_{\partial_*{\id_\sC}} - \colon \Alg_{\partial_*{\id_\sC}}(\Sp) \to \Alg_{\partial_*{\id_\sD}}(\Sp).
\]
We call $\maAlg$ the \emph{Goodwillie transform}.
Just like the derivatives functor, it extends to the whole $2$-category $\diff$.
The Goodwillie transform inherits the preservation properties of $\partial_*$ proved in \cref{introthm: products-cotensors-pullbacks}, making it a very useful tool for calculation.
For instance, applying $\maAlg$ to a nonunital symmetric monoidal category $\sC$ in $\diff_{\Sp}$ yields a symmetric monoidal structure on $\Alg_{\partial_*{\id_\sC}}(\Sp)$.
It also preserves (lax) symmetric monoidal functors.

The properties of the Goodwillie transform proved in this paper are the main ingredients for the proofs in \cite{blansheuts2026characterizationspectrallieoperad}, where the first author and Heuts establish analogues of many classical facts about pointed spaces in the category of spectral Lie algebras and prove a characterization of the spectral Lie operad.
\end{remark}

To illustrate our results, we end this paper with some examples.
First of all, we compute the derivatives of the identity functor in the category of algebras over an operad in spectra.

\begin{introtheorem} \label{introthm:der-id-o-alg}
    Let $\sO$ be a strongly positive\footnote{This means that $\sO(0) = 0$ and $\sO(1) = \Sph$. This condition is sometimes referred to in the literature as \emph{reduced}, but this clashes with standard Goodwillie calculus terminology.} operad in spectra and write $\Alg_\sO(\Sp)$ for the category of $\sO$-algebras.
    Then there is an equivalence of operads $\partial_*\id_{\Alg_\sO(\Sp)} \simeq \sO$.
\end{introtheorem}
\begin{remark}
    Pereira already proved in his thesis \cite[Theorem 11.13]{PereiraThesis} that the underlying symmetric sequence of $\partial_*{\id_{\Alg_\sO(\Sp)}}$ is equivalent to $\sO$. 
\end{remark}

We also calculate the derivatives of the identity functor in categories of sheaves.
This result is most efficiently stated in terms of the Goodwillie transform.

\begin{introtheorem} \label{introthm:sheaves}
    Let $\sT$ be a site and $\sC$ a differentiable category.
    Write $\Sh(\sT; \sC)$ for the category of $\sC$-valued sheaves on $\sT$.
    Then there is an equivalence of categories
    \[
    \maAlg(\Sh(\sT; \sC)) \simeq \Sh(\sT; \Alg_{\partial_*{\id_\sC}}(\Sp{\sC})).
    \]
\end{introtheorem}

There is an important technical result that plays a role in the proof of both the product rule and \cref{introthm:der-id-o-alg}, and that is interesting in itself.
Given a symmetric sequence $A \in \SSeq_{\geq 1}(\Sp)$ concentrated in positive degrees, we can form a functor $\Lambda_A \colon \Sp \to \Sp$ given by the formula
\[
    \Lambda_A(X) = \bigoplus_{n \geq 1} (A_n \otimes X^{\otimes n})_{h\Sigma_n}.
\]
Since this is the direct sum of homogeneous functors corresponding to $A$, we find that $\partial_* \Lambda_A \simeq A$.
The construction $A \mapsto \Lambda_A$ extends to a strong monoidal functor
\[
\Lambda \colon \SSeq_{\geq 1}(\Sp) \to \End^{\ast, \omega}(\Sp),
\]
with respect to the composition product on the source and functor composition on the target.
We will show that the derivatives functor provides a strong monoidal retraction of this functor.
This implies for instance that if $\sO$ is an operad in spectra, then we have an equivalence of operads $\partial_*\Lambda_\sO \simeq \sO$.

We will prove this result as a consequence of a more general $2$-categorical statement. 
Writing $\diff_{\St} \subseteq \diff$ for the $2$-category of stable differentiable categories, the derivatives functor restricted to $\diff_{\St}$ factors through a subcategory $\pressymstc \subseteq \MMor$ consisting of stable categories and symmetric sequences.
The construction above extends to a functor $\Lambda \colon \pressymstc \to \diff_{\St}$, and we will show:

\begin{introtheorem} \label{introthm: derivatives-of-lambda}
    The composite
    \[
    \pressymstc \xrightarrow{\Lambda} \diff_{\St} \xrightarrow{\partial_*} \pressymstc
    \]
    is equivalent to the identity functor.
\end{introtheorem}

What follows is an overview of the contents of this paper.
In Section 2, we review the material on operads, Koszul duality, and the chain rule that we will need.
Section 3 is concerned with various 2-categorical structures preserved by the derivatives functor and contains a proof of \cref{introthm: products-cotensors-pullbacks}.
In Section 4 we prove the product rule (\cref{introthm: product-rule}) and its consequences (\cref{introcor: product-rule-koszul-duality,introcor: right-module-der}).
In Section 5 we give examples, proving \cref{introthm:der-id-o-alg,introthm:sheaves}.
Appendix A is devoted to the proof of \cref{introthm: derivatives-of-lambda}.
In Appendix B, we record a proof of the fact that the derivatives functor $\partial_* \colon \Fun(\Sp, \Sp) \to \SSeq(\Sp)$ admits a right adjoint, and that the unit of this adjunction on a homogeneous functor is given by the norm map.

\subsection*{Acknowledgments}

The authors would like to thank Gijs Heuts, Connor Malin and Niall Taggart for numerous conversations related to the contents of this work.
The second author is especially grateful to Connor Malin for explaining to him that a product rule of the form proved in \cref{introthm: product-rule} would allow one to compute operad structures on derivatives in terms of coendomorphism operads (\cref{introcor: product-rule-koszul-duality}).

MB was supported by the Royal Society through grant URF\textbackslash R1\textbackslash 211075 and wishes to thank the University of Oxford for its hospitality.
TB is grateful to the Max Planck Institute for Mathematics in Bonn for its hospitality and support during the writing of this paper.

\subsection*{Notation}
\begin{itemize}
    \item We write category for $(\infty, 1)$-category and $2$-category for $(\infty, 2)$-category.
    \item Given a stable category $\cC$, we write $\ulmap(x,y)$ for the mapping spectrum between $x,y \in \cC$.
    \item Given another category $\cD$, we write $\ulnat(F,G)$ for the spectrum of natural transformations between functors $F,G \colon \cD \to \cC$, i.e.\ the mapping spectrum in $\Fun(\cD,\cC).$
    \item Throughout this paper, $\LSp$ will always denote an accessible reflective localization
    \[\begin{tikzcd}[ampersand replacement=\&]
        \Sp \&\& \LSp
        \arrow[""{name=0, anchor=center, inner sep=0}, "L", shift left=2, from=1-1, to=1-3]
        \arrow[""{name=1, anchor=center, inner sep=0}, shift left=1, hook', from=1-3, to=1-1]
        \arrow["\dashv"{anchor=center, rotate=-90}, draw=none, from=0, to=1]
    \end{tikzcd}\]
    of $\Sp$ such that $\LSp$ is stable.
    \item If $\cC \curvearrowright \cM \curvearrowleft \cD$ is a bitensored category, then we will sometimes write $\BMod_{(\cC,\cD)}$ for $\BMod_{(\cC,\cD)}(\cM)$, leaving $\cM$ implicit.
\end{itemize}

\section{Preliminaries}

This section contains a brief recollection of various preliminaries needed throughout the paper.
We assume the reader is familiar with the theory of Goodwillie calculus as developed in \cite{Goodwillie2003} and \cite[\S 6.1]{HA}.
We will also freely use the language of $2$-category theory; one can find an exposition of the basic definitions and results that we need in \cite[Appendix A]{blansblom2025chainrulegoodwilliecalculus}.

\subsection{The chain rule} \label{ssec: prelim-chain-rule}

The purpose of this subsection is to recall the main results on the chain rule from \cite{blansblom2025chainrulegoodwilliecalculus}.
We will particularly emphasize the $2$-categorical perspective on this result, as this will play a central role throughout our paper.

We start by defining the categories and functors to which the theory applies.

\begin{definition}
    A category $\sC$ is called \emph{differentiable} if it is presentable, pointed, and filtered colimits commute with finite limits in $\sC$.
    A functor $F \colon \sC \to \sD$ is called \emph{reduced} if it preserves the zero object and \emph{finitary} if it preserves filtered colimits.
\end{definition}

We write $\Fun^{\ast, \omega}(\sC, \sD)$ for the full subcategory of $\Fun(\sC, \sD)$ spanned by the reduced finitary functors.

\begin{definition}
    Let $\diff$ denote the $2$-category of differentiable categories and reduced finitary functors.
    Write $\diffsp$ for the full subcategory of $\diff$ spanned by the stable presentable categories.
\end{definition}

Next we define the target of the derivatives functor.

\begin{definition}\label{def:functor-sseq}
    Let $\sA$ and $\sB$ be stable presentable categories.
    The category of \emph{functor symmetric sequences} from $\sA$ to $\sB$ is defined as
    \[
    \SSeq(\sA, \sB) \coloneqq \FunL(\Sym\sA, \sB),
    \]
    where $\Sym\sA = \bigoplus_{n \geq 0} \sA^{\otimes n}_{h\Sigma_n}$ is the symmetric algebra functor in $\presl$.
\end{definition}

We will usually write symmetric sequence instead of functor symmetric sequence.

\begin{remark}
    Somewhat more concretely, a symmetric sequence $F \in \SSeq(\sA, \sB)$ consists of a sequence of functors
    \[
    F_n \colon \sA^{\times n}_{h\Sigma_n} \to \sB \qquad \text{for $n \geq 0$}
    \]
    such that the underlying functor of $F_n$ preserves colimits in each variable separately.
    We call $F_n$ the arity $n$ component of $F$.
    We define $F$ to be \emph{positive} if $F_0 \simeq \ast$, and we write $\SSeq_{\geq 1}(\sA, \sB) \subseteq \SSeq(\sA, \sB)$ for the full subcategory spanned by the positive symmetric sequences.
    In case $\sA = \sB$, we call $F$ \emph{strongly positive} if it is positive and additionally $F_1 \simeq \id_\sA$.
    We write $\SSeq_+(\sA,\sA) \subseteq \SSeq(\sA, \sA)$ for the full subcategory spanned by the strongly positive symmetric sequences.
\end{remark}

\begin{remark}\label{remark:SSeq-vs-FunSSeq}
    Classically, a symmetric sequence in a category $\cA$ is defined as a functor 
    \[
    \Fin^{\simeq} \to \cA,
    \]
    where $\Fin^{\simeq}$ denotes the category of finite sets and bijections.
    By the equivalence $\amalg_{n \geq 0} B\Sigma_n \simeq \Fin^{\simeq}$, this is the same as a sequence $\{X_n\}_{n \geq 0}$, where $X_n \in \cA^{B\Sigma_n}$.
    We will usually write $\SSeq(\sA)$ for $\Fun(\Fin^\simeq, \sA)$.
    
    To see how this notion of symmetric sequence is related to the one we just defined, let us mention that if $\cA$ is an idempotent algebra in $\preslst$ and $\cB$ an $\cA$-module, then $\SSeq(\cA,\cB)$ is equivalent to the classical category of symmetric sequences $\SSeq(\cB)$.
    Namely, if $\cA$ is idempotent, then $\Sym \cA \simeq \bigoplus_{n \geq 0} \sA_{h\Sigma_n}$, and hence
    \[\SSeq(\cA,\cB) = \FunL(\Sym \cA, \cB) \simeq \bigoplus_{n \geq 0} \FunL(\cA, \cB)^{h\Sigma_n} \simeq \Fun(\amalg_n B\Sigma_n, \cB).\]
    This is for instance the case for $\sA = \Sp$.
    However, in general the category of (functor) symmetric sequences $\SSeq(\cA,\cB)$ of \cref{def:functor-sseq} is not of the form $\Fun(\Fin^\simeq,\cC)$ for any stable presentable category $\cC$.
\end{remark}

\begin{remark}
    Given symmetric sequences $F \in \SSeq(\sB, \sC)$ and $G \in \SSeq(\sA, \sB)$, we can form a new symmetric sequence
    \[
    F \circ G \in \SSeq(\sA, \sC),
    \]
    called the \emph{composition product} of $F$ and $G$.
    If $F$ and $G$ are positive, $F \circ G$ is positive as well and given by the formula
    \[
    (F \circ G)_I \simeq \bigoplus_{E \in \mathrm{Part}(I)} F_E \circ \{G_J\}_{J \in E}.
    \]
    Here we have indexed our symmetric sequences on non-empty finite sets.
    The direct sum is indexed by the set of partitions of $I$.
    See \cite[Proposition 3.2.9]{blansblom2025chainrulegoodwilliecalculus} for a formula that also works for non-positive symmetric sequences.
\end{remark}

Symmetric sequences can be assembled in a $2$-category.

\begin{proposition}[{\cite[\S 3.2]{blansblom2025chainrulegoodwilliecalculus}}]\label{prop:definition-prsymst}
    There is a $2$-category $\pressymst$ such that:
    \begin{enumerate}[\upshape{(}1\upshape{)}]
        \item The objects of $\pressymst$ are presentable stable categories;
        \item The category of $1$-morphisms from $\sA$ to $\sB$ in $\pressymst$ is equivalent to $\SSeq(\sA, \sB)$;
        \item Composition in $\pressymst$ is given by the composition product.
    \end{enumerate}
\end{proposition}

\begin{remark}
    The $2$-category $\pressymst$ is defined as the full subcategory of the $2$-category $\coCAlg(\preslst)$ of commutative coalgebras in $\presl$ spanned by the cofree objects.
    The proof that this definition satisfies the properties stated in the previous proposition relies on the fact that for a presentable category $\sC$, the category $\Sym(\sC)$ is the cofree commutative coalgebra on $\sC$ in $\presl$.
\end{remark}

\begin{remark}
    This proposition implies that for every stable presentable category $\sA$, the composition product defines a monoidal structure on the category $\SSeq(\sA, \sA)$.
    In case $\sA = \Sp$, we have an equivalence $\SSeq(\Sp, \Sp) \simeq \Fun(\Fin^\simeq, \Sp)$, and this monoidal structure agrees with the ordinary composition product of symmetric sequences, as defined for instance in \cite[\S 4.1.2]{brantnerThesis}; see \cite[Proposition 3.2.16]{blansblom2025chainrulegoodwilliecalculus} for a comparison.
    It follows that an algebra for the composition product in $\SSeq(\Sp, \Sp)$ is the same as an operad in spectra.
\end{remark}

\begin{remark}
    Since positive symmetric sequences are closed under the composition product, $\pressymst$ has a locally full subcategory $\pressymstc$ spanned by the positive symmetric sequences.
\end{remark}

If $F \colon \sC \to \sD$ is a reduced finitary functor between differentiable categories, then its Goodwillie derivatives define a symmetric sequence
\[
\partial_*F \in \SSeq_{\geq 1}(\Sp\sC, \Sp\sD),
\]
see \cite[\S 3.1.3]{blansblom2025chainrulegoodwilliecalculus}.
One of the main results of \cite{blansblom2025chainrulegoodwilliecalculus} is that this assignment can be extended to a functor of $2$-categories.

\begin{theorem}[{\cite[Theorem 3.3.2]{blansblom2025chainrulegoodwilliecalculus}}]
    There is a lax functor
    \[
    \partial_* \colon \diff \to \pressymstc
    \]
    that sends $\sC$ to $\Sp\sC$ and that is given by the derivatives functor
    \[
    \partial_* \colon \Fun^{\ast, \omega}(\sC, \sD) \to \SSeq_{\geq 1}(\Sp\sC, \Sp\sD)
    \]
    on mapping categories.
    The restriction of $\partial_*$ to $\diffsp$ is a strong functor of $2$-categories.
\end{theorem}

\begin{remark}\label{remark:bimodule-structure-on-partialF}
    This theorem implies that the derivatives functor
    \[
    \partial_* \colon \Fun^{\ast, \omega}(\sC, \sC) \to \SSeq_{\geq 1}(\Sp\sC, \Sp\sC)
    \]
    admits a lax monoidal structure, which is strong monoidal if $\sC$ is stable.
    Moreover, for every pair of differentiable categories $\sC$ and $\sD$, the functor
    \[
    \partial_* \colon \Fun^{\ast, \omega}(\sC, \sD) \to \SSeq_{\geq 1}(\Sp\sC, \Sp\sD)
    \]
    admits the structure of a lax functor of bitensored categories, where the bitensoring is given by pre- and postcomposition with endomorphisms.
    The derivatives of the identity $\partial_*{\id_\sC}$ therefore obtain the structure of an algebra in $\SSeq_{\geq 1}(\Sp\sC, \Sp\sC)$, and for every reduced finitary functor $F \colon \sC \to \sD$, its derivatives $\partial_*{F}$ obtain the structure of a $(\partial_*{\id_\sD}, \partial_*{\id_\sC})$-bimodule in $\SSeq(\Sp\sC, \Sp\sD)$.
\end{remark}

The chain rule in Goodwillie calculus states that the assignment that sends $F$ to the bimodule $\partial_*{F}$ is functorial: it takes functor composition to the relative composition product of bimodules.
This functoriality can best be expressed in terms of the following construction.

\begin{proposition}[{\cite[Theorem 5.1]{Blom2024StraighteningEveryFunctor}}]
    There exists a $2$-category\footnote{By construction, the Morita category is only a $2$-precategory, meaning that the Segal object in categories defining it is not complete/univalent. In this paper, $\MMor$ always denotes the completion of this $2$-precategory (cf.\ \cite[Remark A.1.16]{blansblom2025chainrulegoodwilliecalculus}).} $\MMor$, called the Morita category of $\pressymst$, such that:
    \begin{enumerate}[\upshape{(}1\upshape{)}]
        \item The objects of $\MMor$ are pairs $(\sA, \sO)$, where $\sA$ is a stable presentable category and $\sO \in \Alg(\SSeq(\sA, \sA))$;
        \item There is an equivalence of categories
        \[
        \MMor((\sA, \sO), (\sB, \sP)) \simeq \BMod_{(\sP, \sO)}(\SSeq(\sA, \sB)),
        \]
        where the left-hand side denotes the mapping category from $(\sA, \sO)$ to $(\sB, \sP)$ in $\MMor$;
        \item If $M \in \MMor((\sB, \sP), (\sC, \sQ))$ and $N \in \MMor((\sA, \sO), (\sB, \sP))$ are morphisms in the Morita category, then their composition is the $(\sQ, \sO)$-bimodule given by the relative composition product $M \circ_\sP N$ of bimodules.
    \end{enumerate}
\end{proposition}

\begin{remark}
    The relative composition product $M \circ_\sP N$ mentioned in the previous proposition is the bar construction computed as the colimit of the usual simplicial diagram
    \[
    \begin{tikzcd}
          \cdots M \circ \sP \circ \sP \circ N \ar[r, shift left=2] \ar[r] \ar[r, shift right=2] & M \circ \sP \circ N \ar[l, shift left, shorten=0.4em] \ar[l, shift right, shorten=0.4em] \ar[r, shift left] \ar[r, shift right] & M \circ N, \ar[l, shorten=0.4em]
   \end{tikzcd}
    \]
    where the face maps are given by the multiplication of $\sP$ and its right and left action on $M$ and $N$ respectively.
\end{remark}

The chain rule can then be stated as follows:

\begin{theorem}[{\cite[Theorem 4.4.7]{blansblom2025chainrulegoodwilliecalculus}}]\label{thm:chain-rule-cherry}
    There is a strong functor of $2$-categories
    \[
    \partial_* \colon \diff \to \MMor
    \]
    that sends $\sC$ to the pair $(\Sp\sC, \partial_*{\id_\sC})$ and a functor $F \colon \sC \to \sD$ to the $(\partial_*{\id_{\sD}}, \partial_*{\id_{\sC}})$-bimodule $\partial_*{F}$.
\end{theorem}

\begin{remark}
    The fact that this is a strong functor of $2$-categories has the following implication: if $F \colon \sD \to \sE$ and $G \colon \sC \to \sD$ is a pair of reduced finitary functors between differentiable categories, then there is an equivalence
    \[
    \partial_*{FG} \simeq \partial_*F \circ_{\partial_*{\id_\sD}} \partial_*{G}.
    \]
    This is the usual way of writing the chain rule.
\end{remark}

\begin{remark}
    If $\sA$ is a stable presentable category, we write
    \[
    \unit_{\sA} \in \SSeq(\sA, \sA)
    \]
    for the identity morphism of $\sA$ in $\pressymst$.
    It is given by the identity functor $\id_\sA \colon \sA \to \sA$ considered as a symmetric sequence concentrated in arity $1$.
    If $\sB$ is another stable presentable category, then there is clearly an equivalence of categories
    \[
    \BMod_{(\unit_\sB, \unit_\sA)}(\SSeq(\sA, \sB)) \simeq \SSeq(\sA, \sB).
    \]
    In fact, the full subcategory of $\MMor$ spanned by the pairs $(\sA, \unit_\sA)$ is equivalent to $\pressymst$.
    Since the restriction of the lax derivatives functor $\partial_* \colon \diff \to \pressymst$ to $\diffsp$ is a strong functor, we have that
    \[
    \partial_*{\id_\sA} \simeq \unit_\sA
    \]
    for every stable presentable category $\sA$.
    It follows that the functor $\partial_* \colon \diff \to \MMor$ restricted to $\diffsp$ factors as
    \[
    \diffsp \xrightarrow{\partial_*} \pressymst \subseteq \MMor.
    \]
\end{remark}

We will need the following result about the exactness of composition in $\MMor$.

\begin{proposition} \label{prop: exactness-composition-mor}
    Let $(\sA, \sO)$, $(\sB, \sP)$, and $(\sC, \sQ)$ be objects of $\MMor$.
    The relative composition product
    \[
    \BMod_{(\sQ, \sP)}(\SSeq(\sB, \sC)) \times \BMod_{(\sP, \sO)}(\SSeq(\sA, \sB)) \to \BMod_{(\sQ, \sO)}(\SSeq(\sA, \sC))
    \]
    preserves small colimits and finite limits in the first variable.
    It preserves sifted colimits in the second variable.
\end{proposition}
\begin{proof}
    It follows from \cite[Proposition 3.2.9]{blansblom2025chainrulegoodwilliecalculus} that the composition product preserves all small colimits in the first variable and sifted colimits in the second variable.
    Since the forgetful functor from bimodules to symmetric sequences preserves and reflects sifted colimits, and since colimits commute with colimits, we therefore find that the relative composition product commutes with sifted colimits in both variables.
    
    To prove that the relative composition product preserves all small colimits in the first variable, it now suffices to show that it preserves finite coproducts.
    The initial object in $\BMod_{(\sQ, \sP)}$ is given by $\sQ \circ 0 \circ \sP$, and since it is clear that this is preserved by the relative composition product, we only have to verify that it preserves binary coproducts.
    By writing our bimodules as a sifted colimit of free bimodules, we can further reduce to the case of a binary coproduct of free bimodules.
    So suppose that $X, Y \in \SSeq(\sB, \sC)$ and $M \in \BMod_{(\sP, \sO)}(\SSeq(\sA, \sB))$.
    Writing $\sqcup$ for the coproduct of bimodules and $\oplus$ for the coproduct of symmetric sequences, we then have
    \begin{align*}
        ((\sQ \circ X \circ \sP) \sqcup (\sQ \circ Y \circ \sP)) \circ_\sP M & \simeq \sQ \circ (X \oplus Y) \circ \sP \circ_\sP M \\
        & \simeq \sQ \circ (X \oplus Y) \circ M \\
        &\simeq  \sQ \circ (X \circ M \oplus Y \circ M) \\
        & \simeq (\sQ \circ X \circ M) \sqcup (\sQ \circ Y \circ M) \\
        &\simeq (\sQ \circ X \circ \sP \circ_\sP M) \sqcup (\sQ \circ Y \circ \sP \circ_\sP M).
    \end{align*}
    The first equivalence follows since the free $(\sQ, \sP)$-bimodule functor preserves colimits.
    The third equivalence follows from the fact that the composition product preserves coproducts in the first variable.
    This also implies that coproducts in right $\sO$-modules are computed in the category of symmetric sequences, so that $X \circ M \oplus Y \circ M$ is also the coproduct in right $\sO$-modules.
    The penultimate equivalence holds since
    \[
    \sQ \circ - \colon \RMod_\sO \to \BMod_{(\sQ, \sO)}
    \]
    is left adjoint to the forgetful functor, and therefore preserves colimits.
    This completes the proof that the relative composition product preserves binary coproducts, and therefore all small colimits in the first variable.
    
    By stability of the categories $\SSeq(\sB, \sC)$ and $\SSeq(\sA,\sC)$, the composition product also preserves limits in the first variable.
    Since both limits and sifted colimits in categories of bimodules are computed in the underlying category, and since sifted colimits commute with finite limits in any stable category, it follows that the relative composition product preserves finite limits in the first variable as well. 
\end{proof}

\subsection{Unreduced functors}\label{ssec:unreduced-functors}

We now briefly discuss derivatives of unreduced functors.
This will be used to upgrade the product rule proved in \cref{sec:prod-rule} to a \emph{unital} strong symmetric monoidal structure; observe that the unit $\const_{\unit}$ of the pointwise tensor product in $\Fun(\sC,\sD)$ is rarely reduced, hence a product rule for reduced functors would necessarily be \emph{nonunital}.

Let us first recall that the $n$th derivative $\partial_n F$ of a reduced finitary functor $F \colon \cC \to \cD$ is defined as the unique colimit preserving functor $(\Sp{\cC})^{\otimes n}_{h\Sigma_n} \to \Sp{\cD}$ such that
\[D_nF(x) \simeq \Omega^\infty_\cD \partial_n F(\Sigma^\infty_\cC x, \cdots , \Sigma^\infty_\cC x)_{h \Sigma_n}\]
naturally in $n$.
Here $D_nF$ denotes the fiber of $P_nF \to P_{n-1}F$.
This definition does not actually require $F$ to be reduced: the existence (and uniqueness) of such a functor $\partial_n F$ only requires that $D_nF$ is $n$-homogeneous and finitary, which also holds if $F$ itself is not reduced.
In fact, $D_n$ does not see the difference between $F$ and its reduction.

\begin{lemma}\label{lem:partial-vs-reduction}
    Let $F \colon \cC \to \cD$ be a finitary functor between differentiable categories.
    Then there are natural equivalences
    \[D_nF \simeq D_n(\red F) \quad \text{and} \quad \partial_n F \simeq \partial_n (\red F)\]
    for $n \geq 1$.
\end{lemma}

\begin{proof}
    Recall that $\red F$ is defined as the fiber of $F \to \const_{F(*)}$, cf.\ \cite[Construction 6.1.3.15]{HA}.
    The result about $D_n$ follows since $P_n$ and $P_{n-1}$ (and hence also $D_n$) preserve fiber sequences by \cite[Theorem 6.1.1.10]{HA}.
    The result about $\partial_n$ then follows from Proposition 3.1.26, Theorem 3.1.29 and Definition 3.1.30 of \cite{blansblom2025chainrulegoodwilliecalculus}.
\end{proof}

\begin{warning}
    Although the definition of the derivatives $\partial_*$ extends to unreduced functors, the chain rule (\cref{thm:chain-rule-cherry}) does not extend to unreduced functors in a naive way. 
    This can be traced back to the fact that \cite[Proposition 1.3]{AroneChing2011} fails for a composite $F \circ G$ if $G$ is not reduced.
\end{warning}

The definition of $\partial_n F$ crucially depends on the fact that $D_nF$ can be (canonically) delooped if $n \geq 1$, cf.\  \cite[Theorem 2.1]{Goodwillie2003} and \cite[Corollary 6.1.2.9]{HA}.
For $n=0$ this generally fails if $F$ is unreduced, since $D_0 F = \const_{F(*)}$.
Hence there is no sensible definition of $\partial_0 F$ for general finitary functors $F \colon \cC \to \cD$.\footnote{If $F$ is reduced, one can of course define $\partial_0 F = *$.} However, if the target $\cD$ is stable, then $\cD \simeq \Sp \cD$ and one can do the following:

\begin{definition}\label{def:definition-partial0}
    Let $F \colon \cC \to \cD$ be a (possibly unreduced) finitary functor between differentiable categories such that $\cD$ is stable.
    Then we define $\partial_0 F = F(*)$.
    More generally, we define $\partial_* F$ to be the symmetric sequence $\{\partial_n F\}_{n \geq 0}$.
\end{definition}

Note that when $F$ is reduced, this does not conflict with the notation $\partial_* F$ for the positive symmetric sequence $\{\partial_n F\}_{n \geq 1}$.
Namely, we generally view $\SSeq_{\geq 1}(\Sp \cC, \Sp \cD)$ as the full subcategory of $\SSeq(\Sp \cC, \Sp \cD)$ spanned by the symmetric sequences $G$ for which $G(0) = *$.

\subsection{The Goodwillie transform} \label{ssec-goodwillie-transform}

The purpose of this section is to define the category $\Alg_\sO(\sA)$ of $\sO$-algebras in $\sA$ for $\sO \in \Alg(\SSeq(\sA, \sA))$ and introduce a functor of $2$-categories, called the Goodwillie transform, that sends a differentiable category $\sC$ to the category of $\partial_*{\id_\sC}$-algebras in $\Sp\sC$.

\begin{definition}
    Let $\Lambda \colon \pressymst \to \Cat$ be the functor corepresented by the one point category $\ast$. 
    We will also write $\Lambda$ for the restriction of this functor to $\pressymstc$.

\end{definition}

\begin{remark} \label{ex: sseq-arity-0}
    For $\sA$ a stable presentable category, we have equivalences
    \[
    \SSeq(\ast, \sA) = \FunL(\Sym(\ast), \sA) \simeq \FunL(\Sp, \sA) \simeq \sA,
    \]
    so that $\Lambda(\sA) \simeq \sA$.
    Suppose that $\sB$ is another stable presentable category and $F \in \SSeq(\sA, \sB)$.
    We will write $\Lambda_F \colon \sA \to \sB$ for the image of $F$ under $\Lambda$.
    By \cite[Proposition 3.2.13]{blansblom2025chainrulegoodwilliecalculus}, it is given by the formula
    \[
    \Lambda_F(x) \simeq \bigoplus_{n \geq 0} F_n(x, \ldots, x)_{h\Sigma_n}.
    \]
\end{remark}

Given a symmetric sequence $F \in \SSeq(\sA,\sB)$, the formula from the previous remark shows that $\Lambda_F$ is the direct sum of homogeneous functors corresponding to $F$, so that we have $\partial_* \Lambda_F \simeq F$.
This is the basis for the following useful result, which we prove in the appendix:

\begin{proposition}[{\cref{prop: app-derivatives-lambda}}] \label{prop: lambda-diff-is-id}
    The composite
    \[
    \begin{tikzcd}
    \pressymstc \ar[r, "\Lambda"] & \diffsp \ar[r, "\partial_*"] & \pressymstc
    \end{tikzcd}
    \]
    is equivalent to the identity functor of the $2$-category $\pressymstc$.
\end{proposition}

\begin{definition}
    Let $\sA$ be a stable presentable category and let $\sO \in \Alg(\SSeq(\sA,\sA))$.
    Applying the functor $\Lambda$ yields a monad $\Lambda_\sO \colon \sA \to \sA$.
    We write $\Alg_\sO(\sA)$ for the category of $\Lambda_\sO$-algebras in $\sA$.
    If the category $\sA$ is clear from context, we will often write $\Alg_\sO$ instead of $\Alg_\sO(\sA)$.
\end{definition}

We will now define an extension of the functor $\Lambda$ to the Morita category.

\begin{definition}
    Let $\reAlg \colon \MMor \to \Cat$ be the functor corepresented by $(\ast, \unit_\ast)$.
\end{definition}

\begin{remark}
    The functor $\reAlg \colon \MMor \to \Cat$ sends a pair $(\sA, \sO)$ to the category
    \[
    \BMod_{(\sO, \unit_\ast)}(\SSeq(\ast, \sA)) \simeq \LMod_\sO(\sA).
    \]
    Note that the left action of $\SSeq(\sA, \sA)$ on $\SSeq(\ast, \sA) \simeq \sA$ is given by the monoidal functor $\Lambda \colon \SSeq(\sA, \sA) \to \End(\sA)$, so that we get an equivalence
    \[
    \LMod_\sO(\sA) \simeq \Alg_\sO(\sA).
    \]
    The functor $\reAlg$ sends a morphism $M \in \MMor((\sA, \sO), (\sB, \sP))$, which is the same as a $(\sP, \sO)$-bimodule in $\SSeq(\sA, \sB)$, to the functor
    \[
    \reAlg(M) \colon \Alg_{\sO}(\sA) \to \Alg_{\sP}(\sB)
    \]
    that sends an $\sO$-algebra $X$ to the $\sP$-algebra $M \circ_{\sO} X$.
\end{remark}

The following functor will play an important role throughout this paper.

\begin{definition}
    The \emph{Goodwillie transform} is the functor $\maAlg \colon \diff \to \Cat$ given by the composite
    \[
    \diff \xrightarrow{\partial_*} \MMor \xrightarrow{\reAlg} \Cat.
    \]
\end{definition}

\begin{remark}
Concretely, the Goodwillie transform sends a differentiable category $\sC$ to $\Alg_{\partial_*{\id_\sC}}(\Sp\sC)$, and a reduced finitary functor $F \colon \sC \to \sD$ to the functor
\[
\Alg_{\partial_*{\id_\sC}} \to \Alg_{\partial_*{\id_\sD}} \colon X \mapsto \partial_*{F} \circ_{\partial_*{\id_\sC}} X. 
\]
\end{remark}

We now prove two basic properties of the Goodwillie transform.

\begin{proposition} \label{prop: goodwillie-transf-reduced-sifted-colim-pres}
    Let $F \colon \sC \to \sD$ be a reduced finitary functor between differentiable categories.
    Then $\maAlg(F) \colon \maAlg(\sC) \to \maAlg(\sD)$ is reduced and preserves sifted colimits.
\end{proposition}
\begin{proof}
    The fact that $\maAlg(F)$ preserves sifted colimits follows immediately from the description of the Goodwillie transform given in the previous remark together with \cref{prop: exactness-composition-mor}.
    That it is reduced follows immediately from the formula of the composition product, using that $\partial_*$ takes values in positive symmetric sequences.
\end{proof}

\begin{proposition} \label{prop: alg-O-differentiable}
    The Goodwillie transform $\maAlg \colon \diff \to \Cat$ factors through $\diff \subseteq \Cat$.
\end{proposition}
\begin{proof}
    It follows from the previous proposition that every functor in the image of $\maAlg$ is reduced and finitary.
    It therefore suffices to show that $\Alg_\sO(\sA)$ is differentiable for $(\sA, \sO) \in \MMor$ with $\sO(0)=0$.
    Since $\Lambda_\sO$ preserves filtered colimits, it follows from \cite[Corollary 6.8]{henry2021highertheoriesmonads} that $\Alg_\sO(\sA)$ is presentable.
    It is pointed since $\Lambda_\sO$ is reduced.
    Now observe that limits and filtered colimits in this category are computed in $\sA$, where they commute since $\sA$ is differentiable.
    Hence $\Alg_\sO(\sA)$ is differentiable as well.
\end{proof}

\subsection{Koszul duality}

We will now briefly recall the relation between Koszul duality and Goodwillie calculus.
The form of Koszul duality we make use of was developed by Lurie in \cite[Section 5.2]{HA}.
In our setting, it takes the following form.

\begin{theorem}[{\cite[Theorem 4.2.4]{blansblom2025chainrulegoodwilliecalculus}}] \label{thm:koszul-duality-operads}
    Let $\sA$ be a stable presentable category. 
    There is an adjoint equivalence of categories
    \[
    \begin{tikzcd}
        \Alg(\SSeq_+(\sA, \sA)) \ar[r, shift left, "B"] & \coAlg(\SSeq_+(\sA, \sA)). \ar[l, shift left, "C"]
    \end{tikzcd}
    \]
\end{theorem}

\begin{remark}
We say that $\sO$ and $B\sO$ are each other's Koszul duals.
We will sometimes write $\Bahr$ and $\Cobar$ instead of $B$ and $C$.
\end{remark}

\begin{remark}
    Suppose $\sO \in \Alg(\SSeq_+(\sA, \sA))$ is a strongly positive algebra.
    The projection $\sO \to \unit_\sA$ onto the arity $1$ component makes $\unit_\sA$ into an $\sO$-bimodule and
    the underlying symmetric sequence of $B\sO$ can be computed as the two-sided bar construction
    \[
    B\sO = \unit_\sA \circ_\sO \unit_\sA.
    \]
    For any strongly positive coalgebra $\sQ$, the underlying symmetric sequence of $CQ$ is computed by the analogous cobar construction.
\end{remark}

\begin{example} \label{ex: koszul-duality-spectra}
    If we let $\sA = \Sp$ in \cref{thm:koszul-duality-operads}, we obtain an adjoint equivalence
    \[
    \begin{tikzcd}
        \Op(\Sp) \ar[r, shift left, "B"] & \mathrm{coOp}(\Sp). \ar[l, shift left, "C"]
    \end{tikzcd}
    \]
    between the categories $\Op(\Sp)$ and $\mathrm{coOp}(\Sp)$ of strongly positive operads and cooperads in spectra.
\end{example}

The version of Koszul duality discussed so far sends operads to cooperads.
By using Spanier--Whitehead duality, there is also a version that sends operads to operads.

\begin{proposition}[{\cite[Lecture 16]{lurieThursdayKoszul}}]
    The functor $\mathbf{D} \colon \SSeq(\Sp)^\op \to \SSeq(\Sp)$ given by taking levelwise Spanier--Whitehead duals admits a lax monoidal structure.
\end{proposition}

As a consequence of this proposition, the Spanier--Whitehead dual of a cooperad has the structure of an operad.

\begin{definition}
    Let $\sO \in \Op(\Sp)$.
    We define the \emph{Koszul dual operad} $K\sO$ of $\sO$ to be the Spanier--Whitehead dual of $B\sO$.
\end{definition}

\begin{remark}
    One can check that Spanier--Whitehead duality restricts to a strong monoidal anti-equivalence on the full subcategory of $\SSeq(\Sp)$ spanned by the symmetric sequences which are levelwise dualizable.
    It follows that if $\sO$ is an operad in spectra such that $\sO(n)$ is a dualizable spectrum for all $n \geq 1$, then $KK\sO \simeq \sO$.
\end{remark}

There is also a notion of Koszul duality for modules.

\begin{proposition}[{\cite[Theorem 3.26]{brantner2023pd}}]
    Let $\sA$ be a stable presentable category and let $\sO \in \Alg(\SSeq_+(\sA, \sA))$.
    For any stable presentable category $\sB$, there is an adjunction
    \[
    \begin{tikzcd}
    \RMod_{\sO}(\SSeq(\sA, \sB)) \ar[r, shift left, "B"] & \mathrm{RcoMod}_{B\sO}(\SSeq(\sA, \sB)) \ar[l, shift left, "C"].
    \end{tikzcd}
    \]
    The left adjoint $B$ sends a right $\sO$-module $M$ to a right $B\sO$-comodule with underlying object $M \circ_\sO \unit$.
    There are also versions of this adjunction for left modules and bimodules.
\end{proposition}

\begin{remark}
    If $\sO$ is an operad in spectra and $M$ is a right $\sO$-module in $\SSeq(\Sp)$, then we will write $KM$ for the right $K\sO$-module Spanier--Whitehead dual to $BM$.
\end{remark}

Suppose that $\sC$ is a differentiable category.
We then have the stabilization adjunction
\[
\begin{tikzcd}
    \Sigma^\infty_\sC \colon \sC \ar[r, shift left] & \Sp{\sC} \ar[l, shift left] \colon \Omega^\infty_\sC.
\end{tikzcd}
\]
Applying the derivatives functor to the comonad $\Sigma^\infty_\sC \Omega^\infty_\sC$, we obtain a strongly positive coalgebra $\partial_*{\Sigma^\infty_\sC \Omega^\infty_\sC}$.
The connection between Goodwillie calculus and Koszul duality is as follows.

\begin{theorem}[{\cite[Theorem 4.3.1, Corollary 4.4.4]{blansblom2025chainrulegoodwilliecalculus}}] \label{thm: calculus-and-koszul-duality-main}
    There is an equivalence
    \[
    \partial_*{\id_\sC} \simeq \Cobar\partial_*{\Sigma^\infty_\sC \Omega^\infty_\sC}.
    \]
    In other words, $\partial_*{\Sigma^\infty_\sC \Omega^\infty_\sC}$ is the Koszul dual of $\partial_*{\id_\sC}$.
    Moreover, for any differentiable category $\sD$, there is a commutative diagram
    \[
    \begin{tikzcd}
        \Fun^{\ast, \omega}(\sC, \sD) \ar[d, "- \circ \Omega^\infty_\sC"'] \ar[r, "\partial_*"] & \RMod_{\partial_*{\id_\sC}}(\SSeq(\Sp{\sC}, \Sp{\sD})) \ar[d, "B"] \\
        \mathrm{RcoMod}_{\Sigma^\infty_\sC\Omega^\infty_\sC}(\Fun^{\ast, \omega}(\Sp{\sC}, \sD) \ar[r, "\partial_*"] \ar[r] & \mathrm{RcoMod}_{\partial_*{\Sigma^\infty_\sC \Omega^\infty_\sC}}(\SSeq(\Sp{\sC}, \Sp{\sD})).
    \end{tikzcd}
    \]
    The analogous result is true for left modules and bimodules.
\end{theorem}

We will often make use of the following result.

\begin{proposition} \label{prop: derivatives-free-left-right-module}
    Let $F \colon \sC \to \sD$ be a reduced finitary functor between differentiable categories. Write $\iota \colon \partial_1{F} \to \partial_*{F}$ for the inclusion of the arity $1$ component.
    \begin{enumerate}[\upshape{(}1\upshape{)}]
        \item If $F$ preserves finite limits, then $\iota$ exhibits $\partial_*{F}$ as the free right $\partial_*{\id_\sC}$-module on $\partial_1{F}$.
        \item If $F$ preserves finite colimits, then $\iota$ exhibits $\partial_*{F}$ as the free left $\partial_*{\id_\sD}$-module on $\partial_1{F}$.
    \end{enumerate}
\end{proposition}
\begin{proof}
    Assume that $F \colon \sC \to \sD$ preserves finite limits.
    Then the composite $F \Omega^\infty_\sC \colon \Sp\sC \to \sD$ is a linear functor, so that $\partial_*{F \Omega^\infty_\sC}$ is given by $\partial_1(F \Omega^\infty_\sC) \simeq \partial_1{F}$ concentrated in arity $1$.
    This is a trivial right $\partial_*{\Sigma^\infty_\sC \Omega^\infty_\sC}$-comodule for degree reasons, so that its Koszul dual is the free right $\partial_*{\id_\sC}$-module on $\partial_1{F}$.
    But this Koszul dual is equivalent to $\partial_*{F}$ by \cref{thm: calculus-and-koszul-duality-main}.
    This proves (1).
    The proof of (2) goes exactly the same, using postcomposition with $\Sigma^\infty_\sD$ instead of precomposition with $\Omega^\infty_\sC$.
\end{proof}

\subsection{Localizations of the category of spectra}\label{ssec:localizations-of-spectra}

For a general stable presentably symmetric monoidal category $\cA$, algebras in $\SSeq(\cA,\cA)$ are not the same as $\cA$-enriched operads.
We will now describe a class of stable categories for which this \emph{is} the case.

\begin{definition}\label{def:localization-of-spectra}
    Let $L : \Sp \rightleftarrows \cA : U$ be an adjunction between presentable stable categories and suppose that $U$ is fully faithful.
    In that case we call $\cA$ a \emph{(Bousfield) localization of the category $\Sp$ of spectra}, and denote it by $\LSp$.
    We will write $L\Sph$ for the image of the sphere spectrum $\Sph$ under $L$.
\end{definition}

Note that we will always implicitly assume that $\LSp$ is stable if we call $\LSp$ a localization of the category of spectra; this excludes, for example, the category $\Sp_{\geq 0}$ of connective spectra.

\begin{definition}
    Let $\cA$ be a stable category and $\LSp$ a localization of $\Sp$.
    Then $\cA$ is called \emph{$\LSp$-local} or \emph{$\LSp$-enriched} if its mapping spectra lie in $\LSp$.
    We write $\presl_\LSp$ for the full subcategory of $\preslst$ spanned by the presentable stable $\LSp$-local categories.
\end{definition}

\begin{proposition}
    Let $L \colon \Sp \to \LSp$ be a localization of $\Sp$.
    Then the map $L \colon \Sp \to \LSp$ exhibits $\LSp$ as an idempotent algebra in $\preslst$.
    Moreover, the forgetful functor $\LMod_\LSp(\presl) \to \presl$ is fully faithful and its essential image is $\presl_\LSp$.
\end{proposition}

\begin{proof}
    This follows from \cite[Proposition 5.2.10]{CarmeliSchlankea2021AmbidexterityHeight} applied to $\mathcal{M} = \Sp$.
    However, there is a slight difference in their assumptions and ours: they assume $L \colon \Sp \to \LSp$ is compatible with the symmetric monoidal structure of $\Sp$ but only require $\LSp$ to be presentable.
    We instead assume that $\LSp$ is stable presentable, but don't require $L$ to be compatible with the symmetric monoidal structure of $\Sp$.
    Let us therefore show that these two conditions are equivalent to each other.

    If $\Sp \to \LSp$ is a map in $\presl$ that is compatible with the symmetric monoidal structure on $\Sp$, then it lifts to a map in $\CAlg(\presl)$ by \cite[Proposition 2.2.1.9]{HA}.
    In particular, it is a map of $\Sp$-modules, hence $\LSp$ is stable by \cite[Proposition 4.8.2.18]{HA}.
    For the converse, suppose $\LSp$ is stable.
    Then $\LSp$ lies in $\preslst$, so let us write $\otimes_{\Sp}$ for the left action of $\Sp$ on $\LSp$.
    By \cite[Example 2.2.1.7]{HA}, we need to show that for any $X$ and $Y$ in $\Sp$, the map $L(X \otimes Y) \to L(X \otimes ULY)$ is an equivalence.
    Observe that $\Sp \to \LSp$ is a map of $\Sp$-modules by \cite[Proposition 4.8.2.18]{HA}, so we may identify the map $L(X \otimes Y) \to L(X \otimes ULY)$ with $X \otimes_{\Sp} L(Y) \to X \otimes_{\Sp} LUL(Y)$.
    Since $L(Y) \to LUL(Y)$ is an equivalence in $\LSp$, this map is invertible.
\end{proof}

It follows in particular that $\LSp$ is a presentably symmetric monoidal category with unit given by $L\Sph$ and tensor product given by $X \otimes Y = L(UX \otimes UY)$.
By \cite[Proposition 4.8.2.10]{HA}, $\LMod_{\LSp}(\presl) \simeq \presl_{\LSp} \hookrightarrow \presl$ is a (nonunital) symmetric monoidal full subcategory inclusion, whose left adjoint is the strong symmetric monoidal functor $\LSp \otimes -$.
Moreover, let us note that by \cite[Corollary 4.8.1.10]{HA}, the presheaves functor $\cP \colon (\Cat,\times) \to (\presl,\otimes)$ is strong monoidal and that $\Fun^{\mathrm{L},\otimes}(\cP(\cC),\cA) \simeq \Fun^\otimes(\cC,\cA)$ for any presentably symmetric monoidal category $\cA$.
Defining $\cP_{\LSp}(\cC) \coloneqq \cP(\cC) \otimes \LSp \simeq \Fun(\cC^\op, \LSp)$, it follows that
\[\Fun^{\mathrm{L},\otimes}(\cP_\LSp(\cC),\cA) \simeq \Fun^{\otimes}(\cC,\cA)\]
for any $\cA$ in $\CAlg(\presl_\LSp)$, where the equivalence is given by restriction along the Yoneda embedding $\cC \hookrightarrow \cP(\cC) \to \cP_\LSp(\cC)$.
We will call this presentably symmetric monoidal structure on $\cP_{\LSp}(\cC)$ the \emph{Day convolution}.

\begin{observation}\label{obs:SSeq-is-free-symmon-cat}
It follows that $\SSeq(\LSp) \coloneqq \cP_\LSp(\Fin^\simeq) \simeq \Fun(\Fin^\simeq,\LSp)$, equipped with its Day convolution, is the free stable presentably symmetric monoidal $\LSp$-local category on a single generator.
\end{observation}

This observation was used by Brantner \cite[\S 4.1.2]{brantnerThesis} to define the composition product monoidal structure on $\SSeq(\LSp)$, whose algebras are \emph{$\LSp$-enriched operads}.\footnote{This definition of $\LSp$-enriched operads has recently been compared to the other common definitions of $\LSp$-enriched operads by Arakawa \cite{Arakawa2026EquivalenceBrantnersChuHaugsengs}.}
Namely, note that
\[\Fun^{\mathrm{L},\otimes}(\SSeq(\LSp),\SSeq(\LSp)) \simeq \SSeq(\LSp)\]
by \cref{obs:SSeq-is-free-symmon-cat}.
The left-hand side has a monoidal structure given by composition of 1-morphisms in $\CAlg(\presl_\LSp)$, and the composition product on $\SSeq(\LSp)$ is defined as the reverse of this monoidal structure.

We saw in \cref{remark:SSeq-vs-FunSSeq} that $\SSeq(\LSp,\LSp) \simeq \SSeq(\LSp)$, where the equivalence $\SSeq(\LSp,\LSp) \to \SSeq(\LSp)$ is given by evaluating a functor symmetric sequence at copies of $L\Sph$, and the converse equivalence sends a symmetric sequence $\{X_n\}_{n \geq 0}$ to the functor symmetric sequence $\{(X_n \otimes (-)^{\otimes n})_{h\Sigma_n}\}_{n \geq 0}$.
We defined another composition product on $\SSeq(\LSp,\LSp)$ in \cref{prop:definition-prsymst} by viewing it as a mapping category in $\pressymst \subset \coCAlg(\preslst)$, so let us show that these agree.

\begin{proposition}\label{prop:bratner-vs-our-composition-prod}
    The composition product on $\SSeq(\LSp,\LSp)$ agrees with Brantner's composition product on $\SSeq(\LSp)$ under the equivalence $\SSeq(\LSp,\LSp) \simeq \SSeq(\LSp)$.
    In particular, algebras in $\SSeq(\LSp,\LSp)$ agree with $\LSp$-enriched operads.
\end{proposition}

For the proof, it will be convenient to work with the full sub-2-category $\pressym_\LSp$ of $\coCAlg(\presl_\LSp)$ spanned by the cofree commutative coalgebras.
It follows as in \cite[Proposition 3.2.2]{blansblom2025chainrulegoodwilliecalculus} that the cofree commutative coalgebra on $\cC \in \preslst$ is given by the formula $\Sym \cC \simeq \coprod_{n \geq 0} \cC^{\otimes n}_{h\Sigma_n}$; we will simply write $\cC$ for this object in $\pressym_\LSp$.

\begin{lemma}\label{lem:inclusion-pressymlsp-into-pressymst}
    There is a fully faithful inclusion $\pressym_\LSp \hookrightarrow \pressymst$ which is the identity on objects.
\end{lemma}

There are two subtleties with this statement:
First, observe that the formula for $\Sym \cC$ depends on whether it is computed in $\preslst$ or $\presl_\LSp$, since $\cC^{\otimes 0} = \Sp$ in $\preslst$ while $\cC^{\otimes 0} = \LSp$ in $\presl_\LSp$.
Secondly, the inclusion $\presl_\LSp \hookrightarrow \preslst$ is only lax monoidal on the unit, hence it does not induce a map $\coCAlg(\presl_\LSp) \hookrightarrow \coCAlg(\preslst)$.

\begin{proof}
    We saw that the left adjoint $- \otimes \LSp \colon \preslst \to \presl_\LSp$ is strong monoidal, hence it induces a functor $\coCAlg(\preslst) \to \coCAlg(\presl_\LSp)$.
    It is easily verified that this functor preserves cofree commutative coalgebras, hence it restricts to a functor $F \colon \pressymst \to \pressym_\LSp$ sending $\cC$ to $\cC \otimes \LSp$.
    We claim that the desired inclusion is the (2-categorical) right adjoint of $F$.
    Given $\cC$ in $\presl_\LSp$, consider the action map $\varepsilon \colon F(\cC) = \cC \otimes \LSp \eqarrow \cC$ in $\presl_\LSp$, viewed as a symmetric sequence concentrated in arity 1.
    It is easily verified that for any $\cD$ in $\pressymst$, both functors
    \[\pressymst(\cD, \cC) \xrightarrow{F} \pressym_\LSp(F(\cD),F(\cC)) \xrightarrow{\varepsilon \circ -} \pressym_{\LSp}(F(\cD),\cC)\]
    are equivalences.
    It follows by the dual of \cite[Corollary A.2.18]{blansblom2025chainrulegoodwilliecalculus} that $F$ admits a fully faithful right adjoint which is the identity on objects.
\end{proof}

It follows that in \cref{prop:bratner-vs-our-composition-prod}, we may view $\SSeq(\LSp,\LSp)$ as the endomorphism monoidal category of $\LSp$ in $\pressym_\LSp$ instead of in $\pressymst$.
The result is then a consequence of the following lemma.

\begin{lemma}\label{lem:Sym-dualizable}
    Let $\cA$ in $\presl_\LSp$ be dualizable.
    Then $\Sym \cA$ is again dualizable in $\presl_\LSp$.
\end{lemma}

\begin{proof}
    This follows by the same proof as \cite[Lemma 3.2.20]{blansblom2025chainrulegoodwilliecalculus} if one replaces the citation \cite[Proposition D.7.3.1]{SAG} with \cite[Theorem 1.49]{ramzi2024dualizablecats}.
\end{proof}

\begin{proof}[Proof of \cref{prop:bratner-vs-our-composition-prod}]
    This follows by exactly the same argument as \cite[Proposition 3.2.16]{blansblom2025chainrulegoodwilliecalculus}, replacing $\preslst$ with $\presl_\LSp$ throughout the proof and using \cref{lem:inclusion-pressymlsp-into-pressymst} to identify $\SSeq(\LSp,\LSp)$ with the endomorphism monoidal category $\pressym_\LSp(\LSp,\LSp)$.
\end{proof}

\subsection{Day convolution and coendomorphism operads}\label{ssec:coend-operads-and-Koszul-duality}

In \cite{BlomMalinTaggart}, Malin, Taggart and the second author give a description of Koszul duality for operads in terms of \emph{coendomorphism operads}.
We will use this description in \cref{ssec:coend-operad-suspension} to prove \cref{introcor: product-rule-koszul-duality,introcor: right-module-der}, which are then used in \cref{ssec:pointed-spaces} to show that the operad structure on $\partial_* \id_{\Spc_*}$ obtained through the chain rule (see \cref{remark:bimodule-structure-on-partialF}) agrees with the spectral Lie operad.
Let us therefore summarize the constructions and results from \cite{BlomMalinTaggart} that we use.
Throughout, fix a localization $\LSp$ of $\Sp$ in the sense of \cref{def:localization-of-spectra}.

If $\cV$ is a presentably symmetric monoidal $\LSp$-local category, that is, an object of $\CAlg(\presl_\LSp)$, then by the equivalence
\[\cV \simeq \Fun^{\mathrm{L},\otimes}(\Sym(\LSp),\cV)\]
we obtain a right $\Fun^{\mathrm{L},\otimes}(\Sym(\LSp),\Sym(\LSp))$-action on $\cV$.
Since Brantner's composition product on $\SSeq(\LSp)$ is obtained by reversing the monoidal structure on $\Fun^{\mathrm{L},\otimes}(\Sym(\LSp),\Sym(\LSp))$, we obtain a left $\SSeq(\LSp,\LSp)$-action on $\cV$.
Given $x \in \cV$, one can verify that the symmetric sequence $\{\ulmap(x^{\otimes n},x)\}_{n \geq 0}$ in $\LSp$ is an endomorphism object of $x$ in the sense of \cite[\S 4.7.1]{HA}.
In particular, by \cite[Corollary 4.7.1.40]{HA} it obtains the structure of an operad; we define this to be the \emph{endomorphism operad} of $x$.
For a general $\cV$, one can do the following:

\begin{construction}
    Let $\cV$ be a locally small stably symmetric monoidal $\LSp$-local category and $x$ an object of $\cV$.
    Pick an arbitrary small full stable symmetric monoidal subcategory $\cV_0$ such that $x \in \cV_0$.
    The \emph{coendomorphism operad} $\coEnd(x)$ of $x$ is defined as the endomorphism operad of the corepresentable $\ulmap(x,-)$ in $\cP_\LSp(\cV_0^\op)$.
    It follows by the enriched Yoneda lemma that $\coEnd(x)_n \simeq \ulmap(x,x^{\otimes n})$.
    Moreover, since endomorphism operads are functorial in strong symmetric monoidal functors, this construction is independent of the chosen full subcategory $\cV_0$.
\end{construction}

In fact, \cite{BlomMalinTaggart} will show that any stably symmetric monoidal $\LSp$-local category $\cV$ comes with a canonical $(\SSeq(\LSp),\circ^\mathrm{rev})$-enrichment given by 
\[\mathrm{hom}(x,y) \simeq \{\ulmap(x,y^{\otimes n})\}_{n \geq 0},\]
where $\circ^\mathrm{rev}$ denotes the reversed composition product defined by $A \circ^\mathrm{rev} B \coloneqq B \circ A$.
From this it follows not only that $\coEnd(x) = \{\ulmap(x,x^{\otimes n})\}_{n \geq 0}$ admits an operad structure, but also that $\{\ulmap(y,x^{\otimes n})\}_{n \geq 0}$ admits a right $\coEnd(x)$-module structure for $y$ in $\cV$.

In \cite{BlomMalinTaggart}, the following is then proved.
Recall that an operad $\cO$ is called \emph{strongly positive} if $\cO(0)=0$ and $\cO(1) = \LSp$.

\begin{theorem}\label{prop:KP-is-coend-operad}
    Let $\cO$ be a strongly positive $\LSp$-enriched operad such that for every $n$, the spectrum $\cO(n)$ is dualizable in $\LSp$.
    Let $\unit$ be the trivial right $\cO$-module $L\Sph$ concentrated in degree 1.
    Then the coendomorphism operad 
    \[
    \coEnd(\unit) = \ulmap(\unit, \unit^{\circledast \bullet})
    \]
    is equivalent to the Koszul dual $K\cO$, where $\circledast$ denotes the Day convolution product on $\RMod_\cO(\SSeq(\LSp))$.
    Moreover, for any levelwise dualizable right $\cO$-module $M$, the right $K\cO$-module $\ulmap(M,\unit^{\circledast \bullet})$ is the Koszul dual $KM$.
\end{theorem}

Let us explain what we mean by the ``Day convolution product'' on $\RMod_{\cO}(\SSeq(\LSp))$.
As mentioned above, $\SSeq(\LSp) = \cP_\LSp(\Fin^\simeq)$ admits a Day convolution symmetric monoidal structure (with respect to disjoint union on $\Fin^\simeq$), which we will denote by $\circledast$.
It is given by the formula
\[
(A \circledast B)_n \simeq \bigoplus_{i + j = n} \Ind_{\Sigma_i\times\Sigma_j}^{\Sigma_n} A_i \otimes B_j.
\]
It makes $\SSeq(\LSp)$ into the free presentably symmetric monoidal $\LSp$-local category.
This Day convolution product on $\RMod_\cO(\SSeq(\LSp))$ is defined as a lift of the Day convolution product of $\SSeq(\LSp)$.
The moral reason for the existence of such a lift is that the composition product right action of $\SSeq(\LSp)$ on itself distributes over $\circledast$, in the sense that there is an equivalence
\[(S \circledast T) \circ U \simeq (S \circ U) \circledast (T \circ U).\]
We refer the reader to \cref{def:convolution-product-sseq} for an explicit construction of the Day convolution product on $\RMod_\cO(\SSeq(\LSp))$ and to \cref{prop:day-convolution-is-indeed-lift} for a proof that it is a lift of the Day convolution product on $\SSeq(\LSp)$.

\section{Properties of the Morita category}

The goal of this section is to prove that the functor $\partial_* \colon \diff \to \MMor$ preserves a number of $2$-categorical constructions.
We start in \cref{ssec: local-exactness-MMor} by showing that it preserves colimits and finite limits on mapping categories.
In \cref{ssec:products-MMor}, we prove that it preserves cartesian products.
\cref{ssec: cotensors-mmor} deals with the preservation of cotensors, and \cref{ssec: pullbacks-mmor} with the preservation of certain pullbacks.

\subsection{Local exactness} \label{ssec: local-exactness-MMor}

\begin{definition}
    Let $F \colon \sX \to \sY$ be a functor of $2$-categories.
    Let $K$ be a small category.
    We say that $F$ \emph{locally preserves $K$-shaped colimits} if for every pair of objects $x, y \in \sX$ the induced functor
    \[
    F \colon \sX(x, y) \to \sY(F(x), F(y))
    \]
    preserves $K$-shaped colimits.
    We say that $F$ \emph{locally preserves colimits} if it locally preserves $K$-shaped colimits for all small categories $K$.
    In the same way, we define what it means for $F$ to locally preserve limits.
\end{definition}

All preservation properties of the derivatives functor that we will prove in this section rely on the following result.

\begin{proposition}\label{prop: local-exactness-derivatives}
    The functor $\partial_* \colon \diff \to \MMor$ locally preserves colimits and finite limits.
\end{proposition}
\begin{proof}
This is proved as \cite[Proposition 3.10]{blansheuts2026characterizationspectrallieoperad} for the restriction of $\partial_*$ to the subcategory of $\diff$ spanned by the differentiable categories whose stabilization is equivalent to $\Sp$.
The statement for $\diff$ is proved in exactly the same way.
\end{proof}

\begin{remark}
        In more concrete terms, this proposition says that for each pair of differentiable categories $\sC$ and $\sD$, the functor
    \[
    \partial_* \colon \Fun^{\ast, \omega}(\sC, \sD) \to \BMod_{(\partial_*\id_\sD, \partial_* \id_\sC)}(\SSeq(\Sp\sC, \Sp\sD)) 
    \]
    preserves colimits and finite limits.
\end{remark}

\begin{corollary}
    The Goodwillie transform $\maAlg \colon \diff \to \diff$ locally preserves colimits and finite limits.
\end{corollary}
\begin{proof}
    Since $\maAlg = \reAlg \circ \partial_*$ and $\partial_*$ locally preserves colimits and finite limits, it suffices to show that the same holds for $\reAlg$, which follows from \cref{prop: exactness-composition-mor}.
\end{proof}

\subsection{Products}\label{ssec:products-MMor}

The purpose of this section is to show that the functor $\partial_* \colon \diff \to \MMor$ preserves finite products.
We will deduce this from the fact that $\partial_*$ \emph{locally} preserves finite products.
Consider the following illustrative example.

\begin{example}\label{ex:cartesian-product-additive-cat}
    In an additive category, cartesian products admit an ``equational'' description.
    Namely, the maps
    \[x \xleftarrow{p_1} w \xrightarrow{p_2} y\]
    exhibit $w$ as the cartesian product of $x$ and $y$ if and only if there exist $i_1 \colon x \to w$ and $i_2\colon y \to w$ such that
    \begin{enumerate}[(1)]
        \item\label{item1:cart-additive} $p_1 i_1 \simeq \id_x$ and $p_2 i_2 \simeq \id_y$,
        \item\label{item2:cart-additive} $p_1 i_2 \simeq 0$ and $p_2 i_1 \simeq 0$,
        \item\label{item3:cart-additive} $i_1 p_1 + i_2 p_2 \simeq \id_{w}$.
    \end{enumerate}
\end{example}

In particular, this implies that any $\mathrm{CGrp}$-enriched functor between additive categories preserves cartesian products.
We will prove a categorification of this statement.

\begin{theorem}\label{thm:equational-description-prod}
    Let $\sX$ be a 2-category such that
    \begin{itemize}
        \item the mapping categories $\sX(-,-)$ admit finite cartesian products,
        \item the mapping categories $\sX(-,-)$ are pointed, and
        \item composition of 1-morphisms $\circ \colon \sX(y,z) \times \sX(x,y) \to \sX(x,z)$ preserves finite cartesian products in the first variable.
    \end{itemize}
    A span $x \xleftarrow{p_1} w \xrightarrow{p_2} y$ in $\sX$ exhibits $w$ as the cartesian product of $x$ and $y$ if and only if $p_1$ and $p_2$ admit right adjoints $i_1$ and $i_2$ such that
    \begin{enumerate}[\upshape{(}1\upshape{)}]
        \item\label{item1:equational-description-prod} the counits $\epsilon_1 \colon p_1 i_1 \Rightarrow \id_x$ and $\epsilon_2 \colon p_2 i_2 \Rightarrow \id_y$ are invertible, and
        \item\label{item2:equational-description-prod} the 2-morphism $(\eta_1,\eta_2) \colon \id_{w} \Rightarrow (i_1 p_1) \times (i_2 p_2)$ is invertible.
    \end{enumerate}
\end{theorem}

\begin{remark} \label{rem: equational-description-terminal}
    There is a similar description of terminal objects in $2$-categories.
    Suppose that $\sX$ is a $2$-category such that all mapping categories in $\sX$ have terminal objects and composition of $1$-morphisms 
    \[
    \circ \colon \sX(y, z) \times \sX(x, y) \to \sX(x,z)
    \]
    preserves terminal objects in the first variable.
    Then an object $x \in \sX$ is terminal in $\sX$ if and only if $\id_x$ is a terminal object in $\End(x)$.
    To see this, note that if $y \in \sX$ is another object and $f \in \sX(y, x)$, then we have
    \[
    f \simeq \id_x \circ f \simeq \ast \circ f \simeq \ast.
    \]
    This means that every object in $\sX(y, x)$ is equivalent to the terminal object, and hence $\sX(y, x) \simeq \ast$.
\end{remark}

Before proving \cref{thm:equational-description-prod}, we describe a number of consequences.

\begin{corollary}\label{cor:local-cart-global-cart-functor}
    Let $\sX$ and $\sY$ be 2-categories satisfying the assumptions of \cref{thm:equational-description-prod} and let $F \colon \sX \to \sY$ be a 2-functor such that for any $x,y \in \sX$, the functor $F_{x,y} \colon \sX(x,y) \to \sY(Fx,Fy)$ preserves finite cartesian products.
    Then $F$ preserves all finite products that exist in $\sX$.
\end{corollary}

\begin{proof}
    This follows directly from the characterization of cartesian products in $\sX$ and $\sY$ given in \cref{thm:equational-description-prod} together with the characterization of terminal objects given in \cref{rem: equational-description-terminal}.
\end{proof}

\begin{corollary}\label{cor:partial-to-mmor-preserves-prod}
    The functor $\partial_* \colon \diff \to \MMor$ preserves finite products.
\end{corollary}
\begin{proof}
    It is clear that both $\diff$ and $\MMor$ satisfy the assumptions from \cref{thm:equational-description-prod}, except perhaps the condition that composition of $1$-morphisms in $\MMor$ preserves finite cartesian products in the first variable, which follows from \cref{prop: exactness-composition-mor}.
    By \cref{prop: local-exactness-derivatives}, the derivatives functor locally preserves finite cartesian products. 
    The statement now follows from \cref{cor:local-cart-global-cart-functor}.
\end{proof}

We can also use the criterion of \cref{thm:equational-description-prod} to give an explicit description of products in $\MMor$.
First of all, if $\sA$ and $\sB$ are stable presentable categories, then their cartesian product $\sA \times \sB$ is also their product in $\pressymst$, since we have
\[
\SSeq(\sC, \sA \times \sB) \simeq \FunL(\Sym(\sC), \sA \times \sB) \simeq \SSeq(\sC, \sA) \times \SSeq(\sC, \sB).
\]
Since $\pressymst$ has all cartesian products, we obtain a functor of $2$-categories
\begin{equation} \tag{i} \label{eq: cartesian-product-functor-pressymst}
- \times - \colon \pressymst \times \pressymst \to \pressymst.
\end{equation}
This implies that given a pair of objects $(\sA, \sO)$ and $(\sB, \sP)$ in $\MMor$, we obtain a new object $(\sA \times \sB, \sO \bartimes \sP)$, where $\sO \bartimes \sP$ is the algebra obtained by applying the monoidal functor
\[
\SSeq(\sA, \sA) \times \SSeq(\sB, \sB) \to \SSeq(\sA \times \sB, \sA \times \sB)
\]
induced by \cref{eq: cartesian-product-functor-pressymst} to the pair of algebras $(\sO, \sP)$.
We then have:

\begin{proposition}
    The category $\MMor$ has all cartesian products.
    The product of $(\sA, \sO)$ and $(\sB, \sP)$ is given by $(\sA \times \sB, \sO \bartimes \sP)$.
\end{proposition}
\begin{proof}
    The functor \cref{eq: cartesian-product-functor-pressymst} is easily seen to preserve geometric realizations on mapping categories.
    Now, the construction that sends a $2$-category to its Morita category is natural in $2$-functors that locally preserve geometric realizations, as follows from the description given in \cite[Part II]{Blom2024StraighteningEveryFunctor}.
    It is moreover a direct consequence of the universal property of Morita categories that the construction preserves products of $2$-categories.
    Applying the Morita construction to the functor \cref{eq: cartesian-product-functor-pressymst} therefore gives rise to a functor
    \begin{equation} \tag{ii} \label{eq: cartesian-product-functor-on-mmor}
    - \times - \colon \MMor \times \MMor \to \MMor,
    \end{equation}
    that sends a pair of objects $(\sA, \sO)$ and $(\sB, \sP)$ in $\MMor$ to the object $(\sA \times \sB, \sO \bartimes \sP)$.
    
    We now prove that $(\sA \times \sB, \sO \bartimes \sP)$
    is the product of $(\sA, \sO)$ and $(\sB, \sP)$ in $\MMor$ using the criterion from \cref{thm:equational-description-prod}.
    First observe that $\MMor((\sB, \sP), \ast) \simeq \ast$, so that there is a unique $1$-morphism $p \colon (\sB, \sP) \to \ast$, which is easily seen to have a right adjoint $i \colon \ast \to (\sB, \sP)$ corresponding to the trivial left $\sP$-module on the zero object in $\SSeq(\ast, \sB) \simeq \sB$.
    It is evident that the counit of this adjunction is an equivalence.
    This gives rise to an adjunction $(\id_\sA, p) \dashv (\id_\sA, i)$ in $\MMor \times \MMor$.
    Applying the functor from \cref{eq: cartesian-product-functor-on-mmor} to this adjunction, we obtain an adjunction
    \[
    p_\sA \colon (\sA \times \sB, \sO \bartimes \sP) \rightleftarrows (\sA, \sO) \colon i_\sA
    \]
    in $\MMor$ for which the counit is an equivalence.
    In the same way, we obtain an adjunction $p_\sB \dashv i_\sB$ that projects onto $(\sB, \sP)$.
    The morphisms $p_\sA$ and $p_\sB$ therefore satisfy condition (1) from \cref{thm:equational-description-prod} and (2) is easily verified.
    This proves the claim.
\end{proof}

\begin{example} \label{ex: inclusion-pressymst-mor-preserves-prod}
    The previous proposition also implies that the inclusion $\pressymst \hookrightarrow \MMor$ preserves products.
    Indeed, if $\sA$ and $\sB$ are stable presentable categories, then their images in $\MMor$ are given by the pairs $(\sA, \unit_\sA)$ and $(\sB, \unit_\sB)$, and the algebra $\unit_\sA \bartimes \unit_\sB$ is equivalent to $\unit_{\sA \times \sB}$, since the functor from \cref{eq: cartesian-product-functor-pressymst} preserves identity morphisms.
    Therefore, the cartesian product of $(\sA, \unit_\sA)$ and $(\sB, \unit_\sB)$ is $(\sA \times \sB, \unit_{\sA \times \sB})$.
    This also follows immediately from \cref{cor:local-cart-global-cart-functor}.
\end{example}

\begin{example} \label{ex: inclusion-pressymlsp-pressymst-preserves-prod}
    \cref{cor:local-cart-global-cart-functor} also implies that the inclusion $\pressym_\LSp \hookrightarrow \pressymst$ described in \cref{ssec:localizations-of-spectra} preserves products.
\end{example}

We now turn to the proof of \cref{thm:equational-description-prod}.
We will first prove a version of the result for $\Cat$ and then deduce the theorem by a Yoneda-type argument.

\begin{lemma}\label{lem:recognizing-cartesian-products-in-Cat}
    Let $\cC$ be a pointed category that admits finite cartesian products and let
    \[\cD \xleftarrow{p_1} \cC \xrightarrow{p_2} \cE\]
    be a span of categories.
    Then $p_1$ and $p_2$ exhibit $\cC$ as the cartesian product of $\cD$ and $\cE$ if and only if
    \begin{enumerate}[(1)]
        \item $p_1$ and $p_2$ admit fully faithful right adjoints $i_1$ and $i_2$, such that
        \item for any $c$ in $\cC$, the map $(\eta_1, \eta_2) \colon c \to i_1p_1(c) \times i_2p_2(c)$ induced by the units of $p_1 \dashv i_1$ and $p_2 \dashv i_2$ is an equivalence.
    \end{enumerate}
\end{lemma}

\begin{proof}
    For the ``only if'' direction, note that the projections $\cD \times \cE \to \cD$ and $\cD \times \cE \to \cE$ admit right adjoints given by $d \mapsto (d,*)$ and $e \mapsto (*,e)$, which clearly satisfy these conditions.
    For the converse, let us write $i_1 \bartimes i_2$ for the composite $\cD \times \cE \xrightarrow{i_1 \times i_2} \cC \times \cC \xrightarrow{\times} \cC$.
    Observe that the first functor is fully faithful by assumption.
    We leave it as an exercise to the reader to verify that for any pointed category $\cC$ with finite cartesian products, the functor $\times \colon \cC \times \cC \to \cC$ is conservative.
    We conclude that $i_1 \bartimes i_2$ is conservative.
    Now note that we have natural equivalences
    \begin{align*}
    &\Map_{\cD \times \cE}((p_1c,p_2c),(d,e)) \simeq \Map_{\cD}(p_1c,d) \times \Map_\cE(p_2c,e) \\
    &\simeq \Map_\cC(c,i_1d) \times \Map_\cC(c,i_2e) \simeq \Map_\cC(c,i_1d \times i_2e),
    \end{align*}
    hence $i_1 \bartimes i_2$ is right adjoint to $(p_1,p_2) \colon \cC \to \cD \times \cE$.
    Chasing $\id_{(p_1c,p_2c)}$ through these equivalences, it follows that the unit of this adjunction is $(\eta_1,\eta_2)$, which is invertible by assumption.
    Since the right adjoint $i_1 \bartimes i_2$ is conservative, it follows that $(p_1,p_2)$ is an equivalence. 
\end{proof}

\begin{remark}
    The two conditions of \cref{lem:recognizing-cartesian-products-in-Cat} correspond to items \ref{item1:cart-additive} and \ref{item3:cart-additive} of \cref{ex:cartesian-product-additive-cat}.
    In particular, the two conditions of \cref{lem:recognizing-cartesian-products-in-Cat} already imply that $p_2 i_1 \simeq *$ and $p_1 i_2 \simeq *$.
\end{remark}

\begin{remark}
    If one only requires that $\cC$ admits finite cartesian product but not that $\cC$ is pointed, then $\times \colon \cC \times \cC \to \cC$ need not be conservative.
    In this case \cref{lem:recognizing-cartesian-products-in-Cat} still holds if one adds the following assumption:
    \begin{enumerate}[(3)]
        \item for any $d$ in $\cD$ and $e$ in $\cE$, the projections $p_1(\pi_1) \colon p_1(i_1d \times i_2e) \to p_1(i_1d)$ and $p_2(\pi_2) \colon p_2(i_1d \times i_2e) \to p_2(i_2e)$ are equivalences.
    \end{enumerate}
    This condition can be seen as an analogue of item \ref{item2:cart-additive} from \cref{ex:cartesian-product-additive-cat}.
\end{remark}

Recall that in a 2-category $\sX$, given a 1-morphism $f \colon x \to y$ and a 2-morphism
\[\begin{tikzcd}
    y & z
    \arrow[""{name=0, anchor=center, inner sep=0}, "g", curve={height=-12pt}, from=1-1, to=1-2]
    \arrow[""{name=1, anchor=center, inner sep=0}, "h"', curve={height=12pt}, from=1-1, to=1-2]
    \arrow["\alpha"', shorten >=3pt, shorten <=3pt, Rightarrow, from=0, to=1]
\end{tikzcd}\]
their \emph{whiskering} is a 2-morphism $\alpha \star f \colon gf \Rightarrow hf$ which can concretely be constructed as the image of the pair $(\alpha,\id_f)$ under the composition functor
\[\sX(y,z) \times \sX(x,y) \to \sX(x,z).\]

\begin{proof}[Proof of \cref{thm:equational-description-prod}]
    To distinguish cartesian products in $\sX$ and in its mapping categories $\sX(-,-)$, we will denote the latter by $\boxtimes$.
    For the ``if'' direction, let $z$ be any object in $\sX$.
    Postcomposition yields adjunctions
    \[\begin{tikzcd}
        {\sX(z,x)} & {\sX(z,w)} & {\sX(z,y)}
        \arrow[""{name=0, anchor=center, inner sep=0}, "{i_{1*}}"', bend right=12, from=1-1, to=1-2,hook]
        \arrow[""{name=1, anchor=center, inner sep=0}, "{p_{1*}}"', bend right=12, from=1-2, to=1-1]
        \arrow[""{name=2, anchor=center, inner sep=0}, "{p_{2*}}", bend left=12, from=1-2, to=1-3]
        \arrow[""{name=3, anchor=center, inner sep=0}, "{i_{2*}}", bend left=12, from=1-3, to=1-2,hook']
        \arrow["\dashv"{anchor=center, rotate=-90}, draw=none, from=1, to=0]
        \arrow["\dashv"{anchor=center, rotate=-90}, draw=none, from=2, to=3]
    \end{tikzcd}\]
    The units and counits of these adjunctions are obtained by whiskering with the units and counits of $p_1 \dashv i_1$ and $p_2 \dashv i_2$. In particular, the counits are invertible and hence $i_{1*}$ and $i_{2*}$ are fully faithful.
    By \cref{lem:recognizing-cartesian-products-in-Cat}, it follows that $(p_{1*},p_{2*}) \colon \sX(z,w) \to \sX(z,x) \times \sX(z,y)$ is an equivalence if for any $f \colon z \to w$, the 2-morphism
    \[(\eta_1 \star f, \eta_2 \star f) \colon f \Rightarrow i_1p_1f \boxtimes i_2p_2f\]
    is invertible.
    Note that this 2-morphism factors as the composite
    \[\begin{tikzcd}[column sep=large]
        f & {(i_1p_1 \boxtimes i_2p_2) \circ f} & {i_1p_1 f \boxtimes i_2p_2 f}
        \arrow["{(\eta_1, \eta_2) \star f}", "\sim"', Rightarrow, from=1-1, to=1-2]
        \arrow["\sim"',Rightarrow, from=1-2, to=1-3]
    \end{tikzcd}\]
    where the left arrow is an equivalence since $(\eta_1,\eta_2)$ is, and the right arrow is an equivalence since $- \circ f$ preserves cartesian products.
    It follows that $p_1$ and $p_2$ exhibit $w$ as the cartesian product $x \times y$.

    For the converse, let $w \simeq x \times y$ with projections $p_1$ and $p_2$ be given.
    We now construct the right adjoint $i_1$, the right adjoint $i_2$ follows similarly.
    For any $z$ in $\sX$, the projection $p_{1*} \colon \sX(z,x \times y) \to \sX(z,x)$ can be identified with the projection $\sX(z,x) \times \sX(z,y) \to \sX(z,x)$.
    As in \cref{lem:recognizing-cartesian-products-in-Cat}, this admits a right adjoint $R_{1,z}$ given by $f \mapsto (f,*)$.
    This right adjoint is natural since $- \circ g$ preserves terminal objects.
    A straightforward diagram chase\footnote{Alternatively, one could use the 2-categorical Yoneda embedding and the fact that right adjoints in $\Fun(\sX^\op,\Cat)$ are pointwise right adjoints satisfying the Beck--Chevalley condition.} shows that $i_1 \coloneqq R_{1,x}(\id_x)$ is right adjoint to $p_1$ in $\sX$, that $R_{1,z}(f) = i_1 f$ for any $f \colon z \to x$, and that the counit $\epsilon_1 \colon p_1i_1 \Rightarrow \id_x$ is invertible.
    To conclude the proof, observe that the unit $\eta_1$ of this adjunction is given by the unit $\id_w \Rightarrow R_{1,w}(p_{1*}(\id_w)) \simeq i_1p_1$ of $R_{1,w} \dashv p_{1*}$, and similarly for $\eta_2$.
    Then the 2-morphism $(\eta_1,\eta_2) \colon \id_w \Rightarrow (i_1p_1) \boxtimes (i_2p_2)$ is invertible since under the equivalence $\sX(w,w) \simeq \sX(w,x) \times \sX(w,y)$, it gets identified with the invertible 2-morphism $(p_1,p_2) \Rightarrow (p_1 \boxtimes *, * \boxtimes p_2)$.
\end{proof}

\subsection{Functor categories} \label{ssec: cotensors-mmor}

Let $\sC$ be a differentiable category and $K$ a small category.
Then $\Fun(K, \sC)$ is again differentiable.
The purpose of this section is to compute the derivatives of the identity functor in $\Fun(K, \sC)$.
The result can conveniently be phrased in terms of cotensors with small categories.

\begin{definition}
    Let $\sX$ be a $2$-category and suppose we are given $y \in \sX$ and a small category $K$. 
    We say $y$ admits a cotensor with $K$ if there exists an object $y^K$ together with a map $K \to \sX(y^K, y)$ such that for every $x \in \sX$ the composite
    \[
    \sX(x, y^K) \to \Fun(\sX(y^K, y), \sX(x, y)) \to \Fun(K, \sX(x, y))
    \]
    is an equivalence of categories.
    We say that $\sX$ admits cotensors with small categories if for every $y \in \sX$ and small category $K$, the object $y$ admits a cotensor with $K$.
\end{definition}

\begin{remark}
    By the 2-categorical Yoneda lemma, we could have equivalently stated this definition as follows:
    $y \in \sX$ admits a cotensor with $K$ if there exists an object $y^K \in \sX$ and a natural equivalence
    \[
    \sX(-, y^K) \simeq \Fun(K, \sX(-, y)),
    \]
    where both source and target are functors of $2$-categories of the form $\sX \to \Cat$.
\end{remark}

\begin{example}
    The $2$-category $\diff$ admits cotensors with small categories.
    Given a differentiable category $\sC$ and a small category $K$, the cotensor $\sC^K$ is given by the functor category $\Fun(K, \sC)$.
    This is witnessed by the functor $\ev \colon K \to \Fun(\sC^K, \sC)$ that sends $k \in K$ to the evaluation functor $\ev_k \colon \sC^K \to \sC$.
\end{example}

\begin{definition}
    Let $F \colon \sX \to \sY$ be a functor of $2$-categories and suppose that $\sX$ admits cotensors with small categories.
    Then we say that $F$ \emph{preserves cotensors} if for every $x \in \sX$ and small category $K$, the composite
    \[
    K \to \sX(x^K, x) \xrightarrow{F} \sY(F(x^K), F(x))
    \]
    exhibits $F(x^K)$ as a cotensor of $F(x)$ with $K$.
\end{definition}

\begin{remark}
    If $F \colon \sX \to \sY$ preserves cotensors, then for every $x \in \sX$ and small category $K$, the object $F(x)$ admits a cotensor with $K$ and $F(x^K) \simeq F(x)^K$.
\end{remark}

The main result of this section is:

\begin{theorem} \label{thm: derivatives-preserves-cotensors}
    The functor $\partial_* \colon \diff \to \MMor$ preserves cotensors with small categories.
\end{theorem}

Before giving the proof, we record a number of consequences.

\begin{corollary} \label{cor: goodwillie-transform-preserves-cotensors}
The Goodwillie transform $\maAlg \colon \diff \to \diff$ preserves cotensors with small categories.    
\end{corollary}
\begin{proof}
    By definition, $\maAlg = \reAlg \circ \partial_*$.
    The derivatives functor $\partial_*$ preserves cotensors by the previous theorem.
    The result now follows since corepresentable functors preserve cotensors with small categories and $\reAlg$ is corepresentable.
\end{proof}

\begin{corollary}
    Let $\sC$ be a differentiable category and $K$ a small category.
    Then the Goodwillie transform sends the colimit functor
    \[
    \colim \colon \Fun(K, \sC) \to \sC
    \]
    to the colimit functor
    \[
    \colim \colon \Fun(K, \Alg_{\partial_*{\id_\sC}}) \to \Alg_{\partial_*{\id_\sC}}.
    \]
    The same holds for the limit functor if $K$ is finite.
\end{corollary}
\begin{proof}
    If $\sX$ is a $2$-category that admits cotensors, and $\alpha \colon I \to J$ is a functor of small categories, then for every $x \in \sX$ there is a restriction morphism $\alpha^* \colon x^J \to x^I$. 
    This morphism is preserved by any functor out of $\sX$ that preserves cotensors.
    For $\sC \in \diff$, one can easily check that the restriction morphism $\alpha^* \colon \sC^J \to \sC^I$ obtained in this way is the usual restriction functor given by precomposition with $\alpha$.

    Applying this to the functor $K \to \ast$, we find that the Goodwillie transform sends the diagonal functor $\Delta \colon \sC \to \Fun(K, \sC)$ to the diagonal functor
    \[
    \Delta \colon \Alg_{\partial_*{\id_\sC}} \to \Fun(K, \Alg_{\partial_*{\id_\sC}}).
    \]
    The colimit and limit functors are the left and right adjoints of $\Delta$ respectively.
    The colimit functor lies in $\diff$ for all small categories $K$.
    For the limit functor, this is in general only true if $K$ is finite; otherwise, it can fail to be finitary.
    The result now follows since any functor of $2$-categories preserves adjunctions.
\end{proof}

\begin{remark}
    The same proof shows that for $f \colon J \to K$, the Goodwillie transform takes the left Kan extension functor $f_! \colon \Fun(J, \sC) \to \Fun(K,\sC)$ to the left Kan extension functor $f_! \colon \Fun(J, \Alg_{\partial_*{\id_\sC}}) \to \Fun(K, \Alg_{\partial_*{\id_\sC}})$.
    Under the further assumption that $J_{k/}$ is finite for every $k$ in $K$, the analogous statement for right Kan extensions follows as well.
\end{remark}

\begin{corollary} \label{cor: Goodwillie-transf-preserves-exact-functors}
    Let $F \colon \sC \to \sD$ be a reduced and finitary functor of differentiable categories and suppose that $K$ is a small category such that $F$ preserves $K$-indexed colimits.
    Then
    \[
    \maAlg(F) \colon \Alg_{\partial_*{\id_{\sC}}} \to \Alg_{\partial_*{\id_{\sD}}}
    \]
    also preserves $K$-indexed colimits.
    The analogous statement for $K$-indexed limits holds if $K$ is finite.
\end{corollary}
\begin{proof}
    Consider the commutative square
    \[
    \begin{tikzcd}
        \Fun(K, \sC) \ar[r, "F \circ -"] & \Fun(K, \sD) \\
        \sC \ar[u, "\Delta"] \ar[r, "F"] & \sD \ar[u, "\Delta"']
    \end{tikzcd}
    \]
    where the vertical arrows are given by the diagonal functors.
    The left adjoints of these functors are the colimits functors, and the corresponding Beck--Chevalley transformation is the assembly map
    \[
    \colim_K \circ F \Rightarrow F \circ \colim_K,
    \]
    so that $F$ preserves $K$-indexed colimits if and only if this Beck--Chevalley transformation is an equivalence.
    Similarly, the functor $F$ preserves $K$-indexed limits if and only if the Beck--Chevalley transformation corresponding to the right adjoints of the diagonal functors is an equivalence.
    
    Since the Goodwillie transform preserves cotensors, it preserves the above square.
    The result for colimits now follows immediately from the fact that every functor of $2$-categories preserves Beck--Chevalley transformations.
    The result for limits follows in the same way, where, as in the previous proof, we take care to note that the limit functor in general only defines a morphism in $\diff$ if $K$ is finite.
\end{proof}

We now turn to the proof of \cref{thm: derivatives-preserves-cotensors}.
We will need the following lemma.

\begin{lemma}
    Let $\sA$ be a stable presentable category and $K$ a small category.
    Then $\sA$ admits a cotensor with $K$ in $\pressymst$, given by the functor category $\Fun(K, \sA)$.
\end{lemma}
\begin{proof}
    This follows from the natural equivalence
    \begin{align*}
    \SSeq(\sB, \Fun(K, \sA)) &= \FunL(\Sym\sB, \Fun(K, \sA)) \\ 
    &\simeq \Fun(K, \FunL(\Sym\sB, \sA)) = \Fun(K, \SSeq(\sB, \sA)).
    \end{align*}
    Here we use that $\Fun(K, \sA)$ is the cotensor of $\sA$ with $K$ in $\presl$.
\end{proof}

\begin{remark}\label{remark:cotensors-pressymst}
    The cotensoring from the previous lemma is exhibited by the functor 
    \[
    \ev \colon K \to \SSeq(\sA^K, \sA)
    \]
    that sends $k \in K$ to the evaluation functor $\ev_k \colon \sA^K \to \sA$ considered as a symmetric sequence concentrated in arity $1$.
\end{remark}

\begin{proof}[Proof of \cref{thm: derivatives-preserves-cotensors}]
    Let $\sC$ be a differentiable category and let $K$ be a small category.
    Observe that there is an equivalence $\Sp(\sC^K) \simeq (\Sp\sC)^K$.
    We need to show that the composite
    \[
    K \xrightarrow{\ev} \Fun(\sC^K, \sC) \xrightarrow{\partial_*} \BMod_{(\partial_*{\id_\sC}, \partial_*{\id_{\sC^K}})}(\SSeq((\Sp\sC)^K, \Sp\sC))
    \]
    exhibits the pair $((\Sp\sC)^K, \partial_*{\id_{\sC^K}})$ as the cotensor of $(\Sp\sC, \partial_*{\id_\sC})$ with $K$ in the Morita category.
    We can consider this composite as an object
    \[
    \partial_*{\ev} \in \BMod_{(\partial_*{\id_\sC}, \partial_*{\id_{\sC^K}})}(\Fun(K, \SSeq((\Sp\sC)^K, \Sp\sC))),
    \]
    where $\Fun(K, \SSeq((\Sp\sC)^K, \Sp\sC))$ has the pointwise left and right tensoring by the monoidal categories $\SSeq((\Sp\sC)^K, (\Sp\sC)^K)$ and $\SSeq(\Sp\sC, \Sp\sC)$ respectively.
    We claim that $\partial_*{\ev}$ is free as a left $\partial_*{\id_{\sC}}$-module and free as a right $\partial_*{\id_{\sC^K}}$-module.
    (Note that this does not imply that it is free as a bimodule.)
    To see this, first note that $\partial_1{\ev}$ is the evaluation functor $K \to \SSeq((\Sp\sC)^K, \Sp\sC)$, which is concentrated in arity $1$.
    The inclusion $\partial_1{\ev} \to \partial_*{\ev}$ induces a map of left $\partial_*{\id_{\sC}}$-modules
    \[
    \partial_*{\id_{\sC}} \circ \partial_1(\ev) \to \partial_*(\ev).
    \]
    This is an equivalence if and only if it is one after evaluation in each $k \in K$.
    But for fixed $k \in K$, the evaluation functor $\ev_k \colon \sC^K \to \sC$ preserves colimits, so that the map of symmetric sequences $\partial_1{\ev_k} \to \partial_*{\ev_k}$ induces an equivalence of left $\partial_*{\id_\sC}$-modules 
    \[
    \partial_*{\id_{\sC}} \circ \partial_1{\ev_k} \simeq \partial_*{\ev_k}
    \]
    by \cref{prop: derivatives-free-left-right-module}.
    The statement about the right module structure is proved in the same way, using that $\ev_k$ preserves limits.

    Now let $(\sA, \sO) \in \MMor$.
    In other words, $\sA$ is a stable presentable category and $\sO \in \Alg(\SSeq(\sA, \sA))$. 
    To complete the proof, it suffices to show that the composite
    \[
    \begin{tikzcd}
        \BMod_{(\partial_*{\id_{\sC^K}}, \sO)}(\SSeq(\sA, (\Sp\sC)^K)) \ar[d] \\
        \Fun\left(\BMod_{(\partial_*{\id_\sC}, \partial_*{\id_{\sC^K}})}(\SSeq((\Sp\sC)^K, \Sp\sC)), \BMod_{(\partial_*{\id_\sC}, \sO)}(\SSeq(\sA, \Sp\sC))\right) \ar[d] \\
        \Fun(K, \BMod_{(\partial_*{\id_\sC}, \sO)}(\SSeq(\sA, \Sp\sC)))
    \end{tikzcd}
    \]
    is an equivalence of categories.
    It sends a $(\partial_*{\id_{\sC^K}}, \sO)$-bimodule $M$ to a bimodule with underlying object
    \[
    \partial_*{\ev} \circ_{\partial_*{\id_{\sC^K}}} M \simeq \partial_1{\ev} \circ \partial_*{\id_{\sC^K}} \circ_{\partial_*{\id_{\sC^K}}} M \simeq \partial_1{\ev} \circ M,
    \]
    where we use that $\partial_*{\ev}$ is free as a right $\partial_*{\id_{\sC^K}}$-module.
    In other words, we have a commutative square
    \[
    \begin{tikzcd}
        \BMod_{(\partial_*{\id_{\sC^K}}, \sO)}(\SSeq(\sA, (\Sp\sC)^K)) \ar[r] \ar[d, "\forget"] & \Fun(K, \BMod_{(\partial_*{\id_\sC}, \sO)}(\SSeq(\sA, \Sp\sC))) \ar[d, "\forget"] \\
        \SSeq(\sA, (\Sp\sC)^K) \ar[r] & \Fun(K, \SSeq(\sA, \Sp\sC)), 
    \end{tikzcd}
    \]
    where the lower horizontal map is an equivalence since $\partial_1{\ev}$ exhibits $(\Sp\sC)^K$ as the cotensor of $\Sp\sC$ with $K$ in $\pressymst$ by \cref{remark:cotensors-pressymst}.
    Both vertical maps are monadic right adjoints, so in order to show that the top horizontal functor is an equivalence, it suffices to show that it preserves free objects by \cite[Corollary 4.7.3.16]{HA}.
    But this functor sends the free bimodule
    \[
    \partial_*{\id_{\sC^K}} \circ X \circ \sO
    \]
    to
    \[
    \partial_*{\ev} \circ_{\partial_*\id_{\sC^K}} \partial_*{\id_{\sC^K}} \circ X \circ \sO \simeq \partial_*{\ev} \circ X \circ \sO,
    \]
    which is the free bimodule on $\partial_1{\ev} \circ X$, since $\partial_*{\ev}$ is the free left module on $\partial_1{\ev}$.
    This completes the proof.
\end{proof}

\subsection{Pullbacks} \label{ssec: pullbacks-mmor}

The goal of this section is to prove that the derivatives functor preserves certain pullbacks of categories.
The precise statement is as follows.

\begin{proposition} \label{prop: pullbacks-preserved-by-derivatives}
    Let
    \[
    \begin{tikzcd}
        \sA \ar[r, "\bar{p}^*", hook] \ar[d, "\bar{q}^*"']  \ar[dr, phantom, "\lrcorner", very near start] & \sC \ar[d, "q^*"] \\
        \sB \ar[r, "p^*"', hook] & \sD
    \end{tikzcd}
    \]
    be a pullback square of categories where $\sB$, $\sC$, and $\sD$ are differentiable.
    Assume that
    \begin{enumerate}[\upshape{(}\roman*\upshape{)}]
        \item $p^*$ and $q^*$ are reduced and finitary,
        \item $p^*$ and $q^*$ both admit a left adjoint, and
        \item $p^*$ is fully faithful.
    \end{enumerate}
    Then $\sA$ is differentiable, $\bar{p}^*$ and $\bar{q}^*$ are reduced and finitary, and the square defines a pullback square in $\diff$ which is preserved by $\partial_* \colon \diff \to \MMor$.
\end{proposition}

Before proving this proposition, we give a couple of easy consequences.

\begin{corollary}
    The Goodwillie transform $\maAlg \colon \diff \to \diff$ preserves the type of pullback squares described in the previous proposition.
\end{corollary}
\begin{proof}
    This immediately follows from the definition of the Goodwillie transform as $\maAlg = \reAlg \circ \partial_*$ and the fact that $\reAlg$ is a corepresentable functor and hence preserves all pullbacks.
\end{proof}

\begin{corollary}
    Let $G \colon \sD \to \sC$ be a conservative finitary right adjoint between differentiable categories.
    Then $\maAlg(G)$ is conservative as well.
\end{corollary}
\begin{proof}
    The functor $G$ being conservative is equivalent to the square
    \[
    \begin{tikzcd}
        \sD \ar[r, hook, "\mathrm{cst}"] \ar[d, "G"'] & \Ar(\sD) \ar[d, "G"] \\
        \sC \ar[r, "\mathrm{cst}"', hook] & \Ar(\sC)
    \end{tikzcd}
    \]
    being a pullback.
    Note that this square satisfies all the conditions of \cref{prop: pullbacks-preserved-by-derivatives}, so that it is sent by $\maAlg$ to a pullback square.
    Since $\maAlg$ preserves cotensors by \cref{cor: goodwillie-transform-preserves-cotensors}, the square is sent to the analogous square with $\sC$ and $\sD$ replaced by $\maAlg(\sC)$ and $\maAlg(\sD)$, and $G$ replaced by $\maAlg(G)$.
    This implies $\maAlg(G)$ is conservative.
\end{proof}

\begin{corollary} \label{cor: monadic-adjunction-goodwillie-transf}
    Let $F \colon \sC \rightleftarrows \sD \colon G$ be an adjunction between differentiable categories such that $G$ is finitary and conservative.
    Then the adjunction
    \[
    \maAlg(F) \colon \Alg_{\partial_*{\id_\sC}} \rightleftarrows \Alg_{\partial_*{\id_\sD}} \colon \maAlg(G)
    \]
    is monadic.
\end{corollary}
\begin{proof}
    By the previous corollary, $\maAlg(G)$ is conservative.
    Also note that $\maAlg(G)$ preserves sifted colimits, since this holds for any functor in the essential image of $\maAlg$ by \cref{prop: goodwillie-transf-reduced-sifted-colim-pres}.
    Therefore, $\maAlg(F) \dashv \maAlg(G)$ is a monadic adjunction by the Barr--Beck--Lurie theorem \cite[Theorem 4.7.3.5]{HA}.
\end{proof}

\begin{remark}
    This last corollary in particular implies that any monadic adjunction with finitary right adjoint is sent to a monadic adjunction by $\maAlg$.
    This is a generalization of \cite[Corollary 6.2.2.16]{HA}, where the analogous result for first derivatives is proved.
\end{remark}

We now turn to the proof of \cref{prop: pullbacks-preserved-by-derivatives}.
We will deduce it from the following more general statement.

\begin{proposition} \label{prop: 2-cat-pullback-pres-general}
    Let
    \begin{equation}\label{eq: pullback-square-cats}
    \begin{tikzcd}
        \sA \ar[r, "\bar{p}^*", hook] \ar[d, "\bar{q}^*"']  \ar[dr, phantom, "\lrcorner", very near start] & \sC \ar[d, "q^*"] \\
        \sB \ar[r, "p^*"', hook] & \sD
    \end{tikzcd}
    \end{equation}
    be a pullback square of categories that admit pushouts and sequential colimits, such that
    \begin{enumerate}[\upshape{(}\roman*\upshape{)}]
        \item $p^*$ and $q^*$ preserve sequential colimits,
        \item $p^*$ and $q^*$ both admit a left adjoint, and
        \item $p^*$ is fully faithful.
    \end{enumerate}
    Let $F \colon \Cat \to \Cat$ be a functor of $2$-categories that locally preserves pushouts and sequential colimits, and assume that $F(\sA)$, $F(\sB)$, $F(\sC)$ and $F(\sD)$ again admit pushouts and sequential colimits and that $F(p^*)$ and $F(q^*)$ again preserve sequential colimits.
    Then $F$ preserves the pullback square.
\end{proposition}
\begin{proof}
    Write $p_!$ and $q_!$ for the left adjoints of $p^*$ and $q^*$ respectively.
    Since $p^*$ is fully faithful, the same holds for $\bar{p}^*$, and its essential image consists of those $x \in \sC$ such that $q^*(x)$ lies in the image of $p^*$.
    Let $x \in \sC$ and consider the object $f(x)$ defined by the pushout square
    \[
    \begin{tikzcd}
        q_! q^*(x) \ar[r, "\epsilon"] \ar[d, "q_! \nu"'] \ar[dr, phantom, "\ulcorner", very near end] & x \ar[d] \\
        q_! p^* p_! q^*(x) \ar[r] & f(x),
    \end{tikzcd}
    \]
    where $\epsilon$ is the counit of the adjunction $q_! \dashv q^*$ and $\nu$ is the unit of the adjunction $p_! \dashv p^*$.
    Define 
    \[
    f^\infty(x) \coloneqq \colim(x \to f(x) \to f^2(x) \to \cdots).
    \]
    We claim that $f^\infty(x)$ lies in the image of $\bar{p}^*$, or in other words, that $q^*f^\infty(x)$ lies in the image of $p^*$.
    Consider the commutative diagram
    \[
    \begin{tikzcd}
        q^*(x) \ar[r, "\eta"] \ar[d, "\nu"'] & q^* q_! q^*(x) \ar[r, "q^* \epsilon"] \ar[d, "q^* q_! \nu"'] & q^*(x) \ar[d] \\
        p^* p_!q^*(x) \ar[r, "\eta"'] & q^* q_! p^* p_! q^*(x) \ar[r] & q^*f(x),
    \end{tikzcd}
    \]
    where the right-hand square is obtained by applying $q^*$ to the pushout square defining $f(x)$ and $\eta$ is the unit of the adjunction $q_! \dashv q^*$.
    The top horizontal composite is equivalent to the identity by the triangle identities, so that the canonical map $q^*(x) \to q^*f(x)$ factors through $p^*p_!q^*(x)$.
    This leads to the factorization
    \[
    \begin{tikzcd}[sep = small]
        q^*(x) \ar[rr] \ar[dr] & & q^*f(x) \ar[rr] \ar[dr] && q^*f^2(x) \ar[rr] \ar[dr] && \cdots \\
        & p^*p_!q^*(x) \ar[ur] & & p^*p_!q^*f(x) \ar[ur] & & \cdots
    \end{tikzcd}
    \]
    A standard cofinality argument together with the fact that $q^*$ preserves sequential colimits now yields
    \[
    q^*f^\infty(x) \simeq \colim_n q^*f^n(x) \simeq \colim_n p^*p_!q^*f^n(x),
    \]
    which lies in the essential image of $p^*$, as this is closed under sequential colimits.

    Now suppose that $y \in \sA$.
    We claim that the canonical map $x \to f^\infty(x)$ induces an equivalence
    \[
    \Map_\sC(f^\infty(x), \bar{p}^*(y)) \xrightarrow{\sim} \Map_\sC(x, \bar{p}^*(y)).
    \]
    Since we have an equivalence
    \[
    \Map_\sC(f^\infty(x), \bar{p}^*(y)) \simeq \lim_n \Map_\sC(f^n(x), \bar{p}^*(y)),
    \]
    it suffices to show that for any $z \in \sC$, the map $z \to f(z)$ is sent to an equivalence by $\Map(-, \bar{p}^*(y))$.
    By definition of $f(z)$, we have a pullback square
    \[
    \begin{tikzcd}
        \Map_\sC(f(z), \bar{p}^*(y)) \ar[r] \ar[d] \ar[dr, phantom, "\lrcorner", very near start] & \Map_\sC(z, \bar{p}^*(y)) \ar[d] \\
        \Map_\sC(q_! p^* p_! q^*(z), \bar{p}^*(y)) \ar[r, "- \circ q_!\nu"] & \Map_\sC(q_!q^*(z), \bar{p}^*(y))
    \end{tikzcd}
    \]
    By adjunction and the fact that $q^* \bar{p}^* \simeq p^* \bar{q}^*$, the bottom arrow is equivalent to
    \[
    \Map_\sD(p^*p_!q^*(z), p^* \bar{q}^*(y)) \xrightarrow{- \circ \nu} \Map(q^*(z), p^* \bar{q}^*(y)),
    \]
    which is an equivalence since $p_! \dashv p^*$ is a reflective localization.
    It follows that the top arrow in the pullback square is an equivalence as well, which proves the claim.

    The upshot is that $\bar{p}^*$ admits a left adjoint $\bar{p}_!$ with a natural equivalence
    \begin{equation} \label{eq: right-left-adjoint-filtered-colim}
    \bar{p}^*\bar{p}_! \simeq \colim ( \id_\sC \rightarrow f \rightarrow f^2 \rightarrow \cdots),
    \end{equation}
    such that the unit is given by the canonical map from $\id_\sC$ to this colimit and $f$ is defined as
    \begin{equation}\label{eq:f-is-pushout-unit-counit}
        f \simeq \colim ( q_! p^* p_! q^* \xleftarrow{q_!\nu q^*} q_! q^* \xrightarrow{\phantom{q_!} \epsilon \phantom{q}} \id_\sC).
    \end{equation}

    Now apply $F$ to the pullback square \cref{eq: pullback-square-cats} to obtain
    \[
    \begin{tikzcd}[sep = large]
        F(\sA) \ar[r, "F(\bar{p}^*)", hook] \ar[d, "F(\bar{q}^*)"'] & F(\sC) \ar[d, "F(q^*)"] \\
        F(\sB) \ar[r, "F(p^*)"', hook] & F(\sD).
    \end{tikzcd}
    \]
    This square again satisfies the three properties from the statement of the proposition: property (i) holds by assumption, and properties (ii) and (iii) of the functors $p^*$ and $q^*$ are preserved by any functor of $2$-categories, since these preserve adjunctions and reflective localizations.
    For the same reason, $F(\bar{p}^*)$ is a fully faithful functor with left adjoint $F(\bar{p}_!)$.
    Write $\sE$ for the pullback of the functors $F(p^*)$ and $F(q^*)$.
    The square above induces a commutative triangle
    \[
    \begin{tikzcd}
        F(\sA) \ar[rr, "F(\bar{p}^*)", hook] \ar[dr, "i^*"', hook] & & F(\sC). \\
        & \sE \ar[ur, "j^*"', hook]
    \end{tikzcd}
    \]
    Note that $j^*$ is fully faithful since it is the pullback of the fully faithful functor $F(p^*)$ and $i^*$ is fully faithful by 2-out-of-3.
    We want to show that $i^*$ is essentially surjective.
    
    Since $F(p^*)$ and $F(q^*)$ satisfy properties (i), (ii), and (iii) from the statement of the proposition, repeating the argument above shows that $j^*$ has a left adjoint $j_!$, such that the unit $\id \to j^* j_!$ is given by the canonical map
    $\id \to \colim_n g^n$
    where $g$ is defined as the pushout
    \[
    \colim ( F( q_! p^* p_! q^*) \xleftarrow{F(q_!\nu q^*)} F(q_! q^*) \xrightarrow{\phantom{q_!} F(\epsilon) \phantom{q}} \id).
    \]
    The fact that $F$ locally preserves pushouts and sequential colimits implies that $\colim_n g^n \simeq F(\colim_n f^n)$, so that the adjunctions $F(\bar{p}_!) \dashv F(\bar{p}^*)$ and $j_! \dashv j^*$ have the same unit. 
    From this it follows immediately that $i^*$ is essentially surjective.
\end{proof}

\begin{remark} \label{rem: change-domain-pullback-preservation}
    The proof we just gave shows that the proposition still holds if we replace the source of $F \colon \Cat \to \Cat$ by a locally full subcategory $\sE \subseteq \Cat$, as long as the inclusion preserves the pullback square in question and locally preserves pushouts and sequential colimits.
\end{remark}

\begin{proof}[Proof of \cref{prop: pullbacks-preserved-by-derivatives}]
    Since $p^*$ and $q^*$ preserve limits and filtered colimits it follows that $\sA$ is presentable, pointed, and $\bar{p}^*$ and $\bar{q}^*$ preserve and jointly reflect limits and filtered colimits.
    This last point implies that filtered colimits and finite limits commute in $\sA$, since this holds in $\sB$, $\sC$, and $\sD$.
    Hence $\sA$ is differentiable and $\bar{p}^*$ and $\bar{q}^*$ are reduced and finitary.
    Moreover, the square defines a pullback square in the $2$-category $\diff$.

    In order to prove that $\partial_*$ preserves the pullback square, it suffices to show this holds after postcomposition with any corepresentable functor.
    So let $x \in \MMor$ and write
    \[
    F = \MMor(x, -) \circ \partial_* \colon \diff \to \Cat.
    \]
    Note that $F$ satisfies the assumptions from \cref{prop: 2-cat-pullback-pres-general}:
    \begin{itemize}
        \item[(a)] Every functor in the image of $F$ preserves filtered colimits, as follows from \cref{prop: exactness-composition-mor};
        \item[(b)] $F$ locally preserves small colimits: this holds since it is true for the functors $\partial_*$ and $\MMor(x, -)$ separately by \cref{prop: local-exactness-derivatives,prop: exactness-composition-mor}.
    \end{itemize}
    Applying the previous proposition and remark to $F$, the result immediately follows.
\end{proof}

\section{The product rule}\label{sec:prod-rule}

The goal of this section is to establish a \emph{product rule} for the Goodwillie derivatives of a functor $F \colon \cC \to \cD$, which is compatible with the $(\partial_* \id_\cD, \partial_* \id_\cC)$-bimodule structure on $\partial_* F$ obtained through the chain rule (see \cref{remark:bimodule-structure-on-partialF}).
To describe it, let us for simplicity assume that $\Sp \cC = \Sp$ and $\cD = \Sp$.\footnote{The actual results are proved in the more general situation where $\Sp \cC = \cD = \LSp$ is a stable Bousfield localization of $\Sp$ in the sense of \cref{ssec:localizations-of-spectra}.}
We then obtain a Day convolution symmetric monoidal structure  $\circledast$ on $\SSeq(\Sp \cC,\Sp) \simeq \Fun(\Fin^\simeq,\Sp)$.
The Day convolution $\circledast$ lifts to a symmetric monoidal structure on $\RMod_{\partial_* {\id_\cC}}(\SSeq(\Sp,\Sp))$ such that both the free and forgetful functors
\[\mathrm{free} : (\SSeq(\Sp,\Sp),\circledast) \rightleftarrows {(\RMod_{\partial_* {\id_\cC}}(\SSeq(\Sp,\Sp)), \circledast)} : \mathrm{fgt}\]
are strong symmetric monoidal.

The main result of this section is then as follows:

\begin{theorem}\label{thm:product-rule-to-spectra}
    Let $\cC$ be a differentiable category such that $\Sp \cC \simeq \Sp$.
    Then the Goodwillie derivatives functor
    \[\partial_* \colon \Fun^{\omega}(\sC, \Sp) \to \RMod_{\partial_* {\id_\cC}}(\SSeq(\Sp,\Sp)) \]
    is strong symmetric monoidal with respect to the pointwise tensor product on $\Fun^{\omega}(\sC, \Sp)$ and the Day convolution $\circledast$ on $\RMod_{\partial_* {\id_\cC}}(\SSeq(\Sp,\Sp))$.
\end{theorem}

This is proved as \cref{thm:convolution-product-rule-unital} below.

\begin{remark}
    In this theorem, we work with the category $\Fun^{\omega}(\sC,\Sp)$ of \emph{all} finitary functors, not just the reduced ones, with the convention that $\partial_0 F = F(*)$ (cf.\ \cref{ssec:unreduced-functors}).
    One reason to consider unreduced functors is that the unit $\const_\Sph$ of $\Fun^{\omega}(\sC,\Sp)$ is not reduced, hence $\Fun^{*,\omega}(\sC,\Sp)$ is merely a \emph{nonunital} symmetric monoidal category.
    We describe how to refine $\partial_* \colon \Fun^{\ast,\omega}(\sC,\Sp) \to \RMod_{\partial_*{\id_\cC}}(\SSeq(\Sp, \Sp))$ to a functor $\Fun^{\omega}(\sC,\Sp) \to \RMod_{\partial_*{\id_\cC}}(\SSeq(\Sp, \Sp))$ in \cref{ssec:unital-prod-rule} below.
\end{remark}

The strategy of the proof is surprisingly similar to a classical calculus argument that derives the product rule from the chain rule: writing $\mu$ for the multiplication map $\bbR \times \bbR \to \bbR$, we see that
\[(f \cdot g)'(x) = \frac{\partial \mu}{\partial f} \frac{\mathrm{d}f}{\mathrm{d}x} + \frac{\partial \mu}{\partial g} \frac{\mathrm{d}g}{\mathrm{d}x} = g(x)f'(x) + f(x)g'(x)\]
by the multivariable chain rule.
In other words, one obtains the product rule from the chain rule by deriving the multiplicative structure on $\bbR$.
We will similarly prove \cref{thm:product-rule-to-spectra} by deriving the tensor product of $\Sp$.

In fact, this strategy not only works for $\Sp$, but for any (nonunital) commutative monoid in $\Diff$.
More precisely, we obtain the following (proved as \cref{thm:General-prod-rule} below):

\begin{theorem}\label{thm:General-prod-rule-intro}
    Let $\cC$ be a differentiable category and $\cD$ a (nonunital) commutative monoid in $\Diff$.
    Then there exists a (nonunital) symmetric monoidal structure on $\BMod_{(\partial_* {\id_{\cD}}, \partial_* {\id_\cC})}$ such that the derivatives functor
    \[\partial_* \colon \Fun^{\ast,\omega}(\cC,\cD) \to \BMod_{(\partial_* {\id_{\cD}}, \partial_* {\id_\cC})}\]
    is (nonunitally) strong symmetric monoidal with respect to the pointwise symmetric monoidal structure on $\Fun^{\ast,\omega}(\cC,\cD)$.
\end{theorem}
For example, $\Spc_*$ with its smash product defines a nonunital commutative monoid in $\Diff$, hence the functor
\[\partial_* \colon \Fun^{\ast, \omega}(\sC, \Spc_*) \to \BMod_{(\partial_* {\id_{\Spc_*}}, \partial_* {\id_\cC})}(\SSeq(\Sp,\Sp)) \]
refines to a nonunital strong symmetric monoidal functor with respect to the pointwise smash product on $\Fun^{\ast, \omega}(\cC,\Spc_*)$.
However, the (nonunital) symmetric monoidal structure that we obtain on the target $\BMod_{(\partial_* {\id_{\Spc_*}}, \partial_* {\id_\cC})}$ does not admit a straightforward description in terms of Day convolution.
In particular, it is not simply a lift of the Day convolution monoidal structure on $\SSeq(\Sp,\Sp)$ to $\BMod_{(\partial_* {\id_{\Spc_*}}, \partial_* {\id_\cC})}$ and it fundamentally depends on extra structure of the operad $\partial_* \id_{\Spc_*}$.
Moreover, unlike \cref{thm:product-rule-to-spectra} there is no obvious way to extend this to a \emph{unital} strong symmetric monoidal functor, since there is no analogue of $\partial_0 F$ for functors $F \colon \cC \to \Spc_*$.

A brief outline of this section is as follows.
We start by discussing monoids in general 2-categories and the ``pointwise'' monoidal structure they induce on mapping categories.
From this we directly deduce the general product rule, \cref{thm:General-prod-rule-intro}.
The resulting symmetric monoidal structure on $\BMod_{(\partial_* {\id_{\Spc_*}}, \partial_* {\id_\cC})}$ is generally quite hard to describe, so in \cref{ssec:conv-on-right-modules,ssec:comparing-convolution-prod-rule} we relate it to Day convolution and prove the nonunital version of \cref{thm:product-rule-to-spectra}.
The main idea is to show that the Day convolution symmetric monoidal structure on $\SSeq(\Sp)$ is, in a certain precise sense, uniquely determined by its multiplication functor (see \cref{prop:uniqueness-monoid-LSp}).
We then describe how to upgrade the nonunital version of \cref{thm:product-rule-to-spectra} to the unital version in \cref{ssec:unital-prod-rule}.
We conclude with \cref{ssec:coend-operad-suspension}, which uses these results to relate the operad $\partial_* \id_\cC$ to the coendomorphism operad of $\Sigma^\infty_\cC$ in $\Fun^\omega(\cC,\Sp)$, proving \cref{introcor: product-rule-koszul-duality}.

\subsection{Monoids and pointwise tensor products}

Throughout this subsection, we will write $\cE^\otimes \to \Fin_*$ to either denote the $E_n$-operad or the nonunital $E_n$-operad (for some $1 \leq n \leq \infty$).
The results of this subsection and the next one have analogues for arbitrary $\infty$-operads $\cO^\otimes \to \Fin_*$ in the sense of \cite[\S 2.1.1]{HA}, but for ease of exposition we will focus only on these cases.

\begin{definition}
    Let $\sX$ be a 2-category.
    An \emph{$\cE$-monoid} in $\sX$ is defined as a functor $F \colon \cE^\otimes \to \sX$ such that for every $n$, the Segal map
    \[F(n) \to F(1) \times \cdots \times F(1)\]
    is an equivalence (cf.\ \cite[Definition 2.4.2.1]{HA}).
    We write $\Mon_\cE(\sX)$ for the full sub-2-category of $\Funtwo(\cE^\otimes,\sX)$ spanned by the $\cE$-monoids.
    If $\cE$ is the (nonunital) $E_\infty$-operad, then we call an $\cE$-monoid a (nonunital) commutative monoid, and we write $\CMon(\sX)$ and $\CMon^\mathrm{nu}(\sX)$ for their categories.
\end{definition}

\begin{example}
    An $E_n$-monoid in $\Cat$ is an $E_n$-monoidal category in the sense of \cite[Definition 2.1.2.13]{HA}.
\end{example}

\begin{example}\label{exa:differentable-P-monoids}
    We will call an $\cE$-monoid in $\Diff$ a \emph{differentiable $\cE$-monoidal category}.
    Concretely, these are $\cE$-monoidal categories $\cC$ such that
    \begin{enumerate}
        \item the underlying category $\cC$ is differentiable,
        \item the tensor product $\otimes \colon \cC \times \cC \to \cC$ is reduced and finitary, and
        \item (in the unital case) the unit of $\cC$ is the zero object.
    \end{enumerate}
\end{example}

\begin{example}
    By the previous example, it follows that if $\cC$ is a presentably $E_n$-monoidal category whose underlying category is differentiable, then $\cC$ is a differentiable \emph{nonunital} $E_n$-monoidal category.
    However, note that $\cC$ is almost never a unital $E_n$-monoid in $\Diff$. Namely, this would imply that the unit of $\cC$ is the zero object, hence that $c \simeq c \otimes \varnothing \simeq \varnothing$ for any $c$ in $\cC$. This can only happen if $\cC = *$.
\end{example}

\begin{lemma}\label{lem:pointwise-tensor-prod}
    Let $\sX$ be a 2-category and $A$ an $\cE$-monoid in $\sX$.
    Then the representable functor $\sX(-,A) \colon \sX^\op \to \Cat$ refines to a functor $\sX(-,A) \colon \sX^\op \to \Mon_\cE(\Cat)$ such that
    \[\begin{tikzcd}[ampersand replacement=\&]
        \& {\Mon_\cE(\Cat)} \\
        {\sX^\op} \& \Cat
        \arrow[from=1-2, to=2-2]
        \arrow[from=2-1, to=1-2, "{\sX(-,A)}"]
        \arrow[from=2-1, to=2-2, "{\sX(-,A)}"']
    \end{tikzcd}\]
    commutes.
\end{lemma}

\begin{proof}
    The hom-functor $\sX(-,-) \colon \sX^\op \times \sX \to \Cat$ is the 2-functor obtained by straightening the Segal copresheaf $\Tw(\sX) \to \sX^\op \times \sX$ from \cite[Example 4.36]{Nuiten2023StraighteningSegalSpaces}.
    Now form the composite
    \[\sX^\op \times \Funtwo(\cE^\otimes,\sX) \xrightarrow{\mathrm{\mathrm{const}} \times \id} \Funtwo(\cE^\otimes,\sX^\op \times \sX) \xrightarrow{\sX(-,-)} \Funtwo(\cE^\otimes,\Cat).\]
    It follows by unwinding the functoriality of $\sX(-,-)$ in the second variable (cf.\ \cite[Proposition A.2.31]{blansblom2025chainrulegoodwilliecalculus}) and the definition of cartesian products in $2$-categories (see e.g.\ \cite[Definition A.2.19]{blansblom2025chainrulegoodwilliecalculus}) that $\sX(-,-)$ preserves cartesian products in the second variable.
    The composite defined above thus restricts to a functor $\sX^\op \times \Mon_\cE(\sX) \to \Mon_\cE(\Cat)$.
    The result follows by inserting $A$ in the second variable.
\end{proof}

\begin{definition}\label{def:pointwise-tensor}
    In the situation of \cref{lem:pointwise-tensor-prod}, we refer to the $\cE$-monoidal structure on $\sX(x,A)$ as the \emph{pointwise $\cE$-monoidal structure}.
\end{definition}

The pointwise monoidal structure on $\sX(x,A)$ satisfies the following naturality property:

\begin{lemma}\label{lem:naturality-pointwise-tensor}
    Let $F \colon \sX \to \sY$ be a functor between 2-categories that preserves finite cartesian products and let $A$ be an $\cE$-monoid in $\sX$.
    Then the induced functor
    \[F \colon \sX(x,A) \to \sY(F(x), F(A))\]
    is strong $\cE$-monoidal with respect to the pointwise $\cE$-monoidal structures, naturally in $x$.
\end{lemma}

\begin{proof}
    Since the projection $\Tw(\sX) \to \sX^\op \times \sX$ in \cite[Example 4.36]{Nuiten2023StraighteningSegalSpaces} is given by restriction along $[n]^\op \hookrightarrow [n]^\op \star [n] \hookleftarrow [n]$, it follows that
    \[\begin{tikzcd}[ampersand replacement=\&]
        {\Tw(\sX)} \& {\Tw(\sY)} \\
        {\sX^\op \times \sX} \& {\sY^\op \times \sY}
        \arrow[from=1-1, to=1-2, "{\Tw(F)}"]
        \arrow[from=1-1, to=2-1]
        \arrow[from=1-2, to=2-2]
        \arrow[from=2-1, to=2-2, "F^\op \times F"]
    \end{tikzcd}\]
    commutes.
    It is easily verified that this is a map of Segal copresheaves, hence by \cite[Corollary 4.33]{Nuiten2023StraighteningSegalSpaces} we obtain a natural transformation $\sX(-,-) \Rightarrow \sY(F(-),F(-))$.
   Since $F$ preserves finite cartesian products, both the source and target of this natural transformation preserve cartesian products in the second variable.
   It follows by the construction from \cref{lem:pointwise-tensor-prod} that we obtain a natural transformation
   \[\sX(-,A) \Rightarrow \sY(F(-),F(A))\]
   of functors $\sX^\op \to \Mon_\cE(\Cat)$.
\end{proof}

\subsection{The general product rule}\label{ssec:general-product-rule}

Our general product rule follows directly from the previous observations about $\cE$-monoids.

\begin{theorem}\label{thm:General-prod-rule}
    Let $\cC$ be a differentiable category and $\cD$ a differentiable $\cE$-monoidal category.
    Then there exists an $\cE$-monoidal structure on $\BMod_{(\partial_* {\id_{\cD}}, \partial_* {\id_\cC})}$ such that the derivatives functor
    \[\partial_* \colon \Fun^{\ast,\omega}(\cC,\cD) \to \BMod_{(\partial_* {\id_{\cD}}, \partial_* {\id_\cC})}\]
    is strong $\cE$-monoidal with respect to the pointwise $\cE$-monoidal structure on $\Fun^{\ast,\omega}(\cC,\cD)$.
\end{theorem}

\begin{proof}
    The functor $\partial_* \colon \Diff \to \MMor$ preserves cartesian products by \cref{cor:partial-to-mmor-preserves-prod}.
    In particular, it takes $\cD$ to an $\cE$-monoid $(\Sp\cD, \partial_* \id_\cD)$ in $\MMor$.
    By \cref{lem:pointwise-tensor-prod}, we obtain an $\cE$-monoidal structure on
    \[\BMod_{(\partial_* \id_\cD, \cO)}(\SSeq(\cA,\Sp\cD)) \simeq \MMor((\sA,\cO), (\Sp\cD,\partial_* \id_\cD))\]
    for any stable presentable category $\sA$ and any $\cO \in \Alg(\SSeq(\cA,\cA))$.
    The desired result follows directly from \cref{lem:naturality-pointwise-tensor}.
\end{proof}

We will now unwind what the multiplication map of the resulting $\cE$-monoidal structure on $\BMod_{(\partial_* {\id_{\cD}}, \partial_* {\id_\cC})}$ looks like.
Let $\otimes \colon \cD \times \cD \to \cD$ be the multiplication map of $\cD$.
Applying $\partial_*$, we obtain a $(\partial_* \id_\cD, \partial_*\id_\cD \bartimes \partial_* \id_\cD)$-bimodule structure on $\partial_*(\otimes)$.
The multiplication map of the $\cE$-monoidal structure on $\BMod_{(\partial_* {\id_{\cD}}, \partial_* {\id_\cC})}$ is then given by the relative composition product
\[\begin{tikzcd}[ampersand replacement=\&, row sep = scriptsize]
	{\BMod_{(\partial_* {\id_{\cD}}, \partial_* {\id_\cC})} \times \BMod_{(\partial_* {\id_{\cD}}, \partial_* {\id_\cC})}} \\
	{\BMod_{(\partial_* {\id_{\cD}} \bartimes \partial_*{\id_{\cD}}, \partial_* {\id_\cC})}} \\
	{\BMod_{(\partial_* {\id_{\cD}}, \partial_* {\id_\cC})}.}
	\arrow["\simeq", from=1-1, to=2-1]
	\arrow["{\partial_*(\otimes) \circ_{\partial_* {\id_{\cD}} \bartimes \partial_*{\id_{\cD}}} (-) }", from=2-1, to=3-1]
\end{tikzcd}\]
If we further assume that $\cD$ is stable and that the tensor product preserves colimits in each variable separately, then this construction can be described rather explicitly.
To this end, let us briefly recall what the derivatives of functors in two variables look like.
Given a functor $F \colon \cA \times \cB \to \cE$ with $\cA$, $\cB$ and $\cE$ differentiable, its derivatives form a symmetric sequence
\[\partial_* F \colon \Sym(\Sp\cA \oplus \Sp\cB) \to \Sp\cE.\]
Observe that
\[\Sym(\Sp\cA \oplus \Sp\cB) \simeq \bigoplus_{n \geq 0} (\Sp\cA \oplus \Sp\cB)^{\otimes n}_{h\Sigma_n} \simeq \bigoplus_{m,n \geq 0} (\Sp\cA)^{\otimes m}_{h\Sigma_m} \otimes (\Sp\cB)^{\otimes n}_{h\Sigma_n},\]
(cf.\ \cite[Proposition 3.2.4.7]{HA}).
In other words, its derivatives can be thought of as a \emph{bisymmetric sequence}:
a collection of functors
\[\partial_{m,n} F \colon (\Sp\cA)^{\times m}_{h\Sigma_m} \times (\Sp\cB)^{\times n}_{h\Sigma_n} \to \Sp\cE \qquad \text{for } m,n \geq 0 \]
that preserve colimits in each variable separately.
We will say that $\partial_{m,n} F$ is the derivative of $F$ in \emph{bidegree} $(m,n)$.
Note that the ordinary $n$th derivative of $F$ is recovered by $\partial_n F = \bigoplus_{i+j = n} \Ind_{\Sigma_i \times \Sigma_j}^{\Sigma_n} \partial_{i,j} F$.
The $(m,n)$-th derivative $\partial_{m,n} F$ of $F$ can be computed by the composite
\begin{align*}\Fun^\omega(\cA \times \cB, \cE) &\xrightarrow{\hspace{1em}\mathclap{\simeq}\hspace{1em}} \Fun^\omega(\cA, \Fun^\omega(\cB,\cE)) \\\
&\xrightarrow{\hspace{1em}\mathclap{(\partial_n)_*}\hspace{1em}} \Fun^\omega(\cA,\FunL((\Sp\cB)^{\otimes n}_{h\Sigma_n}, \Sp\cE)) \\
&\xrightarrow{\hspace{1em}\mathclap{\partial_m}\hspace{1em}} \FunL((\Sp\cA)^{\otimes m}_{h\Sigma_m}, \FunL((\Sp\cB)^{\otimes n}_{h\Sigma_n}, \Sp\cE)) \\
&\xrightarrow{\hspace{1em}\mathclap{\simeq}\hspace{1em}} \FunL((\Sp\cA)^{\otimes m}_{h\Sigma_m} \otimes (\Sp\cB)^{\otimes n}_{h\Sigma_n}, \Sp\cE).\end{align*}
This can be deduced from the formula for the cross-effects and the fact that
\[(\partial_{m,n}F)(a_1,\ldots,a_m,b_1,\ldots,b_n) \simeq (\partial_{m+n}F)((a_1,0),\ldots,(a_m,0),(0,b_1),\ldots,(0,b_n)).\]
If $\cD$ is stable and $\otimes \colon \cD \times \cD \to \cD$ preserves colimits in each variable, it thus follows that $\partial_{*,*}(\otimes)$ is concentrated in bidegree $(1,1)$, where it is given by the tensor product $\otimes \colon \cD \times \cD \to \cD$ itself.
Applying this to a stable presentable localization $\LSp$ of $\Sp$, we obtain the following:

\begin{lemma}\label{lem:hessian}
    Let $\LSp$ be a localization of $\Sp$ in the sense of \cref{def:localization-of-spectra}.
    Then the derivatives symmetric sequence $\partial_* \mu_2$ of the multiplication map $\mu_2 \colon \LSp \times \LSp \to \LSp$ is concentrated in arity 2, where it is the projection
    \begin{align*}
    ({\color{cbblue} \LSp} \oplus {\color{cborange} \LSp})^{\otimes 2}_{h\Sigma_2} &\simeq ({\color{cbblue} \LSp} \otimes {\color{cbblue} \LSp} \oplus {\color{cbblue} \LSp} \otimes {\color{cborange} \LSp} \oplus {\color{cborange} \LSp} \otimes {\color{cbblue} \LSp} \oplus {\color{cborange} \LSp} \otimes {\color{cborange} \LSp})_{h\Sigma_2} \\
    &\to ({\color{cbblue} \LSp} \otimes {\color{cborange} \LSp} \oplus {\color{cborange} \LSp} \otimes {\color{cbblue} \LSp})_{h\Sigma_2} \simeq (\LSp \oplus \LSp)_{h\Sigma_2} \simeq \LSp.
    \end{align*}
    Here the colours indicate onto which summands we are projecting. \qed
\end{lemma}

\begin{remark}
    Observe the analogy with the Hessian matrix of the multiplication map $\bbR^2 \to \bbR$:
    \[\begin{pmatrix}
        \dfrac{\partial^2 xy}{\partial x^2} & \dfrac{\partial^2 xy}{\partial x \partial y} \\[1em]
        \dfrac{\partial^2 xy}{\partial y \partial x} & \dfrac{\partial^2 xy}{\partial y^2}
    \end{pmatrix}
    =
    \begin{pmatrix}
        0 & 1 \\
        1 & 0
    \end{pmatrix}.\]
\end{remark}

\subsection{The convolution product on right modules}\label{ssec:conv-on-right-modules}

The goal of the next two sections is to give a more familiar description of the monoidal structure from \cref{thm:General-prod-rule} in terms of Day convolution, at least when the target $\cD$ is a localization $\LSp$ of the category $\Sp$ of spectra in the sense of \cref{def:localization-of-spectra}.
Recall from \cref{ssec:localizations-of-spectra} that for any small symmetric monoidal category $\cC$ and any $\cA \in \CAlg(\presl_\LSp)$, there is a natural equivalence
\[\Fun^{\mathrm{L},\otimes}(\cP_\LSp(\cC), \cA) \simeq \Fun^\otimes(\cC,\cA),\]
where $\cP_\LSp(\cC)$ is endowed with its Day convolution symmetric monoidal structure $\circledast$.
Applying this to $\Sym^\times \cC$, the free symmetric monoidal category on a small category $\cC$, it follows by comparing universal properties that $\cP_{\LSp}(\Sym^\times \cC) \simeq \Sym\cP_\LSp(\cC)$.
In particular, we obtain equivalences
\begin{equation}\label{eq:SSeq-without-L}\begin{split}
    \SSeq(\cP_{\LSp}(\cC), \LSp) &\simeq \FunL(\Sym\cP_\LSp(\cC),\LSp) \\ 
    &\simeq \Fun(\Sym^\times \cC, \LSp) \simeq \cP_\LSp(\Sym^\times \cC^\op).\end{split}\end{equation}
The right-hand side comes equipped with the Day convolution symmetric monoidal structure.
In particular, in the case $\sC = \ast$, this is the Day convolution on $\SSeq(\LSp)$ from \cref{ssec:coend-operads-and-Koszul-duality}.

\begin{proposition}\label{prop:Day-convolution-as-pointwise-tensor}
    There exists a commutative monoid structure on $\LSp$ in $\pressymst$ such that for any small category $\cC$, the pointwise symmetric monoidal structure on
    \[\SSeq(\cP_{\LSp}(\cC),\LSp)\]
    in the sense of \cref{def:pointwise-tensor} agrees with the Day convolution symmetric monoidal structure under the string of equivalences \cref{eq:SSeq-without-L}.
\end{proposition}

\begin{proof}
    Observe that by replacing $\sX$ with $\sX^\op$ in \cref{lem:pointwise-tensor-prod}, it follows that if $Q$ is a commutative coalgebra in the cocartesian symmetric monoidal structure on $\sX^{\amalg}$, then $\sX(Q,x)$ obtains a symmetric monoidal structure for any $x$ in $\sX$.
    Let us call this the \emph{induced symmetric monoidal structure} throughout this proof.
    Also note that the fully faithful inclusion $\pressym_\LSp \hookrightarrow \pressymst$ from \cref{lem:inclusion-pressymlsp-into-pressymst} preserves finite cartesian products by \cref{ex: inclusion-pressymlsp-pressymst-preserves-prod}, so it suffices to exhibit a commutative monoid structure on $\LSp$ in $\pressym_{\LSp}$.
    
    Since $- \otimes \LSp \colon \presl \to \presl_\LSp$ is strong monoidal, it follows that $\cP_{\LSp}(\cC)$ is dualizable in $\presl_{\LSp}$ for any $\cC$ in $\Cat$, with $\cP_{\LSp}(\cC)^\vee \simeq \cP_\LSp(\cC^\op)$.
    Moreover, we have the commutative square
    \[\begin{tikzcd}[ampersand replacement=\&]
        {\coCAlg(\presl_{\LSp,\mathrm{dual}})} \& {\CAlg(\presl_{\LSp,\mathrm{dual}})^\op} \\
        {\presl_{\LSp,\mathrm{dual}}} \& {(\presl_{\LSp,\mathrm{dual}})^\op}
        \arrow["{(-)^\vee}", "\sim"', from=1-1, to=1-2]
        \arrow[""{name=0, anchor=center, inner sep=0}, "\fgt", shift left=2, from=1-1, to=2-1]
        \arrow[""{name=1, anchor=center, inner sep=0}, "\fgt", shift left=2, from=1-2, to=2-2]
        \arrow[""{name=2, anchor=center, inner sep=0}, "\Sym", shift left=2, from=2-1, to=1-1]
        \arrow["{(-)^\vee}", "\sim"', from=2-1, to=2-2]
        \arrow[""{name=3, anchor=center, inner sep=0}, "\Sym", shift left=2, from=2-2, to=1-2]
        \arrow["\dashv"{anchor=center, rotate=180}, draw=none, from=2, to=0]
        \arrow["\dashv"{anchor=center, rotate=180}, draw=none, from=3, to=1]
    \end{tikzcd}\]
    by \cref{lem:Sym-dualizable}.
    It therefore suffices to show that $\Sym \LSp$ has the structure of a coalgebra in $\CAlg(\presl_{\LSp})$ with respect to the cocartesian symmetric monoidal structure, such that the induced symmetric monoidal structure on
    \[\FunL_{\CAlg}(\Sym \LSp, \cP_\LSp(\Sym^\times \cC)^\vee) \simeq \FunL(\LSp, \cP_\LSp(\Sym^\times \cC^\op)) \simeq \cP_\LSp(\Sym^\times \cC^\op) \]
    agrees with the Day convolution symmetric monoidal structure on $\cP_{\LSp}(\Sym^\times \cC^\op)$.
    By the 2-categorical Yoneda lemma \cite{Hinich2020YonedaLemmaEnriched}, this is equivalent to showing that the functor $\FunL_{\CAlg}(\Sym \LSp, -) \colon \CAlg(\presl_{\LSp}) \to \widehat{\Cat}$ can be lifted to a functor $\CAlg(\presl_{\LSp}) \to \CMon(\widehat{\Cat})$ such that it takes $\cP_{\LSp}(\Sym^\times \cC^\op)$ to the Day convolution symmetric monoidal structure on $\FunL_{\CAlg}(\Sym \LSp,\cP_{\LSp}(\Sym^\times \cC^\op)) \simeq \cP_{\LSp}(\Sym^\times \cC^\op)$.
    Since \[\FunL_{\CAlg}(\Sym \LSp, -) \colon \CAlg(\presl_{\LSp}) \to \Cat\] is simply the forgetful functor by \cite[Proposition 2.2.15]{blansblom2025chainrulegoodwilliecalculus}, it follows that such a lift is given by the inclusion $\CAlg(\presl_{\LSp}) \hookrightarrow \CMon(\widehat{\Cat})$.
\end{proof}

We will now upgrade this Day convolution monoidal structure to one on right modules.
We do this by lifting the corresponding monoid structure on $\LSp$ to $\MMor$:

\begin{construction}\label{constr:LSp}
    By \cref{ex: inclusion-pressymst-mor-preserves-prod}, the inclusion $\pressymst \hookrightarrow \MMor$ preserves cartesian products.
    It follows that $\LSp$, with its commutative monoid structure from \cref{prop:Day-convolution-as-pointwise-tensor}, is taken to a commutative monoid $(\LSp,\unit_{\LSp})$ in $\MMor$ by the inclusion $\pressymst \hookrightarrow \MMor$.
\end{construction}

\begin{definition}\label{def:convolution-product-sseq}
    Let $\cO$ be an $\LSp$-enriched operad, i.e. an algebra in $\SSeq(\LSp,\LSp)$.
    Then the \emph{Day convolution symmetric monoidal structure} on $\RMod_\cO(\SSeq(\LSp,\LSp))$ is defined as the pointwise symmetric monoidal structure (see \cref{def:pointwise-tensor}) on
    \[\RMod_\cO(\SSeq(\LSp,\LSp)) \simeq \MMor((\LSp,\cO),(\LSp,\unit_{\LSp}))\]
    with respect to the commutative monoid structure on $(\LSp,\unit_{\LSp})$ in $\MMor$ from \cref{constr:LSp}.
    We will write $\circledast$ for the corresponding tensor product on $\RMod_\cO(\SSeq(\LSp,\LSp))$.
\end{definition}

Let us conclude by observing that this monoidal structure is a lift of the Day convolution monoidal structure on $\SSeq(\LSp,\LSp) \simeq \Fun(\Fin^\simeq,\LSp)$, justifying its name:

\begin{proposition}\label{prop:day-convolution-is-indeed-lift}
    Let $\cO$ be an $\LSp$-enriched operad.
    Then both the free and forgetful functors
    \[\free : \SSeq(\LSp,\LSp) \rightleftarrows \RMod_\cO(\SSeq(\LSp, \LSp)) : \fgt\]
    are strong symmetric monoidal with respect to the Day convolution symmetric monoidal structures.
\end{proposition}

\begin{proof}
    It follows from \cite[Proposition 4.6.2.17]{HA} that under the equivalences 
    \[\SSeq(\LSp,\LSp) \simeq \MMor((\LSp,\unit_\LSp),(\LSp,\unit_\LSp)), \text{ and}\]
    \[\RMod_\cO(\SSeq(\LSp,\LSp)) \simeq \MMor((\LSp,\cO),(\LSp,\unit_\LSp)),\]
    the functors $\free$ and $\fgt$ are given by $- \circ \cO$ and $- \circ_\cO \cO$, respectively.
    The result therefore follows from the functoriality of $\MMor(-, (\LSp, \unit_\LSp))$ in the first variable established in \cref{lem:pointwise-tensor-prod}.
\end{proof}

It follows that for two right $\cO$-modules $A$ and $B$, the underlying symmetric sequence of $A \circledast B$ is given by
\begin{equation}\label{eq:Day-convolution-formula}
    (A \circledast B)_{n} = \bigoplus_{i+j=n} \Ind^{\Sigma_n}_{\Sigma_i \times \Sigma_j} A_i \otimes B_j,
\end{equation}
see also \cref{ssec:coend-operads-and-Koszul-duality}.
Moreover, the right action map of $\cO$ on $A \circledast B$ is given by
\[(A \circledast B) \circ \cO \simeq (A \circ O) \circledast (B \circ O) \xrightarrow{\mathrm{act} \circledast \mathrm{act}} A \circledast B.\]

\subsection{Comparing the convolution and the derived monoidal structure}\label{ssec:comparing-convolution-prod-rule}

Let $\cC$ be a differentiable category such that $\Sp \cC = \LSp$.
In \cref{thm:General-prod-rule}, we constructed a symmetric monoidal structure on $\RMod_{\partial_* \id_\cC}(\SSeq(\LSp ,\LSp))$ such that
\[\partial_* \colon \Fun^{\ast,\omega}(\cC,\LSp) \to \RMod_{\partial_* {\id_\cC}}(\SSeq(\LSp,\LSp))\]
defines a (nonunital) strong symmetric monoidal functor; in other words, a \emph{product rule} for the Goodwillie derivatives of functors to $\LSp$.
In \cref{def:convolution-product-sseq}, we constructed the more explicit Day convolution symmetric monoidal structure on $\RMod_{\partial_* \id_\cC}(\SSeq(\LSp, \LSp))$.
We will now show that these symmetric monoidal structures agree.

Both of these (nonunital) symmetric monoidal structures arose by applying \cref{lem:pointwise-tensor-prod} to a (nonunital) commutative monoid structure on $(\LSp,\unit_{\LSp})$ in $\MMor$.
The result therefore follows if we can prove that these nonunital commutative monoid structures agree.
Observe that since $(\LSp,\unit_{\LSp})$ lies in the full sub-2-category $\pressymst \hookrightarrow \MMor$, it suffices to prove this in $\pressymst$.
We will deduce this from the following uniqueness statement:

\begin{proposition}\label{prop:uniqueness-monoid-LSp}
    Let $\LSp$ be a localization of $\Sp$ in the sense of \cref{def:localization-of-spectra}.
    Then there exists a unique nonunital commutative monoid structure on $\LSp$ in $\pressymst$ such that the multiplication map is a symmetric sequence concentrated in arity 2, where it is given by the projection
    \begin{equation}\label{eq:hessian-multiplication}\begin{split}
    ({\color{cbblue} \LSp} \oplus {\color{cborange} \LSp})^{\otimes 2}_{h\Sigma_2} &\simeq ({\color{cbblue} \LSp} \otimes {\color{cbblue} \LSp} \oplus {\color{cbblue} \LSp} \otimes {\color{cborange} \LSp} \oplus {\color{cborange} \LSp} \otimes {\color{cbblue} \LSp} \oplus {\color{cborange} \LSp} \otimes {\color{cborange} \LSp})_{h\Sigma_2} \\
    &\to ({\color{cbblue} \LSp} \otimes {\color{cborange} \LSp} \oplus {\color{cborange} \LSp} \otimes {\color{cbblue} \LSp})_{h\Sigma_2} \simeq (\LSp \oplus \LSp)_{h\Sigma_2} \simeq \LSp.
    \end{split}\end{equation}
\end{proposition}
Here the colours indicate onto which summands we are projecting.

\begin{proof}
    Observe that all structure maps of the nonunital commutative monoid structure on $\LSp$ must have arity at least 1, so we may replace $\pressymst$ with $\pressymstc$ without loss of generality.

    The monoid structure exists by \cref{lem:hessian}.
    For the uniqueness, let us write $M$ for the symmetric sequence $\Sym(\LSp \oplus \LSp) \to \LSp$ from the statement.
    Then $\Lambda_M \colon \LSp \times \LSp \to \LSp$ is the functor
    \[(X,Y) \mapsto M((X,Y),(X,Y))_{h\Sigma_2} \simeq X \otimes Y,\]
    i.e.\ the tensor product $\otimes \colon \LSp \times \LSp \to \LSp$.
    Since the composite
    \[
    \begin{tikzcd}
        \pressymstc \ar[r, "\Lambda"] & \diffsp \ar[r, "\partial_*"] & \pressymstc
    \end{tikzcd}
    \]
    is equivalent to the identity by \cref{prop: app-derivatives-lambda}, it suffices to prove that there is a unique nonunital symmetric monoidal structure on $\LSp$ whose multiplication map $\LSp \times \LSp \to \LSp$ is equivalent to the usual tensor product on $\LSp$.
    This follows by combining the following statements:
    \begin{enumerate}[(i)]
        \item The given nonunital symmetric monoidal structure on $\LSp$ lifts to an object in $\CAlg^\mathrm{nu}(\preslst)$. This follows since the multiplication map $\LSp \times \LSp \to \LSp$ preserves colimits in each variable separately.
        \item By \cite[Corollary 5.4.4.7]{HA}, the forgetful functor $\CAlg(\preslst) \to \CAlg^\mathrm{nu}(\preslst)$ is an inclusion, with image the quasi-unital commutative algebras and quasi-unital morphisms.
        \item $L\Sph$ is a quasi-unit for the given nonunital symmetric monoidal structure on $\LSp$, since the definition of a quasi-unit in \cite[Remark 5.4.4.4]{HA} only depends on the multiplication map $\LSp \otimes \LSp \to \LSp$.
        \item Since $\LSp$ is idempotent in $\preslst$, by \cite[Proposition 4.8.2.9]{HA} there is a unique presentably symmetric monoidal structure on $\LSp$ with unit $L\Sph$. \qedhere
    \end{enumerate}
\end{proof}

\begin{remark}
    This proof moreover shows that if $\LSp$ is equipped with such a monoid structure in $\pressymst$, then its space of automorphisms in $\CMon^\mathrm{nu}(\pressymst)$ is contractible.
\end{remark}

\begin{corollary}\label{cor:Day-monoid-equals-derived-monoid}
    The nonunital commutative monoid structure on $\LSp$ in $\MMor$ constructed in \cref{thm:General-prod-rule} agrees with the commutative monoid structure on $\LSp$ constructed in \cref{prop:Day-convolution-as-pointwise-tensor} (after forgetting the unit).
\end{corollary}

\begin{proof}
    Throughout this proof, we call a functor $\Fin^\simeq \times \Fin^\simeq \to \LSp$ a \emph{bisymmetric sequence}.
    Given $m,n \geq 0$ and $X \in \Fun(B\Sigma_m \times B\Sigma_n, \LSp)$, we write $X(m,n)$ for the bisymmetric sequence concentrated in bidegree $(m,n)$ with value $X$.

    By \cref{prop:uniqueness-monoid-LSp}, it suffices to prove that the multiplication map $\mu$ of the monoid structure on $\LSp$ from \cref{prop:Day-convolution-as-pointwise-tensor} agrees with \cref{eq:hessian-multiplication}.
    Since this monoid structure was constructed using the Yoneda lemma, it follows that its multiplication map $\mu$ is the image of the identity under
    \[\SSeq(\LSp \oplus \LSp, \LSp \oplus \LSp) \xrightarrow{\mu \circ -} \SSeq(\LSp \oplus \LSp, \LSp).\]
    Identifying $\LSp \oplus \LSp$ with $\cP_\LSp(* \sqcup *)$ and observing that $\Sym^\times(* \sqcup *) \simeq \Fin^\simeq \times \Fin^\simeq$, it follows from \cref{prop:Day-convolution-as-pointwise-tensor} that we may identify this functor with the Day convolution product
    \[\circledast \colon \cP_{\LSp}(\Fin^\simeq \times \Fin^\simeq)^{\times 2} \to \cP_{\LSp}(\Fin^\simeq \times \Fin^\simeq).\]
    Unwinding the string of equivalences \cref{eq:SSeq-without-L}, it follows that the identity of $\LSp \oplus \LSp$ corresponds to the pair of bisymmetric sequences $(L\Sph(1,0), L\Sph(0,1))$.
    The Day convolution product of these bisymmetric sequences is $L\Sph(1,1)$. Unwinding the string of equivalences \cref{eq:SSeq-without-L} again, it follows that this bisymmetric sequence corresponds to the functor symmetric sequence $\Sym(\LSp \oplus \LSp) \to \LSp$ from \cref{prop:uniqueness-monoid-LSp}, which was what we needed to show.
\end{proof}

\begin{corollary}\label{cor:convolution-prod-rule-nonunital-2}
    Let $\cC$ be a differentiable category such that $\Sp \cC \simeq \LSp$.
    Then the Goodwillie derivatives form a nonunital strong symmetric monoidal functor
    \[\partial_* \colon (\Fun^{\ast,\omega}(\cC, \LSp), \otimes_{\mathrm{pt}}) \to (\RMod_{\partial_* {\id_\cC}}(\SSeq(\LSp,\LSp)), \circledast),\]
    where the target is equipped with the Day convolution symmetric monoidal structure from \cref{def:convolution-product-sseq}.
\end{corollary}

\begin{proof}
    By \cref{cor:Day-monoid-equals-derived-monoid}, the monoidal structure from \cref{thm:General-prod-rule} and \cref{def:convolution-product-sseq} are induced by the same nonunital monoid structure on $\LSp$ in $\pressymst \subset \MMor$, hence they agree.
\end{proof}

\begin{remark}
    This proof also shows that the nonunital commutative monoid structure on $(\LSp,\unit_{\LSp})$ in $\MMor$ is quasi-unital.
    We expect this to be false if $\LSp$ is replaced by an unstable presentably symmetric monoidal category, such as $(\Spc_*,\wedge)$.
    If this were the case, then one could use this to construct a unit for the smash product on the category $\Alg_{\Lie}(\Sp)$ of spectral Lie algebras.
    We suspect that such a unit does not exist; see also the related Remark 4.10 of \cite{blansheuts2026characterizationspectrallieoperad}.
\end{remark}

\subsection{A unital version of the product rule}\label{ssec:unital-prod-rule}

Given $F \colon \cC \to \cD$ in $\Diff$ with $\cD$ stable, the Goodwillie derivatives $\partial_* F$ form a right $\partial_* \id_\cC$-module in $\SSeq(\Sp \cC, \cD)$ that is concentrated in degrees $\geq 1$.
Since $\cD$ is stable, one could reasonably extend the Goodwillie derivatives by declaring that $\partial_0 F = F(*)$, where one allows $F$ to be unreduced.
We will now show that one can indeed do this, and that the product rule of \cref{cor:convolution-prod-rule-nonunital-2} upgrades to a \emph{unital} strong symmetric monoidal functor in this case.

\begin{theorem}\label{thm:convolution-product-rule-unital}
    Let $\cC$ be a differentiable category such that $\Sp \cC \simeq \LSp$.
    Then the Goodwillie derivatives refine to a \emph{unital} strong symmetric monoidal functor
    \begin{align*}
    \partial_* \colon (\Fun^{\omega}(\cC, \LSp), \otimes_{\mathrm{pt}}) &\to (\RMod_{\partial_* {\id_\cC}}(\SSeq(\LSp,\LSp)), \circledast) \\
    F &\mapsto \{\partial_n F\}_{n \geq 0}.
    \end{align*}
\end{theorem}

For the proof we first need to know that $\Fun^{\omega}(\cC,\LSp)$ is presentable.

\begin{lemma}\label{lem:fun-omega-presentable}
    Let $\cC$ be a differentiable category and consider $\Fun^{\omega}(\cC,\LSp)$ as a symmetric monoidal category by equipping it with the pointwise tensor product.
    Then $\Fun^\omega(\cC,\LSp)$ lies in $\CAlg(\presl_\LSp)$.
    Similarly, $\Fun^{*,\omega}(\cC,\LSp)$ lies in $\CAlg^\mathrm{nu}(\presl_\LSp)$.
\end{lemma}

\begin{proof}
    First observe that the pointwise tensor product preserves colimits in each variable since colimits in $\Fun^{\omega}(\cC,\LSp)$ and $\Fun^{*,\omega}(\cC,\LSp)$ are computed pointwise.
    It therefore suffices to prove that $\Fun^{\omega}(\cC,\LSp)$ and $\Fun^{*,\omega}(\cC,\LSp)$ lie in $\presl_\LSp$.
    By \cite[Proposition 3.3.3]{blansblom2025chainrulegoodwilliecalculus}, there exists a presentable category $\cP^\mathrm{fin}(\cC)$ such that $\Fun^{*,\omega}(\cC,\LSp) \simeq \FunL(\cP^\mathrm{fin}(\cC),\LSp)$.
    Since $\FunL(-,-)$ is the internal hom of $\presl$ and the localization $\presl \to \presl_\LSp$ is strong monoidal, it follows that $\FunL(\cP^\mathrm{fin}(\cC),\LSp)$ lies in $\presl_\LSp$.
    For $\Fun^{\omega}(\cC,\LSp)$, note that we have the functor
    \[\Phi \colon \Fun^{*,\omega}(\cC,\LSp) \times \LSp \to \Fun^{\omega}(\cC,\LSp); \quad (F,X) \mapsto F \oplus \const_X.\]
    We claim this functor is an equivalence, from which the result follows.
    To see this, observe that we have natural equivalences
    \[\Map(F \oplus {\const_X}, G) \simeq \Map(F,G) \times \Map(\const_X, G) \simeq \Map((F,X),(\red G, G(*))),\]
    where $\red$ denotes the reduction functor, see e.g.\ \cite[Proposition 3.1.19]{blansblom2025chainrulegoodwilliecalculus}.
    The functor $\Phi$ therefore admits a right adjoint given by $G \mapsto (\red G, G(*))$.
    The unit $(F,X) \to (\red(F \oplus \const_X), (F \oplus \const_X)(*))$ and counit $\red(G) \oplus G(*) \to G$ are easily checked to be invertible, so we conclude that $\Phi$ is an equivalence.
\end{proof}

Recall from \cite[Proposition 5.4.4.8]{HA} that the forgetful functor $\CAlg(\presl_\LSp) \to \CAlg^\mathrm{nu}(\presl_\LSp)$ admits a left adjoint sending $\cA$ to its \emph{unitalization} $\cA^\mathrm{un} \simeq \cA \oplus \LSp$.
The main technical input to \cref{thm:convolution-product-rule-unital} is that both the left-hand and the right-hand side are unitalizations of nonunital presentably symmetric monoidal $\LSp$-local categories.

\begin{lemma}\label{lem:nonreduced-decomposition-fun}
    Let $\cC$ be a differentiable category.
    Then the canonical functor
    \[\Phi \colon \Fun^{\ast,\omega}(\cC,\LSp)^\mathrm{un} \simeq \Fun^{\ast,\omega}(\cC, \LSp) \oplus \LSp \to \Fun^{\omega}(\cC,\LSp)\]
    is a symmetric monoidal equivalence with respect to the pointwise tensor product.
\end{lemma}

\begin{proof}
    The symmetric monoidal functor $\Phi$ is defined as the adjunct of the nonunital symmetric monoidal inclusion $\Fun^{\ast, \omega}(\cC,\LSp) \hookrightarrow \Fun^{\omega}(\cC,\LSp)$ in $\CAlg^\mathrm{nu}(\presl_\LSp)$, so by \cite[Proposition 5.4.4.8]{HA} it is given on underlying categories by $(F,X) \mapsto F \oplus {\const_X}$.
    We saw in the proof of \cref{lem:fun-omega-presentable} that this functor is an equivalence.
\end{proof}

\begin{lemma}\label{lem:decomposition-rmod}
    Let $\cO$ be an $\LSp$-enriched operad.
    Then the canonical functor
    \[\Phi \colon \RMod_{\cO}(\SSeq_{\geq 1}(\LSp,\LSp)) \oplus \LSp \to \RMod_{\cO}(\SSeq(\LSp,\LSp))\]
    is a symmetric monoidal equivalence with respect to the Day convolution symmetric monoidal structure.
\end{lemma}

\begin{proof}
    Let us write $\SSeq_{=0}(\LSp,\LSp) \subset \SSeq(\LSp,\LSp)$ for the symmetric sequences concentrated in degree 0; this full subcategory is canonically equivalent to $\LSp$.
    Since $\cO(0) = 0$, it follows that the inclusion
    \[\RMod_\cO(\SSeq_{\geq 1} (\LSp,\LSp))\hookrightarrow \RMod_\cO(\SSeq(\LSp,\LSp))\]
    admits a right adjoint given by $M \mapsto M_{\geq 1}$ by \cite[Construction 7.4.9]{Blans-Thesis}.
    Also note that if $F$ is a symmetric sequence concentrated in degree 0, then $F \circ \cO \simeq F$ by the formula for the composition product \cite[Proposition 3.2.9]{blansblom2025chainrulegoodwilliecalculus}.
    It follows that the unit of the composite adjunction
    \[
    \begin{tikzcd}[column sep = large]
        \SSeq_{=0}(\LSp,\LSp) \ar[r, "\mathrm{incl}", shift left] & \SSeq(\LSp,\LSp) \ar[l,shift left,"M_0 \mapsfrom M"] \ar[r,"{\mathrm{free}_\cO}", shift left] & \RMod_\cO(\SSeq(\LSp,\LSp)) \ar[l, "\mathrm{fgt}", shift left]
    \end{tikzcd}
    \]
    is invertible, and that the counit is invertible on those right $\cO$-modules that are concentrated in arity 0.
    We therefore have an inclusion
    \[\LSp \simeq \SSeq_{=0}(\LSp,\LSp) \simeq \RMod_\cO(\SSeq_{=0}(\LSp,\LSp)) \hookrightarrow \RMod_\cO(\SSeq(\LSp,\LSp))\]
    with right adjoint given by $M \mapsto M_0$.
    It now follows by the same proof as \cref{lem:nonreduced-decomposition-fun} that $\Phi$ admits a right adjoint which is a two-sided inverse.
\end{proof}

\begin{proof}[Proof of \cref{thm:convolution-product-rule-unital}]
    Applying $(-)^\mathrm{un} \colon \CAlg^\mathrm{nu}(\presl_\LSp) \to \CAlg(\presl_\LSp)$ to
    \[\partial_* \colon (\Fun^{\ast,\omega}(\cC, \LSp), \otimes_{\mathrm{pt}}) \to (\RMod_{\partial_* {\id_\cC}}(\SSeq_{\geq 1}(\LSp,\LSp)), \circledast),\]
    it extends to a symmetric monoidal functor $\Fun^{\omega}(\cC, \LSp) \to \RMod_{\partial_* {\id_\cC}}(\SSeq(\LSp,\LSp))$.
    To see that this is indeed the Goodwillie derivatives functor, note that $\partial_0 F = F(*)$ by \cref{def:definition-partial0}, and $\partial_n F \simeq \partial_n (\red F)$ by \cref{lem:partial-vs-reduction} for $n \geq 1$.
\end{proof}

\subsection{The coendomorphism operad of \texorpdfstring{$\Sigma^\infty$}{Σ\textasciicircum ∞}}\label{ssec:coend-operad-suspension}

We will now leverage the product rule and the description of Koszul duality from \cref{prop:KP-is-coend-operad} to obtain a description of the operad structure on (the Koszul dual of) $\partial_* \id_\cC$ that makes no reference to Goodwillie calculus.

\begin{proposition}\label{prop:main-computation-coend-operad}
    Let $\cC$ be a differentiable category such that $\Sp \cC \simeq \LSp$ and let $F \colon \cC \to \LSp$ be a finitary functor.
    Then the map
    \[\ulnat(F, (\Sigma^\infty_\cC)^{\otimes n}) \xrightarrow{\partial_*} \ulmap_{\RMod_{\partial_* \id_\cC}}(\partial_* F, (\partial_* \Sigma^\infty_\cC)^{\circledast n})\]
    is an equivalence of spectra for every $n \geq 0$.
\end{proposition}

\begin{proof}
First assume $F$ is reduced. Observe that $\partial_* \colon \Diff \to \MMor$ takes the adjunction $\Sigma^\infty_\cC \dashv \Omega^\infty_\cC$ to an adjunction $\partial_* \Sigma^\infty_\cC \dashv \partial_* \Omega^\infty_\cC$ in $\MMor$.
The mate correspondence therefore gives a commutative square
\[\begin{tikzcd}[ampersand replacement=\&]
	{\ulnat(F,(\Sigma^\infty_\cC)^{\otimes n})} \& {\ulmap_{\RMod}(\partial_*F,(\partial_*\Sigma^\infty_\cC)^{\circledast n})} \\
	{\ulnat(F\Omega^\infty_\cC,\id_{\LSp}^{\otimes n})} \& {\ulmap_{\SSeq}(\partial_*(F\Omega^\infty_\cC), \unit^{\circledast n})}
	\arrow["{\partial_*}", from=1-1, to=1-2]
	\arrow["\sim"', from=1-1, to=2-1]
	\arrow["\sim", from=1-2, to=2-2]
	\arrow["{\partial_*}"', from=2-1, to=2-2]
\end{tikzcd}\]

The functor $\partial_* \colon \Fun^{\ast, \omega}(\LSp,\LSp) \to \SSeq(\LSp,\LSp)$ admits a right adjoint $\Psi$ by \cref{cor: right-adjoint-derivatives-finitary}.
To see that the bottom map in the square is an equivalence, it therefore suffices to show that the unit map $\id_{\LSp}^{\otimes n} \to \Psi \partial_* (\id_{\LSp}^{\otimes n})$ is an equivalence.
By \cref{cor: unit-right-adjoint-der-norm}, this unit is given by applying the right adjoint of the inclusion $\Fun^{\ast, \omega}(\LSp, \LSp) \hookrightarrow \Fun^{\ast}(\LSp, \LSp)$
to the norm map
\[(\partial_n \id_{\LSp}^{\otimes n} \otimes (-)^{\otimes n})_{h\Sigma_n} \to (\partial_n \id_{\LSp}^{\otimes n} \otimes (-)^{\otimes n})^{h \Sigma_n}.\]
This is an equivalence since $\partial_n \id^{\otimes n}_{\LSp} \simeq (\unit^{\circledast n})_n \simeq \Ind_*^{\Sigma_n} L\Sph$ by formula \cref{eq:Day-convolution-formula} for Day convolution.

Now suppose $F$ is unreduced.
By the equivalences $\Fun^\omega(\cC,\LSp) \simeq \Fun^{\ast,\omega}(\cC,\LSp) \oplus \LSp$ and $\RMod_{\cO}(\SSeq_{\geq 1}(\LSp,\LSp)) \oplus \LSp \simeq \RMod_{\cO}(\SSeq(\LSp,\LSp))$ from \cref{lem:nonreduced-decomposition-fun,lem:decomposition-rmod}, the statement for $n \geq 1$ follow from the reduced case proved above.
On the other hand, since $\Sigma_\cC^{\otimes 0} = \const_{L\Sph}$ and $(\partial_* \Sigma^\infty_\cC)^{\circledast 0}$ is $L\Sph$ concentrated in degree 0, the case $n=0$ follows immediately.
\end{proof}

The following computation is a direct consequence.

\begin{corollary}
    Let $F \colon \cC \to \LSp$ be as in \cref{prop:main-computation-coend-operad}.
    Then there is a $\Sigma_n$-equivariant equivalence of spectra
    \[\ulnat(F,(\Sigma^\infty_\cC)^{\otimes n}) \simeq \ulmap(\partial_n(F \Omega^\infty_\cC), L\Sph).\]
    In particular, if $\sC = \LSp$, then $\ulnat(F,\id_{\LSp}^{\otimes n}) \simeq \mathbf{D}\partial_n F$.
\end{corollary}

\begin{proof}
    The proof of \cref{prop:main-computation-coend-operad} shows that $\ulnat(F,(\Sigma^\infty_\cC)^{\otimes n}) \simeq \ulmap_{\SSeq}(\partial_*(F \Omega^\infty_\cC), \unit^{\circledast n})$.
    The result follows since by the formula given in \cref{ssec:coend-operads-and-Koszul-duality}, $\unit^{\circledast n}$ is concentrated in degree $n$, where it is equivalent to $\Ind_*^{\Sigma_n} L\Sph \simeq \Coind_*^{\Sigma_n} L\Sph$.
\end{proof}

\begin{remark}
    For a reduced finitary functor $F \colon \Sp \to \Sp$, the equivalence between $\ulnat(F, \id_{\Sp}^{\otimes n})$ and $\mathbf{D}\partial_nF$ was already observed by Arone--Ching \cite[\S 3.1]{AroneChing2011}.
\end{remark}

The main theorem of this section (cf. \cref{introcor: product-rule-koszul-duality,introcor: right-module-der}) now follows by combining the results proved above with \cref{prop:KP-is-coend-operad}.

\begin{theorem} \label{thm: koszul-dual-der-id-is-coend-sigma-infty}
    Let $\cC$ be a differentiable category with $\Sp \cC \simeq \LSp$ and such that $\partial_n \id_\cC$ is dualizable in $\LSp$ for every $n \geq 1$.
    Then there is an equivalence $K\partial_* \id_\cC \simeq \coEnd(\Sigma_\cC^\infty)$ of operads.
    Moreover, if $F \colon \cC \to \LSp$ is a finitary functor such that $\partial_n F$ is dualizable in $\LSp$ for every $n$, then $K\partial_* F \simeq \ulnat(F,(\Sigma^\infty_\cC)^{\otimes \bullet})$ as right $K\partial_* \id_\cC$-modules.
\end{theorem}

\begin{proof}
    By \cref{thm:convolution-product-rule-unital}, the functor $\partial_* \colon \Fun^{\omega}(\cC,\LSp) \to \RMod_{\partial_* \id_\cC}(\SSeq(\LSp,\LSp))$ is strong symmetric monoidal with respect to the pointwise tensor product on its domain and the Day convolution product on its target.
    In particular, the map $\coEnd(\Sigma_\cC^\infty) \to \coEnd(\partial_* \Sigma_\cC^\infty)$ is a map of algebras in $\SSeq(\LSp,\LSp)$; i.e.\ of $\LSp$-enriched operads.
    Similarly, for any $F$ in $\Fun^{\omega}(\cC,\LSp)$, the map $\ulnat(F,(\Sigma_\cC^\infty)^{\otimes n}) \to \ulmap_{\RMod}(\partial_* F, (\partial_*\Sigma_\cC^\infty)^{\circledast n})$ is a map of right $\coEnd(\Sigma_\cC^\infty)$-modules.
    It follows from \cref{prop:main-computation-coend-operad} that both maps are equivalences. Since $\partial_* \Sigma^\infty_\cC = \unit$ is the trivial right module $L\Sph$ concentrated in degree 1, the result follows from \cref{prop:KP-is-coend-operad}.
\end{proof}

\section{Examples}

\subsection{Pointed spaces}\label{ssec:pointed-spaces}

Ching proved that the derivatives of the identity functor in pointed spaces are equivalent to the spectral Lie operad \cite{chingThesis}.
He worked in a point-set model, and it is not clear how to deduce from his result that the operad $\partial_*{\id_{\Spc_*}}$ coming from the lax monoidal structure on $\partial_*$ described in \cref{ssec: prelim-chain-rule} is given by the spectral Lie operad.
In this section, we give a proof of this fact using the product rule.
A different proof appeared in \cite[Proposition 5.8]{blansheuts2026characterizationspectrallieoperad}.
We start by defining the commutative and spectral Lie operads.

\begin{definition}
    The \emph{unital commutative operad in spectra} $\mathbf{Com}^u$ is the endomorphism operad $\End_{\Sp}(\Sph)$.
\end{definition}

This is perhaps not the standard definition of $\mathbf{Com}^u$.
It is however very easy to compare different definitions of $\mathbf{Com}^u$ by means of the following proposition.

\begin{proposition}
    Let $\sO$ be an operad in spectra such that $\sO(0) \simeq \Sph$ and for all $n \geq 0$ the composition map
    \[
    \sO(n) \otimes \sO(0) \otimes \cdots \otimes \sO(0) \to \sO(0)
    \]
    is an equivalence.
    Then $\sO \simeq \mathbf{Com}^u$.
\end{proposition}
\begin{proof}
    By \cite[Proposition 7.4.6]{Blans-Thesis}, we get a comparison map $\sO \to \End(\sO(0))$ that is given by the adjoint of the composition map
    \[
    \sO(n) \otimes \sO(0) \otimes \cdots \otimes \sO(0) \to \sO(0)
    \]
    in arity $n$.
    The assumptions imply that this is an equivalence of operads.
\end{proof}

\begin{remark}
    This proposition implies that the following definitions of $\mathbf{Com}^u$ are all equivalent:
    \begin{enumerate}[\upshape{(}\arabic*\upshape{)}]
        \item The endomorphism operad $\End_{\Sp}(\Sph)$;
        \item The coendomorphism operad $\coEnd_{\Sp}(\Sph)$;
        \item The functor $\Sigma^\infty_+$ applied to the unital commutative operad in $\mathrm{Set}$;
        \item The construction by Raksit given in \cite[Construction 4.1.6]{raksit2020hochschildhomologyderivedrham};
        \item The Spanier--Whitehead dual of the unital commutative cooperad.
    \end{enumerate}
\end{remark}

By \cite[Construction 7.4.9]{Blans-Thesis}, there is an adjunction (left adjoint on top)
\[
\begin{tikzcd}
    \Alg(\SSeq_{\geq 1}(\Sp)) \ar[r, shift left, hook] & \Alg(\SSeq(\Sp)) \ar[l, shift left, "\tau_{\geq 1}"].
\end{tikzcd}
\]
Given $\sO \in \Alg(\SSeq(\Sp))$, we call the operad $\tau_{\geq 1}\sO$ its \emph{non-unitalization}; the underlying symmetric sequence of $\tau_{\geq 1}\sO$ is obtained by deleting the $0$th term of $\sO$.

\begin{definition}
    The \emph{commutative operad in spectra} $\mathbf{Com}$ is the non-unitalization of $\mathbf{Com}^u$.
    The \emph{spectral Lie operad} $\mathbb{L}$ is its Koszul dual $K\mathbf{Com}$.
\end{definition}

Our proof will make use of the following version of the Yoneda lemma.

\begin{lemma} \label{lem: stable-yoneda-lemma}
    Let $\sC$ be a pointed category and let $F \colon \sC \to \Sp$ be a reduced functor. For each object $x \in \sC$, there is a natural equivalence
    \[
    \ulnat(\Sigma^\infty\Map_\sC(x, -), F) \simeq F(x)
    \]
    induced by evaluation at $\id_x$.
\end{lemma}
\begin{proof}
    The functors $\Fun^*(\sC, \Sp) \to \Sp$ sending $F$ to $\ulnat(\Sigma^\infty\Map_\sC(x, -), F)$ and $F(x)$ both preserve finite limits.
    Therefore, they are equivalent if and only if they are after postcomposition with $\Omega^\infty$ by the universal property of stabilization \cite[Proposition 1.4.2.23]{HA}.
    But by adjunction, we have
    \[
    \Omega^\infty\ulnat(\Sigma^\infty\Map_\sC(x, -), F) \simeq \Map_{\Fun^{\ast}(\sC, \Spc_*)}(\Map_\sC(x, -), \Omega^{\infty}F) \simeq \Omega^\infty F(x).
    \]
    The last step follows from the regular Yoneda lemma, using that as $\sC$ is pointed the forgetful functor $\Fun^\ast(\sC, \Spc_*) \to \Fun(\sC, \Spc)$ is fully faithful with essential image the functors that preserve the final object.
\end{proof}

\begin{theorem}
    There is an equivalence of operads $\partial_*{\id_{\Spc_*}} \simeq \mathbb{L}$.
\end{theorem}
\begin{proof}
    By definition of the spectral Lie operad, it suffices to prove that $K\partial_*{\id_{\Spc_*}} = \mathbf{Com}$. By \cref{thm: koszul-dual-der-id-is-coend-sigma-infty}, there is an equivalence of operads
    \[
    \coEnd_{\Fun^{\omega}(\Spc_*, \Sp)}(\Sigma^\infty) \simeq K\partial_*{\id_{\Spc_*}}.
    \]
    The strong monoidal functor $\ev_{S^0} \colon \Fun^\omega(\Spc_*, \Sp) \to \Sp$ induces a map of coendomorphism operads $\coEnd_{\Fun^{\omega}(\Spc_*, \Sp)}(\Sigma^\infty) \to \coEnd_{\Sp}(\Sph)$.
    In arity $n$, this map is given by
    \[
    \ulnat(\Sigma^\infty, (\Sigma^\infty)^{\otimes n})  \xrightarrow{\ev_{\id_{S^0}}} \ulmap(\Sph, \Sph^{\otimes n}),
    \]
    which is an equivalence if $n \geq 1$ by \cref{lem: stable-yoneda-lemma}, since $\Omega^\infty$ is corepresented by $\Sph$.
    In arity zero, it is given by the map $0 \to \Sph$.
    This means that the map $\coEnd_{\Fun^{\omega}(\Spc_*, \Sp)}(\Sigma^\infty) \to \coEnd_{\Sp}(\Sph)$ exhibits the source as the non-unitalization of the target, so that we can conclude that
    \[
    \coEnd_{\Fun(\Spc_*, \Sp)}(\Sigma^\infty) \simeq \mathbf{Com}. \qedhere
    \]
\end{proof}

\subsection{Algebras over an operad} \label{sec: der-id-in-o-alg}

Throughout this section, we fix a stable presentable category $\sA$ and a strongly positive algebra $\sO \in \Alg(\SSeq_+(\sA, \sA))$.
We will prove the following theorem:
\begin{theorem}\label{thm: der-id-o-alg}
    There is an equivalence of algebras 
    \[
    \partial_*{\id_{\Alg_\sO}} \simeq \sO.
    \]
\end{theorem}

\begin{remark}
Previous work in this direction has been done in the case $\sA = \Sp$.
Pereira proved the result on the level of symmetric sequences \cite{PereiraThesis}.
In a different setting, and using completely different methods, Ching proved that the operads $\sO$ and $\partial_*{\id_{\Alg_\sO}}$ are Morita equivalent \cite{ChingDayConvolution}.
Some of the results from this section already appeared in the thesis of the first author \cite[Section 7.2]{Blans-Thesis}.
\end{remark}

Before giving the proof, we recall a basic fact about the category $\Alg_\sO(\sA)$.
Projection onto the arity $1$ component defines a map of operads $\sO \to \unit_\sA$.
Extension and restriction of scalars along this morphism gives an adjunction
\[
\begin{tikzcd}[sep = large]
\Alg_\sO(\sA) \ar[r, "\cot_\sO", shift left] & \sA. \ar[l, "\triv_\sO", shift left] 
\end{tikzcd}
\]
As explained, for instance, in \cite[Proposition 6.3]{heuts2024koszulduality}, this adjunction exhibits $\sA$ as the stabilization of $\Alg_\sO(\sA)$.
In particular, we have equivalences
\[
    \Sigma^\infty_{\Alg_\sO} \simeq \cot_\sO \qquad \mathrm{and} \qquad \Omega^\infty_{\Alg_\sO} \simeq \triv_\sO.
\]

We now come to the proof of \cref{thm: der-id-o-alg}.
In fact, we give two proofs.

\begin{proof}[First proof]
    Applying the Goodwillie transform to the monadic adjunction
    \[
    \free_\sO \colon \sA \rightleftarrows \Alg_\sO(\sA) \colon \fgt_\sO
    \]
    we obtain a monadic adjunction by \cref{cor: monadic-adjunction-goodwillie-transf}.
    The corresponding monad of this adjunction is
    \[
    \maAlg(\fgt_\sO \circ \free_\sO) \simeq \maAlg(\Lambda_\sO) \simeq \Lambda_\sO,
    \]
    where the last equivalence follows from \cref{prop: lambda-diff-is-id}.
    We therefore get a commutative diagram
    \[
    \begin{tikzcd}
         \Alg_{\partial_*{\id_{\Alg_\sO}}} \ar[dr, "\maAlg(\fgt_\sO)"'] \ar[rr, "\sim"] & &  \Alg_\sO.   \ar[dl, "\fgt_\sO"] \\
        & \sA &
    \end{tikzcd}
    \]
    By \cref{prop: derivatives-free-left-right-module}, $\partial_*{\fgt_\sO}$ is the free right $\partial_*{\id_{\Alg_\sO}}$-module on $\partial_1{\fgt_{\sO}}$.
    But this first derivative is equivalent to $\id_\sA$, since $\fgt_\sO \circ \triv_\sO = \id_\sA$ and $\triv_\sO \simeq \Omega^\infty_{\Alg_\sO}$.
    Hence, $\partial_*{\fgt_\sO} \simeq \partial_*{\id_{\Alg_\sO}}$ as right $\partial_*{\id_{\Alg_\sO}}$-module and $\maAlg(\fgt_\sO) \simeq \fgt_{\partial_*{\id_{\Alg_\sO}}}$.
    The above diagram therefore gives an equivalence of monads
    \[
    \Lambda_{\partial_*{\id_{\Alg_\sO}}} \simeq \Lambda_\sO.
    \]
    Taking derivatives finishes the proof.
\end{proof}

\begin{remark}
    More generally, this proof shows that whenever $\sA$ is stable, $\sC$ is differentiable, and $F \colon \sA \rightleftarrows \sC \colon G$ is a finitary monadic adjunction that induces an equivalence on stabilizations, there is an equivalence of algebras $\partial_*{\id_\sC} \simeq \partial_*{GF}$.
\end{remark}

\begin{proof}[Second proof]
    There is an equivalence of comonads
    \[
    \cot_\sO \triv_\sO \simeq \Lambda_{B{\sO}},
    \]
    as is proved in \cite[Proposition 3.28]{brantner2023pd}.
    It then follows from \cref{prop: lambda-diff-is-id} that we have an equivalence of coalgebras
    \[
    \partial_*{\cot_\sO \triv_\sO} \simeq B{\sO}.
    \]
    On the other hand, we have $\cot_\sO \triv_\sO \simeq \Sigma^\infty_{\Alg_\sO} \Omega^\infty_{\Alg_\sO}$ so that
    \[
    \partial_*{\cot_\sO \triv_\sO} \simeq B\partial_*{\id_{\Alg_\sO}}
    \]
    by \cref{thm: calculus-and-koszul-duality-main}.
    Since the bar--cobar adjunction is an equivalence on strongly positive (co)algebras by \cref{thm:koszul-duality-operads}, we conclude that we have an equivalence of algebras
    \[
    \partial_*{\id_{\Alg_\sO}} \simeq \sO. \qedhere
    \]
\end{proof}

By the product rule, this implies:

\begin{corollary}
    Let $\sO$ be a strongly positive operad in $\LSp$ such that $\sO(n)$ is dualizable for all $n \geq 1$.
    There is an equivalence of operads
    \[
    \mathrm{coEnd}_{\Fun(\Alg_\sO, \LSp)}(\cot_\sO) \simeq K\sO.
    \]
\end{corollary}

\begin{remark}
     This result was already obtained by Malin \cite[Theorem 1.1.2]{malinUnstable1semiadditivityClassifying2025}.
\end{remark}

\subsection{Sheaves}
Suppose $\sT$ is a site and let $\sC$ be a differentiable category.
Then the category $\Sh(\sT;\sC)$ of $\sC$-valued sheaves on $\sT$ is again differentiable.
The purpose of this section is to prove the following result.

\begin{theorem} \label{thm: goodwillie-transf-sheaves}
    There is an equivalence of categories
    \[
    \maAlg(\Sh(\sT; \sC)) \simeq \Sh(\sT; \Alg_{\partial_*{\id_\sC}}),
    \]
    and the Goodwillie transform preserves the sheafification adjunction.
\end{theorem}

We will need a lemma for the proof of this theorem. 
Write
\[
\begin{tikzcd}
    L \colon \sP(\sT; \sC) \ar[r, shift left] & \Sh(\sT; \sC) \colon \iota \ar[l, shift left, hook']
\end{tikzcd}
\]
for the sheafification adjunction, where $\sP(\sT; \sC)$ denotes the category of $\sC$-valued presheaves on $\sT$.
Recall that this is a reflective localization, and that $L$ preserves finite limits.

\begin{lemma}
    There are equivalences of categories
    \[
    \Sp(\sP(\sT; \sC)) \simeq \sP(\sT; \Sp\sC) \qquad \text{and} \qquad \Sp(\Sh(\sT; \sC)) \simeq \Sh(\sT; \Sp\sC),
    \]
    and stabilization preserves the sheafification adjunction.
\end{lemma}
\begin{proof}
    Recall that for a category with finite limits $\sD$, its stabilization is given by 
    \[
    \Sp\sD = \Exc_*(\Spc_*^{\mathrm{fin}}, \sD),
    \]
    where the right-hand side denotes the category of reduced excisive functors from the category of finite pointed spaces to $\sD$.
    It follows immediately from this description that 
    \[
    \Sp(\sP(\sT; \sC)) \simeq \sP(\sT; \Sp\sC)
    \]
    Moreover, since the inclusion $\Exc_*(\Spc_*^{\mathrm{fin}}, \sC) \hookrightarrow \Fun(\Spc_*^{\mathrm{fin}}, \sC)$ preserves limits, we also have that
    \[
    \Sp(\Sh(\sT; \sC)) \simeq \Sh(\sT; \Sp\sC),
    \]
    and under this equivalence, postcomposition with the inclusion functor $\iota \colon \Sh(\sT; \sC) \hookrightarrow \sP(\sT; \sC)$
    induces the inclusion
    \[
    \Sh(\sT; \Sp\sC) \hookrightarrow \sP(\sT; \Sp\sC)
    \]
    on stabilizations.
    Since $\iota$ is a limit preserving functor, the functor induced by postcomposition is equivalent to $\partial_1{\iota}$ by \cite[Example 6.2.1.4]{HA}.
    Since taking first derivatives preserves adjunctions, it follows that
    \[
    \begin{tikzcd}
    \partial_1{L} \colon \sP(\sT; \Sp\sC) \ar[r, shift left] & \Sh(\sT; \Sp\sC) \colon \partial_1{\iota} \ar[l, shift left, hook']
    \end{tikzcd}
    \]
    is the sheafification adjunction.
\end{proof}

\begin{proof}[Proof of \cref{thm: goodwillie-transf-sheaves}]
    Since the Goodwillie transform preserves cotensors by \cref{cor: goodwillie-transform-preserves-cotensors}, we have that
    \[
    \maAlg(\sP(\sT; \sC)) \simeq \sP(\sT; \Alg_{\partial_*{\id_\sC}}).
    \]
    Applying $\maAlg$ to the sheafification adjunction, we get a reflective localization
    \[
    \begin{tikzcd}
    \maAlg(L) \colon \sP(\sT; \Alg_{\partial_*{\id_\sC}}) \ar[r, shift left] & \Alg_{\partial_*{\id_{\Sh(\sT; \sC)}}} \ar[l, shift left, hook'] \colon \maAlg(\iota).
    \end{tikzcd}
    \]
    Since $L$ preserves finite limits, it follows from \cref{prop: derivatives-free-left-right-module} that the underlying right $\partial_*{\id_{\sP(\sT; \sC)}}$-module of $\partial_*{L}$ is free on $\partial_1{L}$:
    \[
    \partial_*{L} \simeq \partial_1{L} \circ \partial_*{\id_{\sP(\sT; \sC)}}.
    \]
    This implies that the square
    \begin{equation*}\label{eq: forget-sheafification-square}
        \begin{tikzcd}
            \sP(\sT; \Alg_{\partial_*{\id_\sC}}) \ar[r, "\maAlg(L)"] \ar[d, "\forget"'] & \Alg_{\partial_*{\id_{\Sh(\sT; \sC)}}} \ar[d, "\forget"] \\
            \sP(\sT; \Sp\sC) \ar[r, "\partial_1(L)"] & \Sh(\sT; \Sp\sC)
        \end{tikzcd}
    \end{equation*}
    commutes; indeed, one has a natural equivalence
    \begin{align*}
    \forget(\maAlg(L)(F)) &= \forget(\partial_*L \circ_{\partial_*{\id_{\sP(\sT; \sC)}}} F) \\  
    &\simeq \partial_1L \circ \partial_*{\id_{\sP(\sT; \sC)}} \circ_{\partial_*{\id_{\sP(\sT; \sC)}}} F \simeq \partial_1(L)(\forget(F)),
    \end{align*}
    for $F \in \sP(\sT; \Alg_{\partial_*{\id_\sC}})$.
    In the same way, one shows that the square
    \begin{equation*}\label{eq: forget-inclusion-square}
        \begin{tikzcd}
            \sP(\sT; \Alg_{\partial_*{\id_\sC}}) \ar[d, "\forget"'] & \Alg_{\partial_*{\id_{\Sh(\sT; \sC)}}} \ar[d, "\forget"] \ar[l, "\maAlg(\iota)"', hook'] \\
            \sP(\sT; \Sp\sC) & \Sh(\sT; \Sp\sC) \ar[l, "\partial_1(\iota)"', hook']
        \end{tikzcd}
    \end{equation*}
    commutes, using that $\iota$ preserves limits.
    Writing $\eta \colon  \id_{\sP(\sT; \sC)} \Rightarrow \iota \circ L$ for the unit of the sheafification adjunction, we also have that
    \[
    \forget(\maAlg(\eta)) \simeq \partial_1{\eta},
    \]
    or in other words, that applying the forgetful functor to the unit of the adjunction $\maAlg(L) \dashv \maAlg(\iota)$ yields the unit of the adjunction $\partial_1{L} \dashv \partial_1{\iota}$.
    This can be seen by considering the commutative square of symmetric sequences
    \[
    \begin{tikzcd}
        \partial_1{\id_{\sP(\sT; \sC)}} \ar[r] \ar[d, "\partial_1{\eta}"'] & \partial_*{\id_{\sP(\sT; \sC)}} \ar[d, "\partial_*{\eta}"] \\
        \partial_1(\iota \circ L) \ar[r] & \partial_*(\iota \circ L)
    \end{tikzcd}
    \]
    and noting that each horizontal arrow exhibits its target as a free right $\partial_*{\id_{\sP(\sT; \sC)}}$-module on its source.
    
    A presheaf $F \colon \sT^\op \to \Alg_{\partial_*{\id_\sC}}$ is a sheaf if and only if the composite
    \[
    \sT^\op \xrightarrow{F} \Alg_{\partial_*{\id_\sC}}  \xrightarrow{\forget} \Sp\sC
    \]
    is one, since the forgetful functor preserves and reflects limits.
    It therefore follows from the second commutative square that $\maAlg(\iota)$ factors through the full subcategory $\Sh(\sT; \Alg_{\partial_*{\id_\sC}})$, so that the adjunction $\maAlg(L) \dashv \maAlg(\iota)$ restricts to an adjunction
    \[
    \begin{tikzcd}
    \maAlg(L) \colon \Sh(\sT; \Alg_{\partial_*{\id_\sC}}) \ar[r, shift left] & \Alg_{\partial_*{\id_{\Sh(\sT; \sC)}}} \colon \maAlg(\iota). \ar[l, shift left, hook']    
    \end{tikzcd}
    \]
    In order to show that this is an equivalence of categories, it suffices to show that the unit is an equivalence, which can be verified after applying the forgetful functor.
    But this adjunction forgets to the sheafification adjunction on $\sP(\sT; \Sp\sC)$ restricted to $\Sh(\sT; \Sp\sC)$, which is an equivalence of categories.
    This completes the proof.
\end{proof}

\appendix

\section{Derivatives of \texorpdfstring{$\Lambda$}{Lambda}}

The purpose of this appendix is to prove the following proposition.

\begin{proposition} \label{prop: app-derivatives-lambda}
    The composite
    \[
    \begin{tikzcd}
        \pressymstc \ar[r, "\Lambda"] & \diffsp \ar[r, "\partial_*"] & \pressymstc
    \end{tikzcd}
    \]
    is equivalent to the identity functor of the $2$-category $\pressymstc$.
\end{proposition}

Before giving the proof, we need to recall some details about the construction of the functor $\partial_* \colon \diffsp \to \pressymstc$ from \cite[Section 3.3]{blansblom2025chainrulegoodwilliecalculus}.
The idea is to extend the standard description of the derivatives functor as multilinearized cross-effects to a commutative diagram of the form
\begin{equation} \label{eq: lax-derivative-cross-mlin}
\begin{tikzcd}[sep = large]
    \Fun^{\ast, \omega}(\sA, \sB) \ar[r, "\cross"] \ar[dr, "\widetilde{\cross}"'] & \SymFun^{\ast, \omega}_{\geq 1}(\sA, \sB) \ar[r, "\mlin"] \ar[d, shift left, hook, "\iota"] & \SSeq_{\geq 1}(\sA, \sB) \\
    & \SSeq_{\geq 1}(\Pfin(\sA), \Pfin(\sB)) \ar[ur, "\widetilde{\mlin}"'] \ar[u, shift left, "L"]
\end{tikzcd}
\end{equation}
Here $\sA$ and $\sB$ are stable presentable categories, and $\SymFun^{\ast, \omega}_{\geq 1}(\sA, \sB)$ is the category whose objects are sequences of functors
\[
F_n \colon \sA^{\times n}_{h \Sigma_n} \to \sB \qquad \text{for $n \geq 1$}
\]
such that the underlying functor of each $F_n$ is finitary and reduced in each variable separately.
The functors $\cross$ and $\mlin$ are given by symmetric cross-effects and multilinearization respectively, so that 
\[
\partial_* = \mlin \circ \cross.
\]
The categories $\Pfin(\sA)$ and $\Pfin(\sB)$ are obtained from $\sA$ and $\sB$ by formally adjoining non-empty finite colimits; they are defined as appropriate localizations of presheaf categories.
The functor $\iota$ is induced by the Yoneda embedding, and its left adjoint $L$ is induced by the functor that sends a formal colimit in $\Pfin(\sB)$ to the actual colimit in $\sB$.
Since $L \circ \iota \simeq \id$, we have that 
\[
\partial_* \simeq \widetilde{\mlin} \circ \widetilde{\cross}
\]
as well.
Since this will be useful later, we now remark that the adjunction $L \dashv \iota$ extends to not necessarily positive symmetric sequences:
\[
L \colon \SSeq(\Pfin(\sA), \Pfin(\sB)) \rightleftarrows \SymFun^{\ast, \omega}(\sA, \sB) \colon \iota.
\]
Here $\SymFun^{\ast, \omega}(\sA, \sB)$ denotes the category of sequences $(F_n)_{n \geq 0}$ where the $F_n$ for $n \geq 1$ are as above, and $F_0$ is an object of $\sB$.

The advantage of the categories of the form $\SSeq(\Pfin(\sA), \Pfin(\sB))$ over those of the form $\SymFun^{\ast, \omega}(\sA, \sB)$ is that the former can be assembled into a 2-precategory with a composition product, whereas this is not true for the latter.
This $2$-precategory is constructed in \cite[Definition 3.3.8]{blansblom2025chainrulegoodwilliecalculus}. 
We denote it by $\fdiffsp$.
Its objects are stable differentiable categories, and the mapping category from $\sA$ to $\sB$ in $\fdiffsp$ is given by $\SSeq(\Pfin(\sA), \Pfin(\sB))$.
Morphisms in this mapping category are called formal symmetric sequences from $\sA$ to $\sB$.
We will write $\sF$ for the smallest locally full subcategory of $\fdiffsp$ such that each of the mapping categories $\sF(\sA, \sB)$ contains the image of the functor
\[
\iota \colon \SymFun^{\ast, \omega}_{\geq 1}(\sA, \sB) \to \SSeq_{\geq 1}(\Pfin(\sA), \Pfin(\sB)) \subseteq \SSeq(\Pfin(\sA), \Pfin(\sB)).
\]
More concretely, morphisms in $\sF$ are finite composites of positive formal symmetric sequences lying in the image of the Yoneda embedding.

The construction of the lax derivatives functor proceeds by writing down the following functors of $2$-precategories:
\[
\begin{tikzcd}[sep = large]
    \diff_{\mathrm{St}} \ar[r, "\widetilde{\cross}"', shift right] & \sF \ar[l, shift right, "\widetilde{\Lambda}"'] \ar[r, "\widetilde{\mlin}", shift left] & \pressymstc \ar[l, "k", shift left]
\end{tikzcd}
\]
These functors can be described as follows.
First of all, $\widetilde{\cross}$ and $k$ are lax functors, $\widetilde{\Lambda}$ is an oplax functor, and $\widetilde{\mlin}$ is a strong functor.
The functors $\widetilde{\cross}$ and $\widetilde{\mlin}$ are the identity on objects and induce the functors with the same name on mapping categories.
The derivatives functor $\partial_*$ is defined as the composite $\widetilde{\mlin} \circ \widetilde{\cross}$.
The functor $\widetilde{\Lambda}$ is also the identity on objects and induces the left adjoint of $\widetilde{\cross}$ on mapping categories.
We therefore say that $\widetilde{\Lambda}$ is the local left adjoint of $\widetilde{\cross}$.
The functor $k$ is the local right adjoint of $\widetilde{\mlin}$; it is fully faithful on mapping categories.
See \cite[Section 3.3]{blansblom2025chainrulegoodwilliecalculus} for full details on these constructions.

We now turn to the proof of \cref{prop: app-derivatives-lambda}.
We will need a 2-categorical concept called an \emph{icon}, originally introduced by Lack in the setting of $(2,2)$-categories \cite{Lack2010Icons}.
Recall from \cite[Definition A.3.9]{blansblom2025chainrulegoodwilliecalculus} that a lax functor $F \colon \sX \to \sY$ between 2-precategories is defined as a map $F \colon \un \sX \to \un \sY$ in $\Cat_{/\Delta^\op}$ that preserves cocartesian lifts of inerts.
Here $\un$ denotes the cocartesian unstraightening.
Dually, an oplax functor $F \colon \sX \to \sY$ is defined as a map $F \colon \un^\vee \sX \to \un^\vee \sY$ in $\Cat_{/\Delta}$ between cartesian unstraightenings that preserves cartesian lifts of inerts.

\begin{definition}\label{def:icons}
    Let $F,G \colon \sX \to \sY$ be lax functors.
    Then an \emph{icon} from $F$ to $G$ is defined as a 2-morphism between $F \colon \un \sX \to \un \sY$ and $G \colon \un \sX \to \un \sY$ in $\Cat_{/\Delta^\op}$.
    We will write $\LaxFun^{\odot}(\sX,\sY)$ for the category of lax functors and icons between them, defined as the full subcategory of
    $\Fun_{/\Delta^\op}(\un \sX,\un \sY)$
    spanned by the lax functors.

    If instead $F$ and $G$ are oplax functors, then an \emph{icon} from $F$ to $G$ is a 2-morphism between $F \colon \un^\vee \sX \to \un^\vee \sY$ and $G \colon \un^\vee \sX \to \un^\vee \sY$ in $\Cat_{/\Delta}$, and we define the category $\OplaxFun^{\odot}(\sX,\sY)$ of oplax functors and icons analogously.

    We will write $\Fun^{\odot}(\sX,\sY)$ for the further full subcategory of strong functors and icons between them.
    Note that it does not matter whether we define this as a full subcategory of $\OplaxFun^{\odot}(\sX,\sY)$ or $\LaxFun^{\odot}(\sX,\sY)$, as both full subcategories are canonically equivalent by straightening-unstraightening.
\end{definition}

\begin{remark}
    The term icon stands for Identity Component Oplax Natural transformation, and can be understood as follows.
    Suppose that $F, G \colon \sX \to \sY$ are lax functors of $2$-(pre)categories.
    Part of the data of an oplax natural transformation from $F$ to $G$ consists of an oplax naturality square
    \[
    \begin{tikzcd}
        F(x) \ar[r] \ar[d] & F(y) \ar[d] \\
        G(x) \ar[r] \ar[ur, Rightarrow] & G(y).
    \end{tikzcd}
    \]
    for every $1$-morphism $x \to y$ in $\sX$.
    An oplax natural transformation is an icon precisely if the vertical 1-morphisms in all these naturality squares are equivalences.
    Also note that an oplax natural transformation is an ordinary natural transformation precisely if the $2$-morphisms in all these naturality squares are equivalences, so that both icons and natural transformations are special instances of the notion of oplax natural transformation.
\end{remark}

\begin{example}
    Suppose that $L \colon \sX \rightleftarrows \sY \colon R$ is a local adjunction between $2$-precategories, where $L$ is a strong functor.
    Then there are icons
    \[
    \eta \colon \id_\sX \Rightarrow RL \qquad \epsilon \colon LR \Rightarrow \id_\sY
    \]
    inducing the units and counits of the corresponding adjunctions on mapping categories.
\end{example}

\begin{remark}
    Suppose that $\eta \colon F \Rightarrow G$ is an icon between two functors from $\sX$ to $\sY$, such that for every morphism $s \colon a \to b$ in $\sX$ the component $\eta_s \colon F(s) \to G(s)$ is an equivalence.
    Then $\eta$ is a natural equivalence from $F$ to $G$.
\end{remark}

We will construct the equivalence between $\partial_* \Lambda$ and the identity functor of $\pressymstc$ as the composite of two icons:
\[
\begin{tikzcd}
\id_{\pressymstc} & \widetilde{\mlin} k \ar[l, "\epsilon"', "\sim", Rightarrow] \ar[r, "\eta", Rightarrow] & \widetilde{\mlin} \widetilde{\cross} \widetilde{\Lambda} k.
\end{tikzcd}
\]
Here, $\epsilon$ is the counit of the local adjunction $(\widetilde{\mlin}, k)$, which is an equivalence since $k$ is locally fully faithful.
The icon $\eta$, which should be given by the unit of the adjunction $(\widetilde{\Lambda}, \widetilde{\cross})$ on mapping categories, is a little more complicated to write down, as its target does not seem to be well-defined: $\widetilde{\Lambda}$ is an oplax functor and $k$ is a lax functor, so it is not obvious how to compose them.
It however turns out that $\widetilde{\Lambda}$ is a strong functor on the essential image of $k$, so that the composite $\widetilde{\Lambda} \circ k$ does make sense.
Moreover, $\widetilde{\Lambda} \circ k$ turns out to be a strong functor equivalent to $\Lambda$.
Since $\widetilde{\mlin} \circ \widetilde{\cross} \simeq \partial_*$, it will then suffice to show that $\eta$ is an equivalence to complete the proof of the proposition.

We start by proving the equivalence $\widetilde{\Lambda}k \simeq \Lambda$.

\begin{lemma} \label{lem: lambda-is-lambda'-k}
    Write $\im k$ for the smallest locally full subcategory of $\sF$ containing the image of $k$.
    The restriction of the oplax functor $\widetilde{\Lambda} \colon \sF \to \diffsp$ to $\im k$ is a strong functor, and the composite $\widetilde{\Lambda}k$ is equivalent to $\Lambda \colon \pressymstc \to \diffsp$.
\end{lemma}
\begin{proof}
    That the restriction of $\widetilde{\Lambda}$ to $\im k$ is a strong functor follows immediately from \cite[Remark 3.3.14]{blansblom2025chainrulegoodwilliecalculus}.
    Indeed, the map $\delta$ discussed in that remark is the oplax comparison map of $\widetilde{\Lambda}$, since this functor is the local left adjoint of $\widetilde{\cross}$.
    The category $\SymFun^{\mathrm{L}}(\sD, \sE)$ is the category of symmetric sequences from $\sD$ to $\sE$, which we denote by $\SSeq(\sD, \sE)$ in this paper, so that the final line of the remark says precisely that this comparison map is an equivalence on the image of $k$. 
    The same remark also implies that the composite 
    \[
    \widetilde{\Lambda} \circ k \colon \pressymstc \to \diffsp
    \]
    is a strong functor.
    
    We would like to construct a natural transformation from $\Lambda$ to $\widetilde{\Lambda}k$ by using the fact that $\Lambda \colon \pressymst \to \Cat$ is corepresented by the point.
    Note however that after restricting the domain of $\Lambda$ to $\pressymstc$, it is no longer corepresentable.
    We will work around this problem by showing that $\widetilde{\Lambda} \circ k$ can be extended to a functor on $\pressymst$.
    
    Inspecting the proof of \cite[Proposition 3.3.13]{blansblom2025chainrulegoodwilliecalculus}, we see that the construction of $\widetilde{\Lambda}$ can also be carried out for the $2$-precategory $\fdiffsp$ and not just its subcategory $\sF$.
    In this way, we obtain a functor $\Lambda^{\mathrm{FSym}} \colon \fdiffsp \to \Cat_{\mathrm{St}}$, where we write $\Cat_{\mathrm{St}}$ for the full subcategory of the $2$-category $\Cat$ on the stable presentable categories.
    This functor induces the composite
    \[
    \SSeq(\Pfin(\sA), \Pfin(\sB)) \xrightarrow{L} \SymFun^{\ast, \omega}(\sA, \sB) \to \Fun(\sA, \sB),
    \]
    on mapping categories, where the final arrow is the left adjoint of the symmetric cross-effects functor.
    The restriction of $\Lambda^{\mathrm{FSym}}$ to $\sF$ is equivalent to $\widetilde{\Lambda}$ by construction.

    For each pair of stable presentable categories $\sA$ and $\sB$, we have a localization functor given by the composite
    \[
    \SSeq(\Pfin(\sA), \Pfin(\sB)) \xrightarrow{L} \SymFun^{\ast, \omega}(\sA, \sB) \xrightarrow{\mlin} \SSeq(\sA, \sB).
    \]
    By an argument analogous to the proof of \cite[Proposition 3.3.17]{blansblom2025chainrulegoodwilliecalculus}, we see that these form a weakly compatible family of localizations in the sense of \cite[Remark A.3.28]{blansblom2025chainrulegoodwilliecalculus}, so that they assemble to an oplax functor $\mlin^{\mathrm{FSym}} \colon \fdiffsp \to \pressymst$.
    That the target of this functor can be identified with $\pressymst$ follows by repeating the argument from \cite[Section 3.3.4]{blansblom2025chainrulegoodwilliecalculus}.
    Passing to the local right adjoint of $\mlin^{\mathrm{FSym}}$ gives a lax functor $k^{\mathrm{FSym}} \colon \pressymst \to \fdiffsp$, which induces the fully faithful functor
    \[
    \SSeq(\sA, \sB) \hookrightarrow \SymFun^{\ast, \omega}(\sA, \sB) \xhookrightarrow{\iota} \SSeq(\Pfin(\sA), \Pfin(\sB))
    \]
    on mapping categories.
    The restriction of $k^{\mathrm{FSym}}$ to $\pressymstc$ is equivalent to $k$ by construction.
    
    The oplax comparison maps of $\Lambda^{\mathrm{FSym}} \colon \fdiffsp \to \Cat_{\mathrm{St}}$ admit the same description as those of $\widetilde{\Lambda}$, which are given in \cite[Remark 3.3.14]{blansblom2025chainrulegoodwilliecalculus}, so that it again follows that the restriction of $\Lambda^{\mathrm{FSym}}$ to the smallest locally full subcategory of $\fdiffsp$ containing the essential image of $k^{\mathrm{FSym}}$ is a strong $2$-functor, and that the composite $\Lambda^{\mathrm{FSym}} \circ k^{\mathrm{FSym}}$ is strong as well.

    Since $\Lambda \colon \pressymst \to \Cat_{\mathrm{St}}$ is corepresented by the point, the $2$-categorical Yoneda lemma \cite{Hinich2020YonedaLemmaEnriched} provides a natural transformation 
    \[
    \eta \colon \Lambda \Rightarrow \Lambda^{\mathrm{FSym}} \circ k^{\mathrm{FSym}}
    \]
    corresponding to the unique object in $\Lambda^{\mathrm{FSym}} \circ k^{\mathrm{FSym}}(\ast) \simeq \ast$.
    Both $\Lambda$ and $\Lambda^{\mathrm{FSym}} \circ k^{\mathrm{FSym}}$ act as the identity on objects, and it is easily verified that the component of $\eta$ on a stable presentable category $\sA$ is the identity functor.
    Hence $\eta$ is a natural equivalence $\Lambda \simeq \Lambda^{\mathrm{FSym}} \circ k^{\mathrm{FSym}}$.
    Restricting to $\pressymstc$, we obtain
    $\Lambda \simeq \widetilde{\Lambda} \circ k$.
\end{proof}

The following technical lemma will allow us to construct the icon $\eta$.

\begin{lemma} \label{lem: icon-mate}
    Suppose we are given functors of $2$-precategories
        \[
        \begin{tikzcd}[sep = large]
            \sX \ar[r, "R"', shift right] & \sY \ar[l, "L"', shift right] & \sZ \ar[l, "F"']
        \end{tikzcd}
        \]
        such that
        \begin{enumerate}[\upshape{(}1\upshape{)}]
            \item $R$ and $F$ are lax functors that both induce an equivalence on the space of objects;
            \item $R$ induces right adjoints on mapping categories, and $L$ is the oplax functor locally left adjoint to $R$;
            \item The restriction of $L$ to the smallest locally full subcategory of $\sY$ containing the image of $F$ is a strong functor, and the composite $LF$ is strong.
        \end{enumerate}
        Then for any strong functor $G \colon \sZ \to \sX$ and any icon $\alpha \colon LF \Rightarrow G$, there exists an icon $\beta \colon F \Rightarrow RG$ such that on mapping categories, $\beta$ is obtained by taking the mate of $\alpha$.
\end{lemma}

\begin{remark}\label{remark:what-are-mates}
    The conclusion of this lemma means that for any 1-morphism $s \colon a \to b$ in $\sZ$, the 2-morphism $\beta_s \colon F(s) \Rightarrow RG(s)$ is the mate of $\alpha_s \colon LF(s) \Rightarrow G(s)$ under the adjunction
    \[
    L : \sY(a, b) \rightleftarrows \sX(a, b) : R.
    \]
    That is, it is the composite $F(s) \xRightarrow{\eta_{F(s)}} RLF(s) \xRightarrow{R(\alpha_s)} RG(s)$.
\end{remark}

Since the proof of \cref{lem: icon-mate} is somewhat involved, we postpone it to \cref{ssec:proof-icon-mate}.

\begin{proof}[Proof of \cref{prop: app-derivatives-lambda}]
    Applying \cref{lem: icon-mate} to the functors
    \[
    \begin{tikzcd}[sep = large]
        \diff_{\St} \ar[r, "\widetilde{\cross}"', shift right] & \sF \ar[l, "\widetilde{\Lambda}"', shift right] & \pressymstc \ar[l, "k"']
    \end{tikzcd}
    \]
    and the identity icon of $\widetilde{\Lambda}k$, we obtain an icon $\eta \colon k \Rightarrow \widetilde{\cross}\widetilde{\Lambda}k$ that is given by the unit of the adjunction $(\widetilde{\cross}, \widetilde{\Lambda})$ on mapping categories.
    We claim that $\widetilde{\mlin} \eta \colon \widetilde{\mlin}k \Rightarrow \widetilde{\mlin}\widetilde{\cross}\widetilde{\Lambda}k$ is an equivalence.
    To prove the claim, it suffices to show that for every $\sA, \sB \in \pressymstc$ and $X \in \SSeq_{\geq 1}(\sA, \sB)$, the unit
    \[
    k(X) \Rightarrow \widetilde{\cross}\widetilde{\Lambda}k(X)
    \]
    of the adjunction $\widetilde{\Lambda} \dashv \widetilde{\cross}$ becomes an equivalence after applying $\widetilde{\mlin}$.
    Now $\widetilde{\cross} = \iota \circ \cross$ and $\widetilde{\mlin} = \mlin \circ L$, where we use the notation of \cref{eq: lax-derivative-cross-mlin}.
    Writing
    \[
    j \colon \SSeq_{\geq 1}(\sA, \sB) \hookrightarrow \SymFun^{\ast, \omega}_{\geq 1}(\sA, \sB),
    \]
    for the inclusion, we also have $k = \iota \circ j$.
    Moreover, the symmetric cross-effects functor has a left adjoint
    \[
    p_! \colon \SymFun^{\ast, \omega}_{\geq 1}(\sA, \sB) \to \Fun^{\ast, \omega}(\sA, \sB),
    \]
    and $\widetilde{\Lambda} = p_! \circ L$, as follows from \cite[Proposition 3.3.11]{blansblom2025chainrulegoodwilliecalculus}.
    The upshot is that it suffices to show that the unit 
    \[
    j(X) \Rightarrow \cross p_!j(X)
    \]
    of the adjunction $p_! \dashv \cross$ becomes an equivalence after applying the multilinearization functor $\mlin \colon \SymFun^{\ast, \omega}_{\geq 1}(\sA, \sB) \to \SSeq_{\geq 1}(\sA, \sB)$.
    This is proved in \cite[Lemma 3.4.5]{blansblom2025chainrulegoodwilliecalculus} for symmetric sequences concentrated in finitely many arities.
    But since both source and target commute with filtered colimits, and every symmetric sequence can be written as a filtered colimit of its truncations, it holds for arbitrary symmetric sequences.
    This proves the claim.

    Now consider the following span of icons
    \[
    \begin{tikzcd}
    \id_{\pressymstc} & \widetilde{\mlin} k \ar[l, "\epsilon"', Rightarrow] \ar[r, "\widetilde{\mlin}\eta", Rightarrow] & \widetilde{\mlin} \widetilde{\cross} \widetilde{\Lambda} k,
    \end{tikzcd}
    \]
    where $\epsilon$ is given by the counit of the local adjunction $(\widetilde{\mlin}, k)$.
    Both maps are equivalences: in case of $\epsilon$ this follows since $\widetilde{\mlin}$ is locally a reflective localization and for $\widetilde{\mlin}\eta$ we just proved this.
    Since $\widetilde{\mlin}\widetilde{\cross} = \partial_*$ by definition and $\widetilde{\Lambda}k \simeq \Lambda$ by \cref{lem: lambda-is-lambda'-k}, this completes the proof.
\end{proof}

\subsection{Proof of \texorpdfstring{\cref{lem: icon-mate}}{the icon lemma}} \label{ssec:proof-icon-mate}

We will use \emph{envelopes} or \emph{lax functor classifiers} to reduce to the case where $F$ is a strong functor.
These were also constructed in, for example, \cite[Chapter 11.A]{GaitsgoryRozenblyum2017Volume1} and \cite{glossnerModelIndependentUniversal2025}.
However, since to our knowledge there is no reference that spells out their effect on icons, we've decided to include a proof.

\begin{lemma}\label{lem:envelope-of-2cat}
    Let $\sX$ be a 2-precategory.
    Then there exists a 2-precategory $\Env(\sX)$ together with a map $\eta \colon \sX \to \Env(\sX)$ such that for any 2-precategory $\sY$, the restriction map
    \[
    \Fun^\odot(\Env(\sX),\sY) \xrightarrow{\eta^*} \LaxFun^\odot(\sX,\sY) 
    \]
    is an equivalence.
\end{lemma}

\begin{proof}
    Let $\sX$ and $\sY$ be 2-precategories and let $\Ar_\mathrm{act}\Delta^\op \subset \Ar\Delta^\op$ denote the full subcategory spanned by the active morphisms.
    By \cite[Proposition 2.2.4]{BarkanHaugsengea2025EnvelopesAlgebraicPatterns}, $\un \sX \times_{\Delta^\op} \Ar_\mathrm{act}\Delta^\op \to \Delta^\op$ is a cocartesian fibration with the property that the inclusion
    \[\un \sX \hookrightarrow \un \sX \times_{\Delta^\op} \Ar_\mathrm{act}\Delta^\op\]
    induces an equivalence
    \[\Fun_{/\Delta^\op}^\mathrm{cocart}(\un \sX \times_{\Delta^\op} \Ar_\mathrm{act}\Delta^\op, \un \sY) \to \Fun_{/\Delta^\op}^\mathrm{inert}(\un \sX, \un \sY) = \LaxFun^\odot(\sX,\sY),\]
    where $\Fun_{/\Delta^\op}^\mathrm{inert}$ (resp.\  $\Fun_{/\Delta^\op}^\mathrm{cocart}$) denote the categories of functors over $\Delta^\op$ that preserve cocartesian lifts of inerts (resp.\ all cocartesian lifts).
    We therefore define $\Env(\sX)$ as the straightening of $\un \sX \times_{\Delta^\op} \Ar_\mathrm{act}\Delta^\op$; to conclude the result we need to show that $\Env(\sX) \colon \Delta^\op \to \Cat$ is indeed a 2-precategory.
    It is clear from the construction that $\Env(\sX)_0 \simeq \sX_0$, which is a space, while the Segal condition follows from \cite[Example 4.3.4]{BarkanHaugsengea2025EnvelopesAlgebraicPatterns}, noting that $\un \sX$ is a generalized nonsymmetric $\infty$-operad.\footnote{The Segal condition can also be proved directly since $\Delta_{[n]/}^\mathrm{act} \simeq \Delta_{[1]/}^\mathrm{act} \times \cdots \times \Delta_{[1]/}^\mathrm{act}$.}
\end{proof}

Using the envelope construction, we will be able to reduce \cref{lem: icon-mate} to the following version of the mate correspondence.

\begin{lemma}\label{lem:general-icon-mate-construction}
    Let $\sX$, $\sY$, $\sZ$ and $\sW$ be 2-precategories, let $F \colon \sZ \to \sX$ and $G \colon \sW \to \sY$ be strong functors and let $L \colon \sX \to \sY$ and $L' \colon \sZ \to \sW$ be oplax functors that admit local right adjoints $R$ and $R'$ in the sense of \cite[Proposition A.3.16]{blansblom2025chainrulegoodwilliecalculus}.
    Then there is an equivalence between the space of lax commuting squares of the form
    \begin{equation}\label{eq:mate-squares-lax-oplax}\begin{tikzcd}[ampersand replacement=\&]
        \sZ \& \sW \\
        \sX \& \sY
        \arrow["{L'}", from=1-1, to=1-2]
        \arrow["F"', from=1-1, to=2-1]
        \arrow["G", from=1-2, to=2-2]
        \arrow[Rightarrow, shorten = 0.7em, from=2-1, to=1-2]
        \arrow["L"', from=2-1, to=2-2]
    \end{tikzcd}
    \qquad \text{and} \qquad
    \begin{tikzcd}[ampersand replacement=\&]
        \sZ \& \sW \\
        \sX \& \sY
        \arrow["F"', from=1-1, to=2-1]
        \arrow[Rightarrow, shorten = 0.7em, from=1-1, to=2-2]
        \arrow["{R'}"', from=1-2, to=1-1]
        \arrow["G", from=1-2, to=2-2]
        \arrow["R", from=2-2, to=2-1]
    \end{tikzcd}\end{equation}
    In other words, there is an equivalence
    \[\Map_{\OplaxFun^\odot}(LF,GL') \simeq \Map_{\LaxFun^\odot}(FR',RG).\]
    On the mapping categories of $\sX$ and $\sY$, this equivalence is given by the usual mate correspondence (cf.\ \cref{remark:what-are-mates}).
\end{lemma}

\begin{proof}
    Write $\PreCattwo^\lax$ and $\PreCattwo^\oplax$ for the 2-categories of 2-precategories, (op)lax functors and icons; more precisely, these are defined as locally full sub-2-categories of $\Cat_{/\Delta^\op}$ and $\Cat_{/\Delta}$, respectively (cf.\ \cref{def:icons}).
    Write $\Ar^\lax(\PreCattwo^\lax)^\mathrm{str} \subset \Ar^\lax(\PreCattwo^\lax)$ and $\Ar^\oplax(\PreCattwo^\oplax)^\mathrm{str} \subset \Ar^\oplax(\PreCattwo^\oplax)$ for the full subcategories whose objects are those arrows that are strong functors between 2-precategories.
    Here $\Ar^\lax$ and $\Ar^\oplax$ denote the (op)lax arrow categories, see e.g. \cite[Notation 2.4.1]{AbellanGagnaea2025StraighteningLaxTransformations}.
    It follows by definition that we may identify squares on the left and right of \cref{eq:mate-squares-lax-oplax} with 1-morphisms in $\Ar^\oplax(\PreCattwo^\oplax)^\mathrm{str}$ and $\Ar^\lax(\PreCattwo^\lax)^\mathrm{str}$ that are componentwise left and right adjoints, respectively.
    Let us write $\Ar^\oplax(\PreCattwo^\oplax)^{\mathrm{str},L}$ and $\Ar^\lax(\PreCattwo^\lax)^{\mathrm{str},R}$ for the wide and locally full subcategories spanned by these 1-morphisms.
    The result will therefore follow if we prove an equivalence
    \[\Ar^\oplax(\PreCattwo^\oplax)^{\mathrm{str},L} \simeq (\Ar^\lax(\PreCattwo^\lax)^{\mathrm{str},R})^{\coop}.\]
    
    This will be a consequence of the following claim:

    \begin{claim*}
    $\Ar^\lax(\PreCattwo^\lax)^\mathrm{str}$ is equivalent to the locally full subcategory of $\Cat_{/[1] \times \Delta^\op}$ such that
    \begin{itemize}
        \item its objects are cocartesian fibrations $\cC \to [1] \times \Delta^\op$ such that for both $i=0,1$, the fiber $\cC_i \to \Delta^\op$ is (the unstraightening of) a 2-precategories, and
        \item its 1-morphisms are functors over $[1] \times \Delta^\op$ that preserve cocartesian lifts of inerts, where a map in $[1] \times \Delta^\op$ is called inert if it is of the form $(\id_i, \alpha)$ with $i \in \{0,1\}$ and $\alpha$ inert in $\Delta^\op$.
    \end{itemize}
    Similarly, $\Ar^\oplax(\PreCattwo^\oplax)^\mathrm{str}$ is equivalent to the locally full sub-2-category of $\Cat_{[1]^\op \times \Delta}$ satisfying the same properties, but with ``cocartesian'' replaced by ``cartesian''.
    \end{claim*}
    
    By exactly the same argument as \cite[Proposition A.3.16]{blansblom2025chainrulegoodwilliecalculus}\footnote{Which is essentially just the argument of \cite[Theorem 3.4.7]{linskensLaxMonoidalAdjunctions}.}, with $\Delta^\op$ replaced by $[1] \times \Delta^\op$, it now follows that there is an equivalence between $\Ar^\oplax(\PreCattwo^\oplax)^{\mathrm{str},L}$ and $(\Ar^\lax(\PreCattwo^\lax)^{\mathrm{str},R})^\coop$, and by (the dual of) \cite[Proposition 3.2.7]{linskensLaxMonoidalAdjunctions} it is given by the mate correspondence.

    We will now prove the claim in the lax case, the other case is similar (and formally dual).
    By \cite[Corollary 2.8.7]{AbellanGagnaea2025StraighteningLaxTransformations}, there is a fully faithful inclusion
    \[\Ar^\lax(\Cat_{/\Delta^\op}) \hookrightarrow \Ar^\lax(\Cat)_{/\const_{\Delta^\op}},\]
    whose essential image consists of the strongly commuting squares
    \[\begin{tikzcd}[ampersand replacement=\&]
        \cC \& \cD \\
        {\Delta^\op} \& {\Delta^\op.}
        \arrow[""{name=0, anchor=center, inner sep=0}, from=1-1, to=1-2]
        \arrow[from=1-1, to=2-1]
        \arrow[from=1-2, to=2-2]
        \arrow[""{name=1, anchor=center, inner sep=0}, equals, from=2-1, to=2-2]
        \arrow["\circlearrowright"{description}, draw=none, from=0, to=1]
    \end{tikzcd}\]
    By \cite[Theorem 5.3.1]{linskensLaxMonoidalAdjunctions}, there is an equivalence between $\Ar^\lax(\Cat)$ and the full subcategory of $\Cat_{/[1]}$ spanned by the cocartesian fibrations; combining these equivalences, it follows that under straightening-unstraightening, $\Ar^\lax(\Cat_{/\Delta^\op})$ is equivalent to the full subcategory of $\Cat_{/[1] \times \Delta^\op}$ spanned by the functors $\cC \to [1] \times \Delta^\op$ that are cocartesian fibrations over $[1]$ in the sense of \cite[Definition 2.2.3]{linskensLaxMonoidalAdjunctions}.
    Writing $\mathrm{Cocart}^\lax(\Delta^\op)$ for the full subcategory of $\Cat_{/\Delta^\op}$ spanned by the cocartesian fibrations, it follows that $\Ar^\lax(\mathrm{Cocart}^\lax(\Delta^\op))$ is equivalent to the full subcategory of $\Cat_{/[1] \times \Delta^\op}$ spanned by the Gray fibrations in the sense of \cite[Definition 2.4.1]{linskensLaxMonoidalAdjunctions}.
    It now follows using the equivalence of items (1) and (5) in \cite[Proposition 2.4.9]{linskensLaxMonoidalAdjunctions} that $\Ar^\lax(\PreCattwo^\lax)^\mathrm{str}$ is the locally full subcategory of $\Cat_{/[1] \times \Delta^\op}$ described in the claim.
\end{proof}

\begin{proof}[Proof of \cref{lem: icon-mate}]
    Write $\widetilde{F}$ and $\widetilde{G}$ for the strong 2-functors corresponding to $F$ and $G$ under the equivalence $\Fun^\odot(\Env(\sZ),-) \simeq \LaxFun^\odot(\sZ,-)$ from \cref{lem:envelope-of-2cat}.
    Applying \cref{lem:general-icon-mate-construction} with $L' = \id_{\Env(\sZ)}$, it follows that there is an equivalence between spaces of icons
    \[\Map_{\OplaxFun^\odot}(L\widetilde{F}, \widetilde{G}) \simeq \Map_{\LaxFun^\odot}(\widetilde{F}, R\widetilde{G}),\]
    which on mapping categories is the usual mate correspondence.
    We first note that $L \widetilde{F}$ is in fact a strong functor.
    To see this, let $\sY' \subset \sY$ be the smallest locally full subcategory of $\sY$ containing the image of $F$, and let $F' \colon \sZ \to \sY'$ and $L' \colon \sY' \to \sX$ be the (co)restrictions of $F$ and $L$, respectively.
    By naturality, the diagram
    \[\begin{tikzcd}[ampersand replacement=\&]
        \& {\sY'} \& \\
        {\Env(\sZ)} \& \sY \& \sX
        \arrow[hook, from=1-2, to=2-2]
        \arrow["{L'}", from=1-2, to=2-3]
        \arrow["{\widetilde{F'}}", from=2-1, to=1-2]
        \arrow["{\widetilde{F}}"', from=2-1, to=2-2]
        \arrow["L"', from=2-2, to=2-3]
    \end{tikzcd}\]
    commutes.
    Since $L'$ is strong by assumption, it follows that $L \widetilde{F} = L' \widetilde{F'}$ is strong.
    We therefore see that $\Map_{\OplaxFun^\odot}(L\widetilde{F}, \widetilde{G}) = \Map_{\Fun^\odot}(L\widetilde{F}, \widetilde{G})$, so we may form the composite
    \begin{align*}
        \Map_{\LaxFun^\odot}(LF,G) &\xrightarrow{(\eta^*)^{-1}} \Map_{\Fun^\odot}(L\widetilde{F}, \widetilde{G}) \\
         &\simeq \Map_{\LaxFun^\odot}(\widetilde{F}, R\widetilde{G}) \xrightarrow{\eta^*} \Map_{\LaxFun^\odot}(F,RG)
    \end{align*}
    Here we use that $\widetilde{F}\eta = F$ and $\widetilde{G}\eta = G$ by definition.
    This gives the desired construction of an icon $\beta \colon F \Rightarrow RG$ from an icon $\alpha \colon LF \Rightarrow G$.

    Now let
    \[\begin{tikzcd}[ampersand replacement=\&]
        \& {\sY(a,b)} \& \\
        {\sZ(a,b)} \&\& {\sX(a,b)}
        \arrow["{L_{a,b}}", from=1-2, to=2-3]
        \arrow["{F_{a,b}}", from=2-1, to=1-2]
        \arrow[""{name=0, anchor=center, inner sep=0}, "{G_{a,b}}"', from=2-1, to=2-3]
        \arrow["{\alpha_{a,b}}", shorten=0.5em, Rightarrow, from=1-2, to=0]
    \end{tikzcd}\]
    be the component of the icon $\alpha$ at two objects $a$ and $b$ of $\sZ$.
    Then it follows by construction that $\beta_{a,b}$ is given by the composite
    \[F_{a,b} \xRightarrow{u} L_{a,b}R_{a,b}F_{a,b} \xRightarrow{R_{a,b}(\alpha_{a,b})} R_{a,b}G_{a,b},\]
    where $u$ denotes the unit of the adjunction $L_{a,b} \dashv R_{a,b}$.
    This composite is precisely the mate of $\alpha_{a,b}$.
\end{proof}

\begin{remark}
    In fact, the map $\eta^* \colon \Map_{\LaxFun^\odot}(\widetilde{F}, R\widetilde{G}) \rightarrow \Map_{\LaxFun^\odot}(F,RG)$ appearing in this proof is also an equivalence.
    In particular, it follows that there is an equivalence between the space of icons $LF \Rightarrow G$ and the space of icons $F \Rightarrow RG$, strengthening the conclusion of \cref{lem: icon-mate}.
    This is not directly clear from the universal property of $\Env(\sZ)$, since $R\widetilde{G}$ is not a strong functor.
    However, one can show that $\eta^*$ is part of a coreflective adjunction
    \[\begin{tikzcd}[ampersand replacement=\&]
        {\LaxFun^\odot(\Env(\sZ),\sX)} \&\& {\LaxFun^\odot(\sZ,\sX)}
        \arrow[""{name=0, anchor=center, inner sep=0}, "{\eta^*}"', shift left=-2, from=1-1, to=1-3]
        \arrow[""{name=1, anchor=center, inner sep=0}, "{\eta_!}"', shift left=-2, hook', from=1-3, to=1-1]
        \arrow["\dashv"{anchor=center, rotate=-90}, draw=none, from=0, to=1]
    \end{tikzcd}\]
    such that the colocal objects are precisely the strong functors $\Env(\sZ) \to \sX$.
    Since $\widetilde{F}$ is strong, it follows that
    \[\Map_{\LaxFun^\odot}(\widetilde{F},R\widetilde{G}) = \Map_{\LaxFun^\odot}(\eta_! F, R\widetilde{G}) \xrightarrow{\eta^*} \Map_{\LaxFun^\odot}(F,RG)\]
    is an equivalence.
    As we don't need this stronger version of \cref{lem: icon-mate} here, we will not prove this.
\end{remark}

\section{Derivatives and the norm}

Let $\LSp$ be a stable presentable localization of $\Sp$, as defined in \cref{ssec:localizations-of-spectra}.
Since this category is stable, the derivatives functor
\[
\partial_* \colon \End^{\ast, \omega}(\LSp) \to \SSeq_{\geq 1}(\LSp)
\]
preserves colimits, so that it admits a right adjoint by the adjoint functor theorem.
It is the purpose of this appendix to record a proof of the well-known fact that the unit of this adjunction on an $n$-homogeneous functor $H$ is given by the norm map
\[
(\partial_nH \otimes \id_{\LSp}^{\otimes n})_{h\Sigma_n} \Rightarrow (\partial_nH \otimes \id_{\LSp}^{\otimes n})^{h\Sigma_n}.
\]

As the previous display shows, the formula for this right adjoint involves homotopy fixed points, which do not commute with filtered colimits.
We therefore need to deal with derivatives of non-finitary functors.
For $n \geq 1$, we let $\SymFunlin_n(\LSp)$ denote the full subcategory of $\Fun(\LSp^{\times n}_{h\Sigma_n}, \LSp)$ spanned by those functors that are linear in each variable separately; recall that a functor between stable categories is called linear if it is reduced and exact.
We write 
\[
\SymFunlin(\LSp) = \prod_{n \geq 1} \SymFunlin_n(\LSp).
\]
Note that $\SSeq_{\geq 1}(\LSp)$ is equivalent to the full subcategory of $\SymFunlin(\LSp)$ spanned by the sequences $(A_n)_{n \geq 1}$ for which each $A_n$ preserves filtered colimits. 
Given $A \in \SymFunlin_n(\LSp)$, we write $A \circ \Delta \in \End^{\ast}(\LSp)^{B\Sigma_n}$ for the restriction of $A$ along the diagonal $\Delta \colon \LSp \times B\Sigma_n \to \LSp^{\times n}_{h\Sigma_n}$.
The functor $(A \circ \Delta)_{h\Sigma_n} \colon \LSp \to \LSp$ is $n$-homogeneous.
The derivatives construction, defined as multilinearized cross-effects, defines a functor
\[
\partial_* \colon \End^{\ast}(\LSp) \to \SymFunlin(\LSp)
\]
by \cite[Remark 6.1.3.23 and Theorem 6.1.4.7]{HA}.
The restriction of this functor to $\End^{\ast, \omega}(\LSp)$ factors through $\SSeq_{\geq 1}(\LSp)$, and is equivalent to the derivatives functor as we have been using it throughout this paper.

Consider the functor $\Psi \colon \SymFunlin(\LSp) \to \End^{\ast}(\LSp)$, defined by the formula
\[
\Psi(A) = \prod_{n \geq 1}(A_n \circ \Delta)^{h\Sigma_n}.
\]
We will prove that $\Psi$ is right adjoint to $\partial_*$.

Suppose that $F \colon \LSp \to \LSp$ is a reduced functor.
Recall that, for instance by \cite[Remark 6.1.6.29]{HA}, we have for each $n \geq 1$ a natural pullback square of the form
\[
\begin{tikzcd}
    P_nF \ar[r, "\eta_n"] \ar[d] & (\partial_nF \circ \Delta)^{h\Sigma_n} \ar[d] \\
    P_{n-1}F \ar[r] & (\partial_nF \circ \Delta)^{t\Sigma_n},
\end{tikzcd}
\]
where the right vertical arrow is the canonical map from the fixed points to the Tate construction.
These are called the Tate squares.

By definition we have a product decomposition $\Psi \simeq \prod_{n \geq 1} \Psi_n$, with 
\[
\Psi_nA = (A_n \circ \Delta)^{h\Sigma_n}.
\]
We define a natural transformation $\eta \colon \id \Rightarrow \Psi \partial_*$ by letting its $n$th component be given by the composite
\[
F \Rightarrow P_nF \xRightarrow{\eta_n} \Psi_n \partial_*F,
\]
where the first arrow is the canonical map from $F$ to its $n$-excisive approximation and the second arrow is given by the top horizontal map from the Tate square.

\begin{proposition} \label{prop: right-adjoint-derivatives}
    The natural transformation $\eta \colon \id \Rightarrow \Psi \partial_*$ exhibits $\Psi$ as the right adjoint of $\partial_*$.
\end{proposition}
\begin{proof}
    Let $A_n \in \SymFunlin_n(\LSp)$.
    By \cite[Proposition 6.1.6.27]{HA}, the norm
    \[
    (A_n \circ \Delta)_{h\Sigma_n} \Rightarrow (A_n \circ \Delta)^{h\Sigma_n}
    \]
    exhibits the source as the $n$-homogeneous approximation of the target.
    In particular, this map induces an equivalence $\partial_n(A_n \circ \Delta)^{h\Sigma_n} \simeq A_n$, which defines a natural transformation $\epsilon_n \colon \partial_n \Psi_n \Rightarrow \id_{\SymFunlin_n(\LSp)}$.
    Taking these together for all $n$, we obtain a natural transformation
    \[
    \epsilon \colon \partial_* \Psi = \partial_* \prod_{n \geq 1} \Psi_n \Rightarrow \prod_{n \geq 1} \partial_n \Psi_n \xRightarrow{\epsilon_n} \prod_{n \geq 1} \id_{\SymFunlin_n(\LSp)} \simeq \id_{\SymFunlin(\LSp)},
    \]
    where the first map combines the coassembly map for $\partial_*$  with the projection $\partial_* \to \partial_n$.

    Since the Tate square is a pullback square, the map $\eta_n$ induces an isomorphism on $n$th derivatives.
    Moreover, since the fiber of the right vertical map is the norm map, this isomorphism is an inverse to the isomorphism used in the definition of $\epsilon_n$.
    This implies that we can choose homotopies of the composites
    \[
    \partial_* \xRightarrow{\eta} \partial_*\Psi\partial_* \xRightarrow{\epsilon} \partial_* \quad \text{and} \quad \Psi \xRightarrow{\eta} \Psi \partial_* \Psi \xRightarrow{\epsilon} \Psi
    \]
    to the identity, so that $\eta$ and $\epsilon$ satisfy the triangle identities.
\end{proof}

\begin{corollary} \label{cor: unit-right-adjoint-der-norm}
    Let $n \geq 1$ and let $H \colon \LSp \to \LSp$ be an $n$-homogeneous functor.
   Then the unit $\eta \colon H \Rightarrow \Psi \partial_*H$ is homotopic to the norm map.
\end{corollary}
\begin{proof}
    Consider the commutative diagram
    \[
    \begin{tikzcd}
        H \ar[r, "\sim"] \ar[d, equal] & (\partial_nH \circ \Delta)_{h\Sigma_n} \ar[d, "\mathrm{N}"] \\
        H \ar[r, "\eta_n"] \ar[d] & (\partial_nH \circ \Delta)^{h\Sigma_n} \ar[d] \\
        \ast \ar[r] & (\partial_nH \circ \Delta)^{t\Sigma_n},
    \end{tikzcd}
    \]
    where the bottom square is the Tate square for $H$ and the vertical columns are fiber sequences.
    The result follows since $\eta_n$ is the unit and $N$ is the norm map.
\end{proof}

By the following lemma, we also get a description of the right adjoint of the derivatives functor on $\End^{\ast, \omega}(\LSp)$.

\begin{lemma}
    The inclusion $\End^{\ast, \omega}(\LSp) \hookrightarrow \End^*(\LSp)$ admits a right adjoint.
\end{lemma}
\begin{proof}
    Choose a regular cardinal $\kappa$ such that $\LSp$ is $\kappa$-compactly generated and write $\End^{\ast, \kappa}(\LSp)$ for the full subcategory of $\End^{\ast}(\LSp)$ spanned by the functors that preserve $\kappa$-filtered colimits.
    Consider the inclusions
    \[
    \End^{\ast, \omega}(\LSp) \hookrightarrow \End^{\ast, \kappa}(\LSp) \hookrightarrow \End^{\ast}(\LSp).
    \]
    The first functor admits a right adjoint since it is a colimit preserving functor between presentable categories.
    The second functor admits a right adjoint given by first restricting to the full subcategory of $\kappa$-compact objects and then left Kan extending to $\LSp$.
    Therefore, the composite of these inclusions admits a right adjoint as well.
\end{proof}

\begin{corollary} \label{cor: right-adjoint-derivatives-finitary}
    The derivatives functor $\partial_* \colon \End^{\ast, \omega}(\LSp) \to \SSeq_{\geq 1}(\LSp)$ admits a right adjoint given by the composite
    \[
    \SSeq_{\geq 1}(\LSp) \xrightarrow{\Psi} \End^{\ast}(\LSp) \xrightarrow{R} \End^{\ast, \omega}(\LSp),
    \]
    where $R$ denotes the right adjoint from the previous lemma. \qed
\end{corollary}

\printbibliography

\end{document}